\documentclass{amsart}
\everymath{\displaystyle}  

\usepackage{amsmath}
\usepackage{amssymb}
\usepackage{amscd}
\usepackage{pb-diagram}
\usepackage[all]{xy}

\def\Ati{Atiyah \cite{Ati}}

\def\Bre{Bredon \cite{Bre}}

\def\Broo{Brooks \cite{Broo}}

\def\Del{Deligne \cite{Del}}

\def\Dix{Dixmier \cite{Dix}}

\def\Eil{Eilenberg \cite{Eil}}

\def\Gol{Golan \cite{Gol}}

\def\GroKahHyp{Gromov \cite{GroKahHyp}}
\def\Grot{Grothendieck \cite{Grot}}

\def\Mal{Malgrange \cite{Mal}}

\def\Mun{Munkres \cite{Mun}}

\def\Ped{Pedersen \cite{Ped}}

\def\Rei{ Reich \cite{Rei}}

\def\RouSan{Rourke Sanderson\cite{RouSan}}

\def\SalWoe{Saloff-Coste Woess \cite{SalWoe}}

\def\Ste{Steenrod \cite{Ste}}

\def\Takesaki{Takesaki \cite{Takesaki}}

\def\Vas{Vas \cite{Vas}}

\def\Whi{Whitehead \cite{Whi}}

\newtheorem{theorem-introduction}{Theorem}
\newtheorem{corollary-introduction}{Corollary}
\newtheorem{definition}[subsubsection]{Definition}
\newtheorem{lemma}[subsubsection]{Lemma}
\newtheorem{corollary}[subsubsection]{Corollary}
\newtheorem{theorem}[subsubsection]{Theorem}
\newtheorem{proposition}[subsubsection]{Proposition}
\theoremstyle{definition}
\newtheorem{remark}[subsubsection]{Remark}
\newtheorem{notations}[subsubsection]{Notations}
\newtheorem{example}[subsubsection]{Example}
\newtheorem{remarks}[subsubsection]{Remarks}

\def\im{\ensuremath{\mathrm Im}}

\def\supp{\ensuremath{\mathrm supp}}
\def\abut{\ensuremath{\Rightarrow}}

\def\ker{\ensuremath{\mathrm Ker}}
\def\Ker{\ensuremath{\mathrm Ker}}
\def\Cok{\ensuremath{\mathrm Coker}}

\def\conv{\ensuremath{\star}}
\def\tr{{\mathrm tr}_{N(G)}}
\def\Tr{{\mathrm Tr}}

\def\NN{{\mathbb  N}}
\def\ZZ{{\mathbb  Z}}
\def\CC{{\mathbb  C}}
\def\RR{{\mathbb  R}}
\def\QQ{{\mathbb  Q}}

\def\HH{{\mathbb  H}}

\def\Cg{\ensuremath{{\mathcal{C}_{God}}}}

\def\dimG{{\mathrm dim}_{N(G)} }
\def\VN{ \text{M}}
\def\Dom{{\mathrm Dom}}
\def\Im{{\mathrm Im}}
\def\Ran{{\mathrm Ran}}
\def\Hom{{\mathrm Hom}}

\def\inc{{\mathrm inc}}

\def\p{{\ensuremath{ p_{*(2)} } }}
\def\jp{{\ensuremath{(j^* p)_{*(2)} } }}
\newcommand{\ap}[1]{\ensuremath{ (a^{*}_{#1} p)_{*(2)} }}
\def\nc{{\mathrm n.c.}}

\def\para{\S}
\def\card{\sharp}
\def\Card{\sharp}

\def\canfilt{\tau_{\leq}}

\def\chainr{{\mathrm r } } 
\def\Op{h}

\def\dto{\dashrightarrow}

\def\Mod{{\text{Mod}}}

\def\T{\ensuremath{{\mathcal{T}}}}
\def\U{\ensuremath{{\mathcal{U}}}}
\def\C{\ensuremath{{\mathcal{C}}}}

\def\I{\ensuremath{{\mathcal{I}}}}
\def\P{\ensuremath{{\mathcal{P}}}}
\def\E{\ensuremath{{\mathcal{E}}}}

\def\ot{\ensuremath{\leftarrow}}

\def\cS{\ensuremath{{\mathcal{S}}}}

\def\R{\ensuremath{{\mathcal{R}}}}
\def\N{\ensuremath{{\mathcal{N}}}}

\def\M{\ensuremath{{\mathcal{M}}}}
\def\O{\ensuremath{{\mathcal{O}}}}
\def\L{\ensuremath{{\mathcal{L}}}}
\def\K{\ensuremath{{\mathcal{K}}}}
\def\F{\ensuremath{{\mathcal{F}}}}
\def\H{\ensuremath{{\mathcal{H}}}}
\def\B{\ensuremath{{\mathcal{B}}}}
\def\A{\ensuremath{{\mathcal{A}}}}

\newcommand{\Int}[1]{\ensuremath{\overset{\circ}{#1}}}
\newcommand{\Adh}[1]{\ensuremath{\overline{#1}}}

\newcommand{\dlog}[2]{\ensuremath{\Omega^{#1}_{X}(log #2)}}
\newcommand{\Slog}[2]{\ensuremath{\mathcal{S}^{\infty, #1}(log #2)}}
\newcommand{\un}[1]{\ensuremath{\underline{#1}}}
\newcommand{\Rm}[2]{\ensuremath{\overset{#1}{\check{#2}}}}

\def\L{\ensuremath{{\mathcal{L}}}}
\def\I{\ensuremath{{\mathcal{I}}}}

\def\dbar{\ensuremath{{\overline{\partial}}}}

\def\C{\ensuremath{{\mathcal{C}}}}

\title[Some mixed Hodge structures on $l^{2}-$cohomology groups]{Some mixed Hodge structures on $l^{2}-$cohomology groups of coverings of K\"{a}hler manifolds.}
\author{P. Dingoyan}
\address{P. Dingoyan\\
 Universit\'e Paris 6, \\case 247, 4 place Jussieu, 75252 Paris Cedex 05, France.}
 \email{dingoyan@math.jussieu.fr}
\begin{document}
\begin{abstract}
We give methods to compute $l^{2}-$cohomology groups of a covering manifold obtained by removing the pullback of a (normal crossing) divisor to a covering of a compact K\"ahler manifold.

We prove that in suitable quotient categories, these groups admit natural mixed Hodge structure whose graded pieces are given by the expected Gysin maps.  
\end{abstract}

\maketitle
\bibliographystyle{plain}
\section{introduction.}

\subsection{}

The Hodge decomposition of the cohomology ring of a compact K\"ahler manifold $X$ defines a Hodge structure. 
This gives strong relations between the topology of the manifold and the holomorphic structure. 
This decomposition still holds for an infinite covering $p:\tilde X\to X$ once one restricts to the space of square integrable harmonic forms. Let us recall two achievements in this setting.

The $l^{2}-$cohomology groups of the covering $p:\tilde X\to X$  are  the De Rham cohomology groups $H^{k}_{d(2)}(\tilde X)=\Ker (d)/\Im(d)$ of the square integrable forms on $\tilde X$.  
The reduced $l^{2}-$cohomology groups $\Ker (d)/\Adh{\Im(d)}$ are isomorphic to the harmonic spaces $\H^{i}_{d(2)}(\tilde X)$ (\cite{Ati}, \cite{CheGro}, \cite{Shu} and \cite{Eck}, \cite{CheGrodeux}, \cite{Luc},  \cite{Dod}).
Assume that $p:\tilde X\to X$ is a Galois covering with Galois group $G$. 

Then a fundamental result of Atiyah \cite{Ati} is that  the Euler characteristic of $X$ is equal to the $l^{2}-$Euler characteristic of $p:\tilde X\to X$. Note that the $l^{2}-$harmonic spaces are modules over the Von Neumann algebra $N(G)$ generated by the left action of $G$ on $l^{2}(G)$. This allows to define a Von Neumann dimension. 

When $X$ is a K\"{a}hler manifold, the $d-$harmonic spaces on $\tilde X$ admit a Hodge decomposition according to bi-type:
$ \H^{r}_{d(2)}(\tilde X) = \oplus_{p+q=r} \H^{(p,q)}_{\dbar(2)}(\tilde X)$. If moreover $\tilde X$ is the universal covering space of $X$,
 Gromov \cite{GroCras}, \cite{GroKahHyp} proves that the non vanishing of $\H^{1}_{d(2)}(\tilde X)$ implies that there exists a proper equivariant holomorphic map from $\tilde X$ to the unit disc.

\subsection{}

Deligne \cite{Del} \cite{Deltrois} discovered that the cohomology ring of any algebraic variety, open or singular, carries a mixed Hodge structure (a real filtration whose quotients have a Hodge structure subject to natural compatibilities). 
In this article, one will put such mixed Hodge structure on the $l^{2}-$cohomology groups of the covering $p^{-1}(X\setminus D)\to X\setminus D$, which is defined by the restriction over a Zariski open subset of a covering of $X$. 

Let  $D=D_{1}\cup \ldots \cup D_{r}$ be a normal crossing divisor in $X$. One first recall that the mixed Hodge structure on the cohomology groups of the quasi-projective manifold $X\setminus D$ is built from the Hodge structures of the  compact projective manifolds $X$, $D_{i}$, $D_{i}\cap D_{j}$, $\ldots$. An example is provided by one irreductible smooth divisor $D$ in $X$: 
the  Gysin-Leray exact sequence  $ \ldots\to H^{p-2}(D,\CC)\overset{i_{!}}{\to} H^{p}(X,\CC)\to H^{p}(X\setminus D,\CC)\overset{res}{\to} H^{p-1}(D,\CC)\ldots$ enables one to filter $H^{p}(X\setminus D,\CC)$ by subspaces whose quotients carry Hodge structures. In general, a spectral sequence considers further relations between the divisors.

Now, the $l^{2}-$cohomology groups of $p^{-1}(X\setminus D)\to X\setminus D$ are defined as the cohomology  over $X\setminus D$ of the locally constant sheaf $\p\CC_{|X\setminus D}$, the sheaf of locally constant square integrable functions in the fiber of $p$. This sheaf is isomorphic to $l^{2}(G)\otimes_{\ZZ[G]}p_{!}\ZZ_{\tilde X\setminus p^{-1}(D)}$ and the cohomology groups are isomorphic to the groups of equivariant cohomology of $p^{-1}(X\setminus D)$ with values in $l^{2}(G)$ (see \ref{Link to singular cohomology}).

The building blocks of the mixed Hodge structure on $H^{.}( X\setminus D,\p\CC)$ will be the Hodge structures on the harmonic spaces $\H^{.}_{d(2)}(\tilde X), \H^{.}_{d(2)}(p^{-1}(D_{i})), \H^{.}_{d(2)}(p^{-1}(D_{i}\cap D_{j})),\ldots $ and Gysin's morphisms between them. Note that $G$ acts co-compactly on the manifolds $\tilde X, p^{-1}(D_{i}),\ldots$.

We first transpose the sheaf theoretic development of mixed Hodge theory as in Deligne \cite{Del} to the $l^{2}-$setting. To this aim, we use the definition of Campana-Demailly \cite{CamDem} for the $l^{2}-$cohomology groups of $G-$equivariant coherent analytic sheaves. This gives enough functoriality to obtain the Gysin's morphisms between the non reduced $l^{2}-$cohomology groups.

In order to use the Hodge structure on the harmonic spaces, we work in a quotient category or localized category (\cite{Gab}, \cite{Grot}, \cite{Ver}).
 In the quotient category, the morphisms from the $\dbar-$harmonic spaces $\H^{(p,q)}_{\dbar(2)}(p^{-1}(\cap_{0\leq i\leq l}D_{t_{i}}))$ to the unreduced Dolbeault cohomology groups $H^{p,q}_{\dbar(2)}(p^{-1}(\cap_{0\leq i\leq l}D_{t_{i}}))$ are isomorphisms. Hence the non reduced parts of the Dolbeault cohomology groups become isomorphic to zero.

Then, the weight spectral sequence, which abuts to the $l^{2}-$cohomology groups of $p^{-1}(X\setminus D)\to X\setminus D$, 
becomes a spectral sequence of  Hodge structures. As in  \Del , it degenerates at $E_{2}$ and is therefore computable (see Theorem 3). 
 
The initial motivation of this work was the study of  non compact divisors in the universal covering of $X$, whose irreducible components are compact. We refer to the article of Nori \cite{Nor} for results on this question. Here we focus on the technical part of mixed Hodge structure modulo some torsion theory. In a subsequent paper, functoriality, geometrical and analytical applications will be given.

\subsection{} Before giving detailed statements, we comment on the use of  torsion theory. The precise definition of a torsion theory will be given in \ref{torsion theory}. It is a quite standard tool in $l^{2}-$cohomology (see Farber \cite{Far}, Eyssidieux \cite{Eys}, L\"uck \cite{Luc}, Sauer-Thom \cite{SauTho} and it appears implicitly in Cheeger-Gromov \cite{CheGrodeux}, Shubin \cite{Shu}). We combine the torsion theory with a theory of $l^{2}-$sheaves (Eyssidieux \cite{Eys}, Campana-Demailly \cite{CamDem}): we can then link the simplicial, the topological and the analytical $l^{2}-$cohomology in the framework of the mixed Hodge structures. 

In the quotient category, the torsion modules are by definition isomorphic to zero and isomorphisms are defined up to torsion modules. 
Hence it may be useful to interpret in the original category an isomorphism in the localised category.
In the case of a Galois covering, the Von Neumann algebra $N(G)$ of the groups $G$ acts faithfully on the $l^{2}-$cohomology groups. We can work modulo the modules with vanishing Von Neumann dimension and results may be interpreted in terms of the orbit $N(G)x$ of elements. This leads to a $\partial\dbar-$lemma (\ref{A ddbar-lemma}) up to a weak isomorphism. 

\subsection{}
  The following theorem gives a functorial description of the $l^{2}-$Hodge decomposition. It will be used later to relate the  $l^{2}-$cohomology groups over $X$, $D$, $D_{i}\cap D_{j}$,$\ldots$  

Let $(\p\Omega^{.},d)\to X$ be the complex of $l^{2}-$direct image of holomorphic forms:  a germ at $x\in X$ is given by a square integrable holomorphic form in a neighborhood $p^{-1}(V)$ of $p^{-1}(x)$ (\cite{CamDem} and \ref{Direct $l^{2}-$image}).

\begin{theorem-introduction} Let $p: \tilde X\to (X,\omega)$ be a Galois covering of a compact hermitian manifold with covering group $G$. Let $N(G)$ be the Von Neumann algebra of $G$. 
\item[1)] There exists a Hodge to De Rham spectral sequence  of $N(G)-$modules $$H^{p,q}_{\dbar(2)}(\tilde X)\simeq H^{q}(X,\p\Omega^{p}) \abut \HH^{p+q}(X, (\p\Omega^{.},d))\simeq H^{p+q}_{d(2)}(\tilde X)\,.$$
\item[2)] Assume that $\omega$ is a K\"{a}hler metric.
Let  $\tau$ be a real torsion theory such that 
$\Adh{\Im  \dbar}/\Im \dbar$ is a torsion module. Then the Hodge to De Rham spectral sequence $H^{p,q}_{\dbar(2)}(\tilde X) \abut  H^{p+q}_{d(2)}(\tilde X)$ degenerates at $E_{1}$ in the quotient category $Mod(N(G))_{/\tau}$ and the isomorphism $\oplus_{p+q=r} \H^{(p,q)}_{\dbar(2)}(\tilde X)\simeq H^{r}_{d(2)}(\tilde X)$  in $Mod(N(G))_{/\tau}$ defines a $\tau-$Hodge structure on $H^{r}_{d(2)}(\tilde X)$.
 \end{theorem-introduction}

Standard torsion theories are $\tau_{dim}$ and $\tau_{\U(G)}$:
%
The torsion theory $\tau_{dim}$ is such that modules of  zero $G-$dimension are torsions (we use the generalised Von-Neumann dimension of \cite{Luc} Chap. 6). 
 
 Let $\U(G)$ be the ring of operators affiliated to $N(G)$: It is the ring of  unbounded operators on $l^{2}(G)$ that commutes with the right action of $G$ on $l^{2}(G)$. It is isomorphic to the quotient ring of $N(G)$ by the multiplicative set of weak isomorphism (see \cite{Luc} Chap. 8). 
 Then $\tau_{\U(G)}$ is the torsion theory such that a module $M$ is a torsion module if $\U(G)\otimes_{N(G)}M$ is isomorphic to zero. When the covering $p:\tilde X\to X$ is Galois and $X$ is compact, then $\Adh{\im\dbar}/\im\dbar$ is a $\tau_{\U(G)}-$torsion module.

\subsection{}
The following theorem enables one to define a mixed Hodge structure on the $l^{2}-$cohomology groups of $p^{-1}(X\setminus D)
\to X\setminus D$. It gives two filtrations for computing these groups. The weight filtration $W$ measures the singularities along $D$ of some representative of a cohomology class in $H^{.}(X\setminus D, \p\CC)$. The Hodge filtration is related to the decomposition of forms into bi-type. 

Let $\p\RR=l^{2}(G,\RR)\otimes_{\RR[G]}p_{!}(\RR_{\tilde X})$ and  $\p\CC=l^{2}(G,\CC)\otimes_{\CC[G]}p_{!}(\CC_{\tilde X})$ be the sheaves on $X$ of locally constant functions which are square integrable in the fibers of $p$.
 Let  $(\p\dlog{.}{D},d)$ be the complex of sheaves on $X$ of $l^{2}-$direct image  (\cite{CamDem} and \ref{Direct $l^{2}-$image}) of the logarithmic forms with pole on $D$. This complex is bi-filtered by the weight filtration $W$ (see \ref{Local setting}) and the Hodge filtration $F$.

\begin{theorem-introduction}

Let $D=D_{1}\cup\ldots\cup D_{r}$ be a normal crossing divisor in $X$.  Set $\underline{D}_{l}=\sqcup_{\card I=l}\cap_{i\in I} D_{i}$, $(1\leq l\leq n)$ and  $\underline{D}_{0}=X$. 

\item[1)] The group $H^{.}(X\setminus D,\p\CC)$ is isomorphic as a $N(G)-$module to $\HH^{.}(X,(\p\dlog{.}{D},d))$.
\item[2)]Let $\tau$ be a real torsion theory on $\Mod(N(G))$ and let $\Mod(N(G))_{/\tau}$ be the quotient category. Assume the $l^{2}-$Hodge to De Rham spectral sequences of each $p^{-1}(\underline{D}_{l})$, $l\in\{0,\ldots, r\}$, degenerates in $\Mod(N(G))_{/\tau}$  as in theorem 1. 
Then:
\item[i)]  The weight spectral sequence for $\HH^{.}(X,(\p\dlog{.}{D},W,d))$ whose  $E_{1}^{-p,p+q}-$term  is isomorphic to $H^{q-p}_{d (2)}(p^{-1}(\underline{D}_{p}))$ 
 degenerates  at $E_{2}$ in $\Mod(N(G))_{/\tau}$.
\item[ii)] The Hodge spectral sequence for $\HH^{.}(X,(\p\dlog{.}{D},F,d))$ which  has term $E_{1}^{p, q}$ isomorphic to 
$H^{q}(X,\p\dlog{p}{D})$ degenerates at $E_{1}$ in $\Mod(N(G))_{/\tau}$.
\item[iii)] Define the weight filtration $W$ on $H^{k}(X\setminus Y,\p\RR)$ to be the shifted filtration $W[k]_{.}:=W_{.-k}$  of the filtration induced by the weight spectral sequence. 

Define the Hodge filtration $F$ on $H^{k}(X\setminus Y,\p\CC)$ to be 
the filtration  induced by the Hodge spectral sequence.

Then $(H^{k} (X\setminus D,\p\RR), W, F)$ is a mixed Hodge structure in $\Mod(N(G))_{/\tau}$.
\end{theorem-introduction}
We note further that for a torsion theory 
$\tau$ defined in terms of the group $G$ only, such as $\tau_{dim}$ or $\tau_{\U(G)}$, the above mixed Hodge structure in $\Mod(N(G))_{/\tau}$ is functorial in the category of $G-$covers. Moreover two compactifications of $X\setminus D$ which are dominated by a third one will define isomorphic mixed Hodge structures.
\subsection{} The reduction of the above  spectral sequences to $\Mod(N(G))_{/\tau_{dim}}$ or  $\Mod(N(G))_{/\tau_{\U(G)}}$ always  defines mixed Hodge structures. The isomorphism $E_{2}(W)\simeq E_{\infty}(W)$  gives the expected realisation in term of harmonic spaces and Gysin maps (notations of \ref{Reduction with respect to the dimension torsion theory}): 
 \begin{theorem-introduction} Let $p:\tilde X\to X$ be a Galois covering of a compact K\"{a}hler manifold. Let $G$ be its group of deck transformations and let $N(G)$ be its von Neumann algebra. Then in the quotient category $\Mod(N(G))_{/\tau_{\dim}}$ 
 or $\Mod(N(G))_{/\tau_{\U(G)}}$,
  the space $Gr^{W}_{l+k}H^{k}(X\setminus D,\p\CC)$ is isomorphic to the middle homology of the Gysin sequence $$ \H^{k-l-2}_{d(2)}(p^{-1}(\underline{D}_{l+1}))\overset{}{\to}\H^{k-l}_{d(2)}(p^{-1}(\underline{D}_{l})) \overset{}{\to}\H^{k-l+2}_{d(2)}(p^{-1}(\underline{D}_{l-1}))\,. $$ 
 Moreover  $$Gr_{F}^{p}H^{p+q}(X\setminus D,\p\CC)\simeq \H^{q}_{\dbar(2)}(\tilde X, \Omega^{p}_{\tilde X}log(p^{-1}(D)))\,. $$

  \end{theorem-introduction}
\bigskip
 Part of the graded module associated to the weight filtration on $H^{.}(X\setminus D,\p\CC)$ has a combinatorial description:

A dual CW complex $\tilde K$ is attached to $(\tilde X, p^{-1}(D))\to (X,D)$ (see \ref{Simplicial structure}): A connected component of $p^{-1}(\underline{D}_{k+1})$ defines a $k-$cell. It is attached along $(k-1)-$cells corresponding to connected components of $p^{-1}(\underline{D}_{k})$ which contain it.

\begin{corollary-introduction}
 In $\Mod(N(G))_{/\tau_{dim}}$ or  $\Mod(N(G))_{/\tau_{\U(G)}}$, the space $Gr^{W}_{2n}H^{2n-(k+1)}(X\setminus D,\p\CC)$ ($n=\dim_{\CC}X$) is isomorphic to the $k-$th (reduced) relative $l^{2}-$homology group $\bar H_{k,(2)}(|\tilde K|,|\tilde K(\infty)|)$ of the dual CW-complex associated to $(\tilde X,p^{-1}(D))\to(X, D)$.  
\end{corollary-introduction}
Here $\tilde K(\infty)$ is the set of cells whose isotropy subgroups under the action of $G$ are infinite.  
%
%
\subsection{ } My hearty thanks go to S. Vassout and G. Skandalis for illuminating discussions on Von Neumann algebras. Thanks to B. Klingler and P. Eyssidieux for stimulating discussions on the subject. I thank the referee for its valuable work and helpful comments. I thank G. Courtois, S. Diverio, E. Falbel, V. Minerbe and M. Wolff for their remarks and corrections on the initial version of this article.

\section{preliminaries.}
\subsection{Real structures.}\label{real structures}
\subsubsection{The Godement resolution}\label{Godement-resolution}(See
Godement \cite{God} or \Bre): 
Let $X$ be a topological space.
Let $R$ be a ring and  $\R$ be the sheaf of rings it defines. Let $\A$ be a sheaf of 
(left) $\R-$modules. Let $\C^{.}(X,\A):=(\Cg^{.}(X,\A),d)$ be the Godement resolution (\cite{God} p. 168)
 by  sheaves of $\R-$modules with $\R-$linear differential $d$.  Let $\Gamma(X,.)$ be the functor of global sections.
 
 Then 
 \begin{eqnarray*}
 \A\to \C^{.}(X,\A) \,,& & \A\to C^{.}(X,\A)=(\Gamma(X,\Cg^{.}(X,\A)),d)
 \end{eqnarray*}
    are 
covariant additive exact functors with values in the category of differential $\R-$sheaves, 
resp. cochain $R-$complexes. If $X$ is clear from the context, we write $\C^{.}(\A),\ldots$ 


If $(\F^{.},d)$ is a differential sheaf, let $\C^{.}(\F^{.})$ be the total complex associated to the double complex $\Cg^{q}(X,\F^{p})$. Let $j:Y\to X$ be a continuous map and $(\F^.,d)$ be a differential sheaf of $\R-$modules on $Y$. One sets  $Rj_
{*}\F^{.}:=j_{*}\C^{.}(\F^{.})$ 

\subsubsection{Real structures}
Assume that $R$ is a $\RR-$algebra, then any sheaf of $\R-$modules is 
also a sheaf of $\RR-$vector spaces so that $\A\otimes_{\RR}\CC$ is a sheaf of $\R
\otimes_{\RR}\CC-$left modules. Let $i:\A\to \A\otimes_{\RR}\CC$.

Then 
\begin{eqnarray}\label{real structure}
  j_{*}(\A)\otimes_{\RR}\CC&\to& j_{*}(\A\otimes_{\RR}\CC)\\
%
\C^{.}(X,\A)\otimes_{\RR}\CC\overset{  \C^{.}(i)\otimes 1_{\CC}}{\longrightarrow}\C^{.}(X,\A\otimes_{\RR}\CC)& &
C^{.}(X,\A)\otimes_{\RR}\CC\overset{ C^{.}(i)\otimes 1_{\CC}}{\longrightarrow}C^{.}(X,\A\otimes_{\RR}\CC)
  \end{eqnarray}    
    are $\R\otimes_{\RR}\CC-$isomorphisms (resp. $R\otimes_{\RR}\CC-$isomorphisms).
\subsubsection{ }
A real structure on a $\R\otimes_{\RR}\CC-$sheaf $\B$ is a $\R-$subsheaf  
$\A\overset{i}{\to}\B$ such that $\A\otimes_{\RR}\CC\overset{i\otimes 1_{\CC}}{\to}\B$ is an isomorphism. It induces 
a real structure on the Godement resolution. 

Then $ H^{.}(X,\,\A)\otimes_{\RR}\CC\to H^{.}(X, \,\A\otimes_{\RR}\CC)$ is a $R\otimes_{\RR}\CC-$isomorphism.  
\begin{definition}
Let $G$ be a discrete  group. Let $N_{l}(G)$ (resp. $N_{r}(G)$) be the left (resp. right) von Neumann algebra of $G$ generated by operators on $l^{2}(G):=l^{2}(G,\CC)$ of left (resp. right) convolution by elements in $\CC[G]$. Then one sets $N(G):=N_{l}(G)$. 
\end{definition}
The $\RR[G]-$isomorphism
 $l^{2}(G,\CC)\ni a\to Re(a)+iIm(a)\in l^{2}(G,\RR)\otimes\CC$   induces  decompositions 
 \begin{eqnarray}
  End_{\CC}( l^{2}(G,\RR)\otimes\CC) &\simeq & End_{\RR}( l^{2}(G,
\RR))\otimes\CC \\
  End_{\CC[G]}( l^{2}(G,\RR)\otimes\CC) &\simeq & End_{\RR[G]}( l^{2}(G,
\RR))\otimes\CC \\
 N_{r}(G,\CC)&\simeq &N_{r}(G,\RR)\otimes \CC
\end{eqnarray}

The left action of $G$ on  $l^{2}(G,\RR)$ defines a left action of $G$ on $End_{\RR}( l^{2}(G,\RR))$ 
and $End_{\CC} l^{2}(G,\CC)$ so that $g(m(a))=(gm)(ga)$.  The set of continuous invariant morphisms for this action, $N_{r}(G,\RR)$ and $N_{r}(G,\CC)$, are  algebras.    
 An element $ \rho(f)\in N_{r}(G,\CC)$ is represented by right convolution  with a function $f\in l^{2}(G,\CC)$ which is moderate (\cite{Dix} 13.8.3): $\exists C\geq 0:\, \forall g\in \CC[G],\, ||g\conv f||_{2}\leq C||g||_{2}$. Then $\rho(f)\in N_{r}(G,\RR)$ is equivalent to $f$ real valued.
 
The symetric statements hold if we work with the right Von Neumann algebra generated by right regular representation.  Hence  $N(G,\CC)\simeq 
N(G,\RR)\otimes \CC$. 

\subsection{A lemma on von Neumann algebras.}
In this section we recall basic definitions about a (finite) Von Neumann algebra $M$, its representations and the Murray-Von Neumann dimension.

Our goal is the lemma \ref{T-lemma} which enables to lift, up to a weak isomorphism,  to the source a vector in the closure of the range of a $M-$Fredholm operator. This lemma, of analytical independent interest, will give a non trivial example of torsion theory.
   
For this section we refer to \Dix, \Ped, \Takesaki.
A survey on von Neumann algebras is given in \cite{Luc} chap. 9. 

\begin{definition}
Let $H_{0}$ be some (separable) Hilbert space. A $\C^{*}-$subalgebra $\VN$ of $\B(H_{0})$ is a Von Neumann algebra if $\VN$ is weakly closed
$\iff \VN$ is strongly closed 
$\iff \VN$ is equal to its bi-commutant in $\B(H)$. 
\end{definition}
One let $M'$ be the commutant of $M$ in $\B(H_{0})$. 
Let $H$ be some Hilbert space.
A representation $\pi:\VN\to \B(H)$ is normal if it is continuous on bounded increasing net of $\VN$. Normality of $\pi$ implies that $\Ker (\pi)$ and $\pi(\VN)$ are Von Neumann algebras. 
Then $(H,\pi)$ (or simply $H$) is called a $\VN-$Hilbert module.
We recall that weakly isomorphic $\VN-$Hilbert modules are $\VN-$isometric:
\begin{lemma}\label{weak isomorphism} Let $f:(H_{1},\pi_{1})\to (H_{2},\pi_{2})$ be a bounded $\VN-$linear weak isomorphism (injective with dense range). Let $f=up$ be its polar decomposition. Then $u: H_{1}\to H_{2}$ is a $\VN-$isometry.  
\end{lemma}
\begin{definition}
Let $\VN$ be a Von Neumann algebra on a separable Hilbert space. Let $\VN_{+}$ be the cone of positive operators in $\VN$.
\item[1)]A  trace is a function $t: \VN_{+}\to [0,+\infty]$  such that if $\lambda>0$ and $x,y\in \VN_{+}$ then $t(\lambda x)=\lambda t(x)$, $t(x+y)=t(x)+t(y)$ and for all unitary $u\in A$, $t(u^*xu)=t(x)$. 
\item[2)] A state  is a positive linear functional of norm one: $\varphi\in \VN^{'}$ such that $\varphi(\VN_{+})\subset \RR^+$ and $\varphi(1)=1$. 
\item[3)] A trace or a state is normal if it is continuous on limit of increasing nets.
\item[4)] A faithful trace is called finite if $t(1)<+\infty$ (and $\VN$ is then called finite von Neumann algebra). It is called semi finite if $\VN_{+}^{t}=\{ y\in \VN_{+}:\, t(y)<+\infty\}$ is weakly dense in $\VN_{+}$. Then for all $x\in\VN_{+}$, $t(x)=\sup_{y\leq x,\ y\in \VN_{+}^{t}}t(y)$.
\end{definition}
\begin{example}
\item[1)]
Let $\delta_{e}\in l^{2}(G)$ be the dirac function at the unit element $e$ of $G$. The trace of $n\in N_{l}(G)$ or $N_{r}(G)$ is $\tr n:=<n(\delta_{e}),\delta_{e}>$. 
\item[2)] Let $(H, \pi)$ be a $M-$Hilbert module. Let $\zeta\in H$ be a vector of unit norm. Then $x\mapsto <\pi(x)\zeta,\zeta>$ is a normal positive linear functional of norm $1$, called the vector state associated to $(\pi,H,\zeta)$.  
\end{example}

Traces and states satisfy Cauchy-Schwartz inequality:
\begin{eqnarray}
 |\varphi(y^*x)|^2\leq \varphi(x^*x)\varphi(y^*y) &\text{ on } & \VN_{2}^{\varphi}=\{x\in\VN:\, \varphi(x^*x)<+\infty\}\,.
\end{eqnarray}
Therefore, the left kernel $L_{\varphi}=\{x\in\VN:\,\varphi(x^*x)=0\}$ of a trace or a state is a linear space. One says that $\varphi$ is faithful if $L_{\varphi}=0$. 
Define a scalar product on $\VN_{2}^{\varphi}/L_{\varphi}$ (which is equal to $\VN/\Ker(\varphi)$ if $\varphi$ is a state) by $(\zeta_{x},\zeta_{y}):=\varphi(y^*x)$ with $x\mapsto \zeta_{x}$ be the quotient map.

 The GNS (Gelfand-Naimark-Segal) construction is  obtained by completion of this pre-Hilbert space (see \cite{Ped} 3.3).
 Recall that a representation $(\pi,H,\zeta)$ of $\VN$ is said to be cyclic if $\pi(\VN)\zeta$ is dense in $H$. 
 \begin{lemma}\label{GNS}
 \item[1)] Let $\varphi$ be a positive linear functional on $\VN$. There exists a  cyclic representation $(\pi_{\varphi}, H_{\varphi},\zeta_{\varphi})$ such that 
 $<\pi_{\varphi}(x)\zeta_{\varphi},\zeta_{\varphi}>=\varphi(x)$. 
  \item[2)] Let $\varphi'$ be a positive linear functional such that $\varphi'\leq \varphi $, then there exists a unique $a\in \pi_{\varphi}(\VN)'$, $0\leq a\leq 1$, such that $\varphi'(x)=(\pi_{\varphi(x)}a\zeta_{\varphi},\zeta_{\varphi})$. 
\item[3)] Hence $(\pi_{\varphi'},H_{\varphi'}, \zeta_{\varphi'})$ is a subrepresentation of $(\pi_{\varphi}, H_{\varphi}, \zeta_{\varphi})$.  
\item[4)] Let $\varphi$ be the vector state associated to $(\pi, H, \zeta)$, then $\VN/\Ker(\varphi)\ni x\mapsto x\zeta\in H$ defines a $\VN-$linear isometry between $(\pi_{\varphi},H_{\varphi},\zeta_{\varphi})$ and $(\pi, \Adh{\pi(\VN)\zeta}^{H}, \zeta)$. 
\item[5)] Assume that $\varphi$ is a faithful state, then for any state $\varphi'$ on $\VN$, $(\pi_{\varphi'},\, H_{\varphi'},\, \zeta_{\varphi'})$ is a sub- representation of $(\pi_{\varphi}, H_{\varphi},\zeta_{\varphi})$. 
\end{lemma}
\subsubsection{The standard form }\label{standarde}
The unitary invariance of  a trace  $t$ implies that $t(xy)=t(yx)$ on $\VN_{2}^{t}$. This gives further properties to the associated GNS space:
Let $(\VN,t)$ be a von Neumann algebra with a normal faithful tracial state (a continuous positive linear form $\varphi$ such that $\varphi(1)=1$ and $\varphi$ is a normal trace). Then $(x,y)\mapsto t(y^*x)=(x,y)_{t}$ is a scalar product.
\begin{definition}
\item[1)] Let $(\pi_{t}, l^{2}(\VN,t), \zeta_{t})$  be the Hilbertian representation obtained through the GNS construction. It is called the standard form associated to $(\VN,t)$. 
\item[2)] The isometric densely defined operator $\VN\ni x\to x^{*}\in\VN$ extends to   $J:L^2(\VN,t)\to L^{2}(\VN,t)$ which is conjugate linear, isometric and involutive.
\item[3)] If $x\in \VN$, the maps $\lambda(x):y\mapsto xy$ and $\rho(x): y\mapsto yx$, defined on $\VN$, extend uniquely as $\rho(x),\lambda(x) \in \B(   l^{2}(\VN,t))$.
\end{definition}
A vector $\zeta\in l^{2}(\VN,t)$ defines two closed, densely defined, unbounded operators:

\begin{eqnarray*}\lambda(\zeta): D(\lambda(\zeta))=\VN\zeta_{t} &\ni & x \zeta_{t}\mapsto \rho(x)\zeta \in l^{2}(\VN,t)\\
\rho(\zeta): D(\rho(\zeta))=\VN\zeta_{t} &\ni & x\zeta_{t}\mapsto \lambda(x)\zeta\in l^{2}(\VN,t) 
\end{eqnarray*}

\begin{lemma}\label{standard}
\item[1)]
$\lambda(\VN)$ and $\rho(\VN)$ are von Neumann subalgebras on $l^{2}(\VN,t)$ such that $J\lambda(\VN)J=\rho(\VN)$ and $\lambda(\VN)'=\rho(\VN)$. 
\item[2)] Let $\zeta\in l^{2}(\VN,t)$. Then $\rho(\zeta)$   is bounded iff $\lambda(\zeta)$   is bounded iff $\zeta\in \VN\zeta_{t}$.
\end{lemma}
\begin{proof} see \cite{Takesaki} Chap. 5 th.2.22 (p. 324) and Lemma 2.21
\end{proof}

 \begin{example}
Let $G$ be a discrete  group, let $N_{l}(G)$, $N_{r}(G)$ left (resp. right) von Neumann algebra generated by the left (resp. right) translations. Then the standard form of $(N_{l}(G),tr)$ is $(\lambda, l^{2}(G), e)$ and $N_{l}(G)'=N_{r}(G)$ and $N_{r}(G)'= N_{l}(G)$ (\cite{Dix} I.5.2).
\end{example}

\subsubsection{The Murray-Von Neumann dimension}

Let $\VN'$ be the commutant of $\VN$ acting on $l^{2}(\VN)$.
Let $\VN$ act on $l^{2}(\VN)\otimes l^{2}(\NN)$ through the faithful representation $x\to x\otimes 1$. It is called amplification (\cite{Takesaki} p. 184).
The structure theorem of normal morphisms between Von Neuman algebras (\cite{Dix}) implies:
\begin{lemma}Let $(H,\pi)$ be a separable $\VN-$module. Then there exists a $\VN-$isometry\\ $u: H\to l^{2}(\VN)\otimes l^{2}(\NN)$ such that $ux=(x\otimes 1)u$.
\end{lemma}
Let $p$ be the orthogonal projection on $u(H)$. Then $p$ belongs to the commutant of $\VN$ acting on $l^{2}(\VN)\otimes l^{2}(\NN)$ which is equal to $\VN'\bar\otimes \B(l^{2}(\NN))$: $p=(p_{ij})$ decomposes as a matrix  such that $p_{ij}\in\VN'$.

\begin{definition} The trace of $p$ is defined by $\Tr(p)=\sum_{i\in \NN}\tau(p_{ii})$. It does not depend on the embedding $u$. It is called the Von Neumann dimension of $H$, denoted $dim_{(\VN,\tau)}H$ or $dim_{\VN}H$ if $\tau$ is understood.
\end{definition}
\begin{definition}\label{fredholm}\cite{Luc, Shu}  A morphism $h:H_{1}\to H_{2}$ between $\VN-$Hilbert modules is called $\VN-$Fredholm if $\dim_{\VN}\Ker(h)<+\infty$ and there exists $L\subset \Ran(h)$ such that $\dim_{\VN}H_{2}\ominus L<+\infty$.
\end{definition}
Then (loc. cit.) $h$ is a $\VN-$Fredholm morphism iff $\dim_{\VN}\Ker (h^{*})<+\infty$ and if $h^{*}h=\int\lambda dE_{\lambda}$ is the spectral decomposition of $h^*h$ then there exists $\lambda>0$ such that $Tr_{\VN}E_{\lambda}<+\infty$.

\begin{definition}(Murray-Von Neumann \cite{MurVon}  Chap. XVI)
\item[1)]
Let $\VN$ be a finite Von Neumann algebra on $H$. 
 A closed densely defined operator $h: \Dom(h)\to H$ is said to be affiliated to $\VN$ if it commutes with $\VN'$: for all unitary $u\in \VN'$, $u\Dom(h)=\Dom(h)$ and $uh=hu$.
 \item[2)] \cite{MurVon, Shu}
 A linear subspace $L$ in $H$ is said to be essentially dense if for any $\epsilon>0$, there exists a closed $\VN-$Hilbert submodule $L_{\epsilon}$ contained in $L$ such that $\dim_{\VN}(H\ominus L_{\epsilon})\leq \epsilon$.
\end{definition}

We have the following important properties:
\begin{lemma}\cite{MurVon, Shu}\label{essentially dense} Let $\VN$ be a finite Von Neumann algebra.
\item[1)]
Let $h:H_{1}\to H_{2}$ be a morphism of $\VN-$Hilbert modules. Let $L$ be a linear subspace of $H_{2}$ which is essentially dense. Then $h^{-1}(L)$ is essentially dense in $H_{1}$.
\item[2)] Assume that $h$ is $\VN-$Fredholm. Then $\Ran(h)$ is essentially dense in $\Adh{\Ran (h)}$.
\item[3)] Assume that $h: H_{1}\to H_{2}$ is a $\VN-$Fredholm weak isomophism. Then for any closed $\VN-$subspace $F$ in $H_{2}$, $h_{|h^{-1}(F)}:h^{-1}(F)\to F$ is a weak isomorphism.
\end{lemma}
The closed densely defined unbounded operators $\rho(x)$ with $x\in l^{2}(M,t)$ are examples of operators affiliated to $\rho(M)$.
The bi-commutant theorem implies that $h$ is affiliated to $\VN$ iff $f(h)\in \VN$ for every bounded Borel function on $Spec (h)$ (see \cite{Ped} 5.3.10). In particular, if $h\geq 0$ then $h$ is affiliated to $\VN$ iff $(1+\epsilon h)^{-1}h\in \VN$ for some $\epsilon>0$.

An essential property of affiliated operators to $(\VN,t)$, a finite von Neumann algebra, is that they form an algebra (a property valid for more general semi-finite von Neumann algebra, see \cite{Luc} Chap. 8, \cite{MurVon} Chap. XVI, \cite{Shu}): The spectral theorem implies that the domain of an operator affiliated to $\VN$ is essentially dense and the above lemma proves that intersection of two essentially dense subspaces is an essentially dense subspace.

The following lemma is a variation on the description of $\Adh{\pi(\VN)x}$ in term of affiliated operators as in \cite{MurVon}  9.2,  developped as Radon-Nykodim theorems in \cite{Dye}, \cite{Ped} 5.3:
\begin{lemma}\label{T-lemma}Let $(\VN,t)$ be a finite von Neumann algebra. 
Let $\pi:\VN\to\B(H)$ be a faithful  represention of $\VN$ as a von Neumann subalgebra of $\B(H)$. Let $\Op$ be an operator in the commutant  $\pi(\VN)'$ of $\pi(\VN)$ which is $\VN-$Fredholm.
For any $\zeta\in \Adh{\Ran (\Op)}$, there exists $r\in \VN$ such that
 $\pi_{t}(r)$(\ref{standarde}) is injective with dense range and $\pi(r)\zeta\in \Ran(\Op)$.     
\end{lemma}
\begin{proof}
\item[1)] First, assume that $H=l^{2}(\VN,t)$ and that $\rho(\zeta),\,\Op\in \VN'$ have dense ranges. Then $\Op=\rho(r)$ with $r\in \VN$ (\ref{standard}). Let $\Op^{-1}\circ \rho(\zeta)=up$ be the polar decomposition of the operator $\Op^{-1}\circ\rho(\zeta)$ affiliated to $\VN'$. Then $u$ is an isometry, $p$ is positive, $u$, $p$ are affiliated to $\VN'$. But $up=[up(1+p)^{-1}](1+p)$. Note that $1+p\geq 1$ hence $(1+p)^{-1},\,p(1+p)^{-1}$ belong to $\VN'$. From (\ref{standard}), there exists $a,y\in \VN$, such that $(1+p)^{-1}=\rho(a)$ and $\rho(y)=up(1+p)^{-1}$. Then $\rho(\zeta)\circ\rho(a)=  \Op\circ \rho(y)$. But $\rho(\zeta)\circ \rho(a)=\rho(a.\zeta)$ and $\Op\circ \rho(y)=\rho(\Op(y))$ (identity is valid on the dense subset $\VN$). Then $\Ran(\rho(a))=\Dom(1+p)$ is dense.
\item[2)] In general, let $\zeta\in \Adh{\Ran (\Op)}$. Then $\Adh{\pi(M)\zeta}\subset \Adh{\Ran (\Op)}$. Let $\Op_{\perp}:\Ker( \Op)^{\perp}\to \Adh{\Ran \Op}$ and \\ $\Op_{1}=\Op_{\perp}\oplus id_{l^{2}(\VN,t)}: \Ker( \Op)^{\perp}\oplus l^{2}(\VN,t)\to \Adh{\Ran (\Op)}\oplus l^{2}(\VN,t)$. Set $\pi_{1}=\pi\oplus\lambda$. Lemma \ref{essentially dense} implies that $$\Op_{1}: {\Op_{1}}^{-1}\left(\Adh{ \pi_{1}(\VN)(\zeta\oplus 1) }\right)\to \Adh{ \pi_{1}(\VN)(\zeta\oplus 1) }$$ is a $\VN-$linear weak isomorphism, for $\Op_{1}$ is a $\VN-$Fredholm weak isomorphism.

Moreover the  vector state $x\to <\pi(x)\zeta,\zeta>+t(x)$ associated to $(\zeta,1)$ dominates $t$. 
Using the GNS construction \ref{GNS} (5), we deduce that 
there exists $\VN-$isomorphisms\\  $U_{1}:{\Op_{1}}^{-1}\left(\Adh{\pi_{1}(\VN)(\zeta\oplus 1)}\right) \to l^{2}(\VN,t)$ and  $U_{2}:\Adh{\pi_{1}(\VN)(\zeta\oplus 1)}\to l^{2}(\VN,t)$. Setting $\tilde\zeta=U_{2}(\zeta,1)$ and $\tilde \Op=U_{2} \circ \Op_{1}\circ U_{1}^{-1}$, we are reduced to the first case of the proof.
\end{proof}

From the proof, one sees that if $x\in l^{2}(G)$,  then the conductor of the affiliated operator $\rho(x)$ to $N_{r}(G)$ is non trivial.
For further reference, we give the following example of weak isomorphism:
\begin{lemma}\label{example of weak isomorphism} Let $G$ be an infinite discrete  group and let $\mu\in l^{1}(G,\RR)$ be a probability measure on $G$ such that the subgroup generated by the support of $\mu$ is equal to $G$. Then convolution with $1-\mu$ defines a weak isomorphism on $l^{2}(G)$.
\end{lemma}
\begin{proof} The hypothesis implies that the only $\mu-$harmonic function in $l^{2}(G)$ is the null function (see Woess \cite{Woe} p. 159). 
\end{proof}

\subsection{$G-$Hilbert Modules and von Neumann dimension.}
Let $G$ be a discrete  group.
A (left) $G-$Hilbert module is a Hilbert space $V$ with a unitary action $U(.)$ of $G$ such that $V$ is $G-$isometric to a $G-$invariant subspace of the free Hilbert $G-$module $H\otimes l^{2}(G)$. Then the Von Neumann algebra generated by $\{U(.)\}$ is isomorphic to $N_{l}(G)$. If $\dim_{\CC}H<+\infty$, then $V$ is said to be finitely generated. Let $V$ be embedded as a closed $G-$invariant subset of $H\otimes l^{2}(G)$. Let $P\in N_{r}(G)\bar\otimes\B(H)$ be the orthogonal projection onto $V$. Then $dim_{N(G)}V=Tr P$.

\begin{example}(see \cite{Shu}~ section 3) If $\tilde E\to \tilde X$ is the pullback of a   smooth vector bundle $E\to X$ under a $G-$covering map $\tilde X\to X$, then the space 
of sections of $\tilde E$ with coefficients in a Sobolev space is a $G-$Hilbert module (see \ref{Sobolev spaces}).
\end{example}
Let  $\Mod(N(G))$ be the category of $N(G)-$modules where $N(G)$ is viewed as an abstract ring. \\
A finitely generated projective $N(G)-$module is represented by an idempotent matrix\\ $A\in M_{n}(N(G))$. Define $\dimG P:= \tr A$.  Following L\"{u}ck \cite{Luc} Chap. 6, we define the dimension function of a $N(G)-$module $M$ by 
\label{definition of dimension}
\begin{eqnarray*}
 \dimG M:=\sup\{\dimG P\,:\, P\subset M\text{ a finitely generated projective submodule}\}\,.
 \end{eqnarray*}
 
If $V$ is a Hilbert $N(G)-$module, the two dimension functions agree: Their dimension is the supremum over the finite dimensional subspaces (\cite{Luc} p. 21 th.1.12 and th. 6.24).

\subsection{Localisation at a Torsion theory}\label{torsion theory}In this section we introduce the main categorical tool:
quotient of an abelian category $\A$ by a subcategory $\B$. 
This gives the category for computating $''\Mod \B''$. Serre \cite{Ser} used it in algebraic topology to obtain isomorphisms modulo finite groups. The motivating example of isomorphism between the $l^{2}-$cohomology and the reduced $l^{2}-$cohomology up to torsion was discussed in the introduction. 

An equivalent well known construction is Verdier's localisation of the category $\A$ by a class of morphisms $\I$, so that morphisms in $\I$ become isomorphisms in $\I^{-1}\A$ (see \cite{Ver}).  

We work with the category of modules over a ring $R$. Then we speak of torsion theory and localisation at a torsion theory.

\begin{definition}A Serre subcategory $\T$ of an abelian category $\A$  is a full subcategory $\T$ of $\A$ such that for any exact sequence $0\to A\to B\to C\to 0$ in $\A$,   $B$ belongs to $\T$ iff $A$ and $C$ belongs to $\T$.
\item[2)] In the category  $\Mod(R)$ of $R-$modules over a ring $R$, a Serre class $\T$ defines a hereditary torsion theory  $\tau=(\T,\F)$ on $R$ (\Vas, \Gol):
Modules in $\T$ are torsion modules. Define the class of free modules by $\F=\{F\in \Mod(R)\,: \forall T\in\T,\, Hom_{R}(T,F)=0\}$.
\end{definition}

\begin{lemma} Let $\tau=(\T,\F)$ be a hereditary torsion theory, then
\item[1)] $\F$ is closed under submodules, direct products and extension.
\item[2)] Any $M\in\Mod(R)$ has a unique maximal $\tau-$torsion submodule, denoted $T_{\tau}(M)$.
\item[3)]\label{cogenerated by injective} A hereditary torsion theory is cogenerated by an injective module $E$: $\tau=(\T,\F)$ with $\T\ni S\iff Hom_{R}(C,E)=0$. 
\item[4)]\label{torsion functor} The functor $T_{\tau}(.): \Mod(R)\to \T$ is a left exact functor ($N\subset M$ implies that $T_{\tau}(M)\cap N=T_{\tau}(N))$.
\end{lemma}
\begin{proof}\item[2)] Let $\{N_{i},\, i\in\Lambda\}$ be the set of all torsion submodules of $M$. The trivial module $\{0\}$ belongs to this set. Then $T_{\tau}(M):=\sum_{i\in\Lambda}N_{i}$ satisfies required properties for it is a homomorphic image of $\oplus_{i\in\Lambda}N_{i}$ which is torsion.
\item[3)\,, 4)] See \cite{Gol}~ p. 5 and p.  24.
\end{proof}

\begin{definition}
\item[1)]
Let $\tau_{1}$ and $\tau_{2}$ be torsion theories, then $\tau_{1}$ is smaller than $\tau_{2}$ ($\tau_{1}\leq \tau_{2}) \text{ if  }\T_{1}\subset \T_{2}$ iff $ \F_{1}\supset \F_{2}\,.$
\item[2)] Then if $\C$ is a class in $\Mod(R)$, the hereditary torsion theory generated by $\C$ is the smallest hereditary torsion theory $\tau$ such that $\C\subset\T$.
\end{definition}
\subsubsection{Quotient category. }
Let $\T$ be a Serre subcategory of an Abelian category $\A$. There exists a quotient category $\A/\T$ and an exact functor $\A\to \A/\T$ with the universal property of factorisation of functor mapping objects in $\T$ to the zero object (\Grot (1.11), Gabriel \cite{Gab} Chap. 3, Verdier \cite{Ver} \para2).
\begin{definition}Let $\tau=(\T,\F)$ be a hereditary torsion theory on $\Mod(R)$. Let $\Mod(R)_{/\tau}$ be the quotient category of $\Mod(R)$ by the Serre class $\T$: 
\item[i)] The objects of $\Mod(R)_{/\tau}$ are identical with the objects of $\Mod(R)$; \item[ii)]The morphisms of $\Mod(R)_{/\tau}$ are elements of the following inductive limit $$\varinjlim\{Hom_{R}(M',N/N')\,:M'\subset M,\, N'\subset N \text{ and } M/M',\, N'\in\T\}\,.$$ 
\end{definition}

Let $\alpha\in Hom_{N(G)}(M,N)$ and let $[\alpha]$ be its image in the quotient category. Then $[\alpha]$ is a monomorphism (resp. epimorphism, resp. isomorphism) iff $\Ker(\alpha)$ (resp. $\Cok(\alpha)$, resp. $\Ker(\alpha)$ and $\Cok(\alpha)$) is a torsion module.
\subsubsection{Remark.}\label{localisation}
Another description of $\Mod(R)_{/\tau}$ is through fractions: 
Let $\I$ be the set of morphisms in $\Mod(R)$ such that $\Ker(f)$ and $\Cok(f)$ are in $\T$. Then $\I$ is a multiplicative system and $\I^{-1}\Mod(R)$ and $\Mod(R)_{/\tau}$ are equivalent categories (\cite{Ver}, \cite{Wei} 10.3.4 and ex. 10.3.2).
Hence a morphism $[f_{1}]\in \Hom_{\Mod(R)_{/\tau} }(M,N)$ is represented by a left fraction $M\overset{i}{\ot} M'\overset{f_{1}}{\to} N$ with $i\in \I$. 

\subsubsection{Torsion theory and complexification}\label{torsion theory and complexification}

Let $R$ be a ring which is also an $\RR-$algebra then $R\otimes_{\RR}\CC$ is a $\CC-$algebra. 
If $A\in \Mod(R)$ then the $(R,\CC)-$isomorphism $(R\otimes_{\RR}\CC)\otimes_{R} A \to A\otimes_{\RR}\CC$ 
gives a structure of $R\otimes_{\RR}\CC-$module to $A\otimes_{\RR}\CC$.
%
The forgetful functor $\phi:\Mod(R\otimes_{\RR}\CC)\ni B\to B_{R}\in \Mod(R)$ is faithful and exact and has a left adjoint 
\begin{eqnarray}
(R\otimes_{\RR}\CC)\otimes \, . \; : Hom_{R}(A,E_{R}) &\simeq& Hom_{R\otimes_{\RR} \CC}((R\otimes_{\RR} \CC)\otimes_{R} A, E)
\end{eqnarray}

%
Let $J=\phi(iId_{B})$ be the automorphism in $B_{R}$ defined by multiplication by $i$ in $B$ and let $J_{\CC}=J\otimes1_{\CC}:B_{R}\otimes_{\RR}\CC\to B_{R}\otimes_{R}\CC$ be its complex linear extension. Let $\bar B$ be the complex conjugate $R\otimes_{\RR}\CC-$module associated to $B$: Its underlying abelian group is $B$, and the module structure is given by the following representation $\rho$ of $R\otimes_{\RR}\CC$ in  $End(B)$: $\rho(r\otimes c)(b)=(r\otimes\bar c).b$. 
The natural antilinear $R-$isomorphism $B\to \bar B$ defines a functor on $\Mod(R\otimes_{\RR}\CC)$. Then antilinear $R-$linear maps from $B$ to $B'$ are in bijection with $R\otimes_{\RR}\CC-$linear maps from $B$ to $\Adh{ B'}$ (or from $\Adh{B}$ to $B'$). An $R\otimes_{\RR}\CC-$isomorphism from $B$ to $\bar B$ is equivalent with a real structure $\alpha: B\simeq_{R\otimes_{\RR}\CC} A\otimes_{\RR}\CC$ on $B$ ($A\in \Mod(R)$), which in turn is equivalent to a conjugation $C$ on $B$ (an $R-$linear involution $C$ on $B_{R}$ such that $CJ=-JC$). This defines an antilinear $\RR-$isomorphism $Hom_{R\otimes_{\RR}\CC}(A\otimes\CC,\bar E)\simeq Hom_{R\otimes_{\RR}\CC}(A\otimes\CC,E)$.
\begin{lemma}\label{identification by complex conjugation}
\item[1)] Let $B\in \Mod(R\otimes_{\RR}\CC)$, then $B_{R}\otimes_{\RR}\CC\simeq_{R\otimes_{\RR}\CC} B\oplus\bar B$. 
\item[2)] Let $C$ be a conjugation on $B\in \Mod(R\otimes \CC)$. Then $C_{\CC}:=C\otimes Id_{\CC}: B_{R}\otimes_{\RR}\CC\to B_{R}\otimes_{\RR}\CC$ maps $B$ to $\bar B$. 
\end{lemma}
\begin{proof} 
\item[1)] Let $P_{\pm}=\frac{Id\mp iJ_{\CC}}{2}: B_{R}\otimes\CC\to B_{R}\otimes\CC$ be the projection onto the eigenspaces $\{J_{\CC}=\pm i Id\}$.
One checks that  $B\ni b\to P_{+}(b\otimes 1)\in \{J_{\CC}= i Id\} \text{ and } \bar B\ni b\to P_{-}(b\otimes 1) \in \{J_{\CC}= -i Id\}\, $ are $R\otimes_{\RR}\CC-$isomorphisms.
\item[2)] The complex extension $C_{\CC}$ of $C$ is an involution which anticommutes with $J$.
\end{proof}
Note that $B\ni b\to P_{+}(b\otimes 1)\in B_{R}\otimes_{\RR}\CC$ is a splitting of the natural surjection $B_{R}\otimes_{\RR}\CC\to B$.
\begin{definition}[see \cite{Gol}]\label{ring extension} Let $\gamma:R\to R\otimes_{\RR}\CC$ be the above ring extension.
\item[i)] Let $\tau=(\T,\F)$ be a torsion theory on $\Mod(R)$. Then $\gamma_{*}\tau$ is the torsion theory on $\Mod(R\otimes_{\RR}\CC)$ such that $B\in \Mod(R\otimes_{\RR}\CC)$ is $\gamma_{*}\tau-$torsion iff $B_{R}$ is $\tau-$torsion.
\item[ii)] Let $\tau_{\CC}=(\T_{\CC},\F_{\CC})$ be a torsion theory on $\Mod(R\otimes_{\RR}\CC)$. Then $\gamma^{*}\tau_{\CC}$ is the torsion theory on $\Mod(R)$ such that $A\in \Mod(R)$ is $\gamma^{*}\tau_{\CC}-$torsion iff $A\otimes_{\RR}\CC$ is $\tau_{\CC}-$torsion.
\end{definition}
One checks that if $\tau$ is cogenerated by an injective $I$ (\ref{cogenerated by injective}), then $\gamma_{*}\tau$ is cogenerated by the injective $I\otimes_{\RR} \CC$. Hence $\gamma^{*}\gamma_{*}\tau=\tau$ for $I\otimes_{\RR}\CC\simeq_{R}I^2$.
However if $\tau_{\CC}$ is a torsion theory on $\Mod(R\otimes_{\RR}\CC)$, $\gamma_{*}\gamma^{*}\tau_{\CC}$ is in general strictly smaller than  $\tau_{\CC}$: $B$ torsion does not imply $B\oplus\bar B$ torsion.
\begin{definition} Let $\tau_{\CC}=(\T_{\tau_{\CC}},\F_{\tau_{\CC}})$ be a torsion theory on $\Mod(R\otimes_{\RR}\CC)$ with torsion functor $T_{\tau_{\CC}}$ (\ref{torsion functor}). The following properties are equivalent:
\item[i)] $\gamma_{*}\gamma^{*}\tau_{\CC}=\tau_{\CC}$.
\item[ii)] There exists a torsion theory $\tau$ on $\Mod(R)$ such that $\tau_{\CC}=\gamma_{*}\tau$.
\item[iii)] $\T_{\tau_{\CC}}$ is stable by conjugation.
\item[iv)] $\forall B\in \Mod(R\otimes_{\RR}\CC)$, $T_{\tau_{\CC}}(\bar B)=\Adh{T_{\tau_{\CC}}(B)}$ (\ref{torsion theory and complexification}).

\item[ ]
A (hereditary) torsion theory on $\Mod(R\otimes_{\RR}\CC)$ is real if it satisfies one of the above properties.
\end{definition}
\begin{example}Let $\tau_{\CC}$ be any torsion theory then $\gamma_{*}\gamma^{*}\tau_{\CC}$ is a real torsion theory.
\end{example}
\begin{corollary}\item[1)]Let $(\tau,\tau_{\CC},\tau')$ be torsion theories on $(R,R\otimes_{\RR}\CC,\,R)$ such that $  \tau\leq\gamma^{*}\tau_{\CC}$ and $\tau_{\CC}\leq \gamma_{*}\tau'$ (iff $ \T_{\tau}\otimes\CC\subset\T_{\tau_{\CC}}$ and $[\T_{\tau_{\CC}}]_{R}\subset \T_{\tau'}$). There exists exact functors 
\begin{eqnarray*}
(.)\otimes\CC:\Mod(R)_{/\tau}\to \Mod(R\otimes_{\RR}\CC)_{/\tau_{\CC}}& &
(.)_{R}:\Mod(R\otimes_{\RR}\CC)_{/\tau_{\CC}}\to \Mod(R)_{/\tau'}
\end{eqnarray*}
such that the following diagrams commute:
$$
\begin{diagram}\node{\Mod(R)}\arrow{e,t}{(.)\otimes\CC}\arrow{s,l}{}\node{\Mod(R\otimes_{\RR}\CC)}\arrow{s,l}{}    \node{\Mod(R\otimes_{\RR}\CC)}\arrow{e,t}{(.)_{R}}\arrow{s,l}{}\node{\Mod(R)}\arrow{s,l}{}\\
\node{\Mod(R)_{/ \tau}}\arrow{e,t}{(.)\otimes\CC}\node{\Mod(R\otimes_{\RR}\CC) _{/ \tau_{\CC}}}  \node{\Mod(R\otimes_{\RR}\CC)_{/ \tau_{\CC}}}\arrow{e,t}{(.)_{R}}\node{\Mod(R) _{/ \tau'} }
\end{diagram}
$$

\item[2)] In particular if $\tau=\tau'$ and $\tau_{\CC}=\gamma_{*}\tau$, then $((.)\otimes\CC,\, (.)_{R})$ is a pair of adjoint functors.
\item[3)] Let $(\tau',\gamma_{*}\tau')$ be torsion theories on $(R,R\otimes_{\RR}\CC)$ such that $\tau'\leq\tau$. Then there exists an exact functor from 
$(\Mod(R)_{/\tau'},\Mod(R\otimes_{\RR}\CC )_{/\gamma_{*}\tau'})$ to  $(\Mod(R)_{/\tau},\Mod(R\otimes_{\RR}\CC )_{/\gamma_{*}\tau})$ which commutes with tensor product up to an equivalence.
\end{corollary}
\begin{proof}The kernel of $\Mod(R)\overset{ (.)\otimes\CC}{\to} \Mod(R\otimes_{\RR}\CC) \to\Mod(R\otimes\CC)_{/\tau_{\CC}}$ is spanned by the modules $M$ such that $M\otimes_{\RR}\CC$ is $\tau_{\CC}$ torsion which follows from $M$ being $\tau$ torsion. One concludes from  \cite{Gab}~ p. 368 Cor. 2 and Cor. 3 or \cite{Fai} 15.9. 
The other assertions are proved in the same way.
\end{proof}
\begin{notations}Let $\tau$ be a (hereditary) torsion theory on $\Mod(R)$. Then $\tau\otimes\CC:=\gamma_{*}\tau$ is the real torsion theory on $\Mod(R\otimes_{\RR}\CC)$ associated to $\tau$.
\end{notations}

\begin{definition}Define a category $\R$ of modules with real structure modulo $\tau$:  
\item[]An object is a pair $(B,\alpha)$, $B\in \Mod( R\otimes_{\RR}\CC)_{/\tau\otimes\CC}$ with $\alpha:B\to  A\otimes_{\RR}\CC$ an isomorphism in $\Mod( R\otimes_{\RR}\CC)_{/\tau\otimes\CC}$.
\item[]  A morphism $f:(B,\alpha)\to (B',\alpha')$ in $\R$ is a map $f:B\to B'$ such that there exists $g:A\to A'$ with $(g\otimes 1_{\CC})\circ \alpha=\alpha' \circ f$.
\end{definition}
\begin{lemma}  \item[1)] The category $\R$ is abelian.
\item[2)] A conjugation $C$ on $B\in \Mod( R\otimes_{\RR}\CC)_{/\tau\otimes\CC}$ (an involution on $B_{R}\in \Mod(R)_{/\tau}$ which anticommutes with $(iId_{B})_{R}$) defines a real structure modulo $\tau$: $B\overset{\alpha}{\simeq} Ker(C-Id_{B_{R}})\otimes \CC$ .  It induces an isomorphism ${C_{\CC}}_{|\bar B}:\bar B \to B$ in $\Mod( R\otimes_{\RR}\CC)_{/\tau\otimes\CC}$ (see \ref{identification by complex conjugation}).
\end{lemma}
\begin{notations}\label{conjugate filtration}
Let $B'\to B$ be a subobject of a module with real structure modulo $\tau$. Then $\bar B'$ will be identified with the subobject 
$\bar B'\to \bar B\overset{{C_{\CC}}_{|\bar B}}{\to }B$.
In particular, if $F^{.}$ is a filtration on $B$, then ${\bar F}^{.}:=(\Adh{ F^{.} })$ will be called the conjugate filtration on $B$. 
\end{notations}

\subsection{Examples of torsion theories}
\subsubsection{}
From Dickson \cite{Dic} section 3, if $\C$ is a class of modules defines 
\begin{eqnarray}L(\C):=\{B\in\Mod(R):\, \forall C\in\C,\,\, Hom_{R}(B,C)=0\},\\
R(\C):=\{B\in\Mod(R):\, \forall C\in\C,\,\, Hom_{R}(C,B)=0\}.
\end{eqnarray}
The torsion theory generated by $\C$ is given by $\T_{\C}=LR(\C)$.
\subsubsection{} 
A multiplicative system $S$ is a subset in $R$ stable by multiplication. Then $\T_{S}=\{M\in \Mod(R):\, \forall m\in M,\, \exists s\in S \text{ s.t.} sm=0\}$ is a Serre class.
We will use the torsion theory generated by $ \Adh{\Ran f}^{H_{2}}/{\Ran f}$ 
with $f:H_{1}\to H_{2}$ a bounded morphism between $N(G)-$Hilbert modules. If $A\in Mod(R)$, let $$S=\{r\in N(G) \text{ with dense range}:\, \exists a\in A\text{ with } ra=0\}\subset\cup_{a\in A}(0:a)\,.$$ The properties of essentially $G-$dense subsets (see \cite{Shu} or \cite{MurVon} Chap. XVI) imply that $S$ is a multiplicative system. 
A $N(G)-$module $M$ is $\tau_{S}=(\T_{S},\F_{S})-$torsion iff $\forall m\in M,\, \exists s\in S$ s.t. $sm=0$.  Note that intersection of multiplicative systems is a multiplicative system.

\subsubsection{}

\label{definition of dim tor}

In \cite{Luc}~Chap. 6, L\"{u}ck defines a dimension function $\dimG:\Mod(N(G))\to [0,+\infty]$ which is additive on short exact sequences and which coincides with Von Neumann dimension for $N(G)-$Hilbert modules. 
Therefore $$\T_{\dim}=\{M\in\Mod(N(G)):\, \dimG M=0\}$$ is a Serre class and defines a torsion theory $\tau_{dim}$ on $N(G)$. 
This torsion theory is real for a $N(G)-$module $P$ is projective finitely generated iff $\bar P$ is (see \ref{definition of dimension}).
Standard examples of $\tau_{dim}-$torsion modules  are given in \ref{example of torsion modules}: Modules of the shape $\frac{\bar A}{A}^{H}$, with $A$ a $G-$invariant subspace of a {\it finite } $G-$dimensional $G-$Hilbert modules $H$, are zero dimensional. This follows from the normality of the dimension function (see a proof in \ref{example of torsion modules}).
\subsubsection{} 

The algebra $\U(G)$ of affiliated operators to $N(G)$ is studied in \cite{MurVon}, \cite{Luc}~ Chap. 8. and  \Rei. An affiliated opertor (to $N(G)$)  is a $G-$equivariant unbounded operator $f: dom(f)\subset l^{2}(G)\to l^{2}(G)$. Let $M\in\Mod(N(G))$. Define $T_{\U}(M):=\Ker (M\to \U(G)\otimes_{N(G)}M)$. This defines a torsion class $$\T_{\U(G)}=\{M\in \Mod(N(G)):\,\, T_{\U}(M)=M\}$$ and a torsion theory $\tau_{\U(G)}$. An element $m\in M$ belongs to $T_{\U(G)}(M)$ iff there exists an $r\in N(G)$ which is a not a divisor of zero (iff $r$ is a weak isomorphism) such that $rm=0$. Indeed $r$ becomes invertible in $\U(G)$.
According to \cite{Vas} p. 673, a module $F\in \Mod(N(G))$ is $\tau_{\U(G)}-$torsion free iff $F$ is flat.
This torsion theory is real for $F$ is flat iff $\bar F$ is.

Exemples of $\tau_{\U(G)}$ torsion modules  are given in \ref{example of torsion modules}. We present the case of a module generated by elements with infinite isotropy.
\begin{lemma}\label{infinite isotropy} Let $H$ be an infinite finitely generated subgroup of the discrete  group $G$. Then $l^{2}(G)\otimes_{\CC[G]}\CC[G/H]$ is a $\tau_{\U(G)}-$torsion module. \end{lemma}
\begin{proof} Let $\mu\in l^{1}(H,\RR)\subset l^{1}(G,\RR)$ be a probability measure with finite support generating $H$.  Lemma \ref{example of weak isomorphism} implies  $\lambda(1-\mu)$ is a weak isomorphism on $l^{2}(G,\CC)\simeq_{\CC[H]} \oplus_{[g]\in H\diagdown G}l^{2}(H)$. But $\U(G)\otimes_{N(G)}l^{2}(G)\otimes_{\CC[G]}\CC[G/H]$ is spanned by elements $r\otimes f\otimes gH$ with $r\in  \U(G)$, $f\in l^{2}(G,\CC)$ and $g\in G$. Then 
\begin{gather*} 
r\otimes f\otimes g\,H    = r\lambda(g(1-\mu))^{-1} \otimes \lambda(g(1-\mu))f\otimes g\,H 
 =  r \lambda(g(1-\mu))^{-1} \otimes f\rho((1-\check{\mu})g^{-1})\otimes g\,H \\
 =r \lambda(g(1-\mu))^{-1} \otimes f\otimes (1-\check{\mu}).e\,H =   0 
\end{gather*}
 for $1-\check{\mu}=0$ in $\CC[G/H]$.
\end{proof}

\begin{remark} Let $p:\tilde X\to X$ be a covering of a compact manifolds. We will see that if the covering group $G$ is Galois, then the torsion theories $\tau_{dim}$ and $\tau_{\U(G)}$ on $\Mod(N(G))$ give valuable informations on the $l^{2}-$cohomology groups of $p:\tilde X\setminus p^{-1}(D)\to X\setminus D$. 

On the opposite, if one considers a covering map $p:\tilde X\to X$ with trivial automorphism, then $N(G)$ is isomorphic to $\CC$ and a torsion theory is either trivial or any module is torsion. 

\end{remark}

\subsection{The sheaves of $l^{2}-$direct images}\label{Direct $l^{2}-$image}
Let $p: \tilde X\to X$ be a covering map between complex manifolds. Let $G$ be the group  of deck transformations. Let $N(G):=N(G,\CC)$ be the its left von Neumann algebra (the bi-commutant of the set of left translations acting on $l^{2}(G,\CC)$) and let $\N(G)$ (or $\N(G,\CC)$) be the sheaf of rings it defines. Let $\N(G,\RR)$ be the sheaves of rings defined by $N(G,\RR)$ (\ref{real structures}).

According to Campana-Demailly \cite{CamDem}, if $\F$ is a coherent analytic sheaf on $X$, there exists a sheaf 
$\p\F$ called $l^{2}-$direct image such that $\p(.)$ is an exact functor from the 
category of coherent analytic sheaves on $X$ to the category of sheaves on $X$(\cite{CamDem} 
prop2.6). The sheaf $\p\O_{X}$ is the sheaf associated to the presheaf $V\to \O(p^{-1}(V))\cap L^{2}(p^{-1}(V),\CC)$. We change from notations of \cite{CamDem}, our $\p\F$ is written there $\p
\tilde\F$. Campana-Demailly \cite{CamDem} Cor 2.7 prove: 
\begin{lemma}\label{local isomorphism}For any analytic coherent sheaf $\F$, the morphism  $\p\O\otimes\F\to \p\F$ is an isomorphism.
\end{lemma}

This isomorphism defines on  $\p{\F}$  a structure of $\N(G)-$sheaf compatible 
with the natural structure of $\ZZ[G]-$sheaf.
\begin{definition} Let $K$ be a subring of $\CC$. Let $\p K$  be the locally constant sheaf
defined by the presheaf $$U\to \{ f\in L^{2}(p^{-1}(U), \CC)\,\,, f \text{ is } K-\text{valued and locally 
constant}\}\,.$$
\end{definition}
Then  $\p(\RR)$ is a $\N(G,\RR)-$module, $\p({\CC})$ is a $\N(G,
\CC)-$module. When $p:\tilde X\to X$ is Galois, $\p \CC$ is isomorphic to $l^{2}(G,\CC)\otimes_{\ZZ[G]}p_{!}(\ZZ_{\tilde X})$ as sheaves of left $\N(G,\CC)-$modules.
\begin{definition} Let $\F\to X$ be a coherent analytic sheaf (or a constant sheaf). \\ The $l^{2}-$cohomology groups $H^{.}(X,\p\F)$ of $p:\tilde X\to X$ with values in $p^{*}\F$ are the cohomology groups of the sheaf $\p\F$ over $X$. 
\end{definition}

\begin{lemma} 
Let $D:E\to F$ be a differential operator with holomorphic coefficients between 
holomorphic vector bundles. Then there exists an operator $\p D:\p E\to \p F$ .
\end{lemma}
\begin{proof}In local trivialisation $\O^n\simeq E$, $\O^m\simeq F$, $D$ is given 
as $\sum_{|I|\leq r}a_{I}\partial_{I}$ with $a_{I}$ an holomorphic function. But $\p E$, and 
$\p F$ are $\O$ modules, hence  it is enough to study $D=\partial_{I}$. The 
claim is then a consequence of the Cauchy inequalities.
\end{proof}
\begin{lemma}\label{squareholomorphicpoincare}$$0\to\p\un{\CC}\to\p\O\overset
{d}{\to}\p\Omega^1\ldots\overset{d}{\to}\p\Omega^n\to 0$$ is well-defined and 
exact.
\end{lemma}

Remark:
It seems natural to extend the functor direct $l^{2}-$images to an exact functor from the category of $D_{X}-$coherent modules to the category of $\N(G)-$sheaves. 
\subsubsection{Link to singular cohomology} \label{Link to singular cohomology}
In this section we recall that the cohomology of a locally constant sheaf is isomorphic to the singular cohomology with local coefficients and to the equivariant cohomology of the universal cover $\tilde X$ with values in a $\pi_{1}(X)-$module (\Eil, \Ste, \Whi). We refer to Dimca \cite{Dim} section 2.5 for a more detailed presentation.

A bundle of groups $L$ (or local system of groups) on $X$ is a covariant functor from the fundamental groupo\"{i}d of  $X$ to the category of abelian groups (\cite{Whi} p. 257). 

A locally constant sheaf $\L$ defines a bundle of groups $L$ for  a locally constant sheaf over a simply connected space is constant (\cite{MebNar} I.2). Hence a path $\gamma:[0,1]\to X$ defines the isomorphism $\L_{\gamma(0)}\simeq \gamma^{*}(\L)_{0}\simeq \gamma^{*}(\L)_{1}\simeq \L_{\gamma(1)}$ with $\L_{x}$ the stalk of $\L$ at a point $x\in X$. 

A bundle of groups $L$ on $X$ and a point $x_{0}\in X$ define a $\pi_{1}(X,x_{0})-$module: let $\gamma_{x_{0}}$ be the constant path at $x_{0}$. Then $L(x_{0}):=L(\gamma_{x_{0}})$ is a $\pi_{1}(X,x_{0})-$module through the representation $\rho_{L}: \pi_{1}(X,x_{0})\ni \gamma\mapsto L(\gamma)\in Hom(L(\gamma_{x_{0}}),\, L(\gamma_{x_{0}}))$.

A $\pi_{1}(X)-$module $M$ defines a locally constant sheaf: the sheaf of cross sections of the fiber bundle $\tilde X\times_{\pi_{1}(X)}M\to X$.

Let $K$ be the kernel of the monodromy representation $\rho_{L}$. Then the pullback of  the bundle of groups $L$  and of the locally constant sheaf $\L$  are constant on $p_{K}: X_{K}\to X$, the covering with fundamental group $K$.  $L(x_{0})$ is a $\pi_{1}(X,\, x_{0})/K-$module.

Let $C^{p}_{s}(X,L)$ be the group of singular cochains with values in $L$: This is the set of functions $c$ which assigns to a singular simplex $\sigma:\Delta_{p}\to X$ an element $c(\sigma)\in L(\sigma(e_{0}))$. This is a group under addition of functional values. This lead to a complex $(C^{.}_{s}(X,L),\delta)$ whose cohomology $H_{sing}^{k}(X,L):=H^{k}(C^{.}_{s}(X,L),\delta)$ is by definition the singular cohomology  group of $X$ with values in the local system $L$ (see  \cite{Whi} p. 270).

Let $\L$ be a locally constant sheaf over $X$ with associated local system $L$. Let $(\C^{.}_{s}(X,\L),\delta)$ be the differential sheaf associated to the presheaf $U\to (C^{.}_{s}(U,L_{|U}),\delta)$ (equivalent to the definition given in \cite{Bre} I.7). It defines a $\Gamma(X,.)-$acyclic resolution of $\L$ and provides an isomorphism between the sheaf cohomology groups of $\L$ and the singular cohomology groups of $L$ (see \cite{Bre} Chap. III).

Let $X_{K}\to X$ be the cover of $X$ with group of deck transformations $\pi_{1}(X,\, x_{0})/K$. Let $(C_{.}(X_{K}),\delta)$ be its singular chain complex. Then $(Hom_{\CC[\pi_{1}(X,\, x_{0})/K]}(C_{.}(X_{K}),L(x_{0})),\delta)$ is a complex whose cohomology $H^{q}_{e}(X_{K},L_{x_{0}})$ is by definition the group of equivariant cohomology of $X_{K}$ with values in $L(x_{0})$.
From \cite{Eil} \para. 24, we have an isomorphism $H^{k}_{e}(X_{K},L(x_{0}))\simeq H_{sing}^{k}(X,L)$.
\subsection{Mixed Hodge structures}
We recall or adapt definitions and fundamental results from Deligne's mixed Hodge structures (\cite{Del}, \cite{Deltrois}). The main points are contained in two theorems: the theorem on strictness of morphisms and the theorem on degenerescence of the two spectral sequences associated to a mixed Hodge complex.
\subsubsection{ }\label{pseudo-morphism}
We recall (see Peters-Steenbrink \cite{PetSte} p. 49) that a pseudo-morphism $f:K^{.}\dto L^{.}$ between two complexes $K^{.}$ and $L^{.}$ in an abelian category is a chain of morphisms $$K^{.}\overset{f}{\to}K^{.}_{1} \overset{ f_{1} }{\ot}K^{.}_{2}\overset{f_{2}}{\to}\ldots\overset{f_{n}}{\to}K^{.}_{n+1}=L^{.}$$ such that $f_{1},\ldots,f_{n}$ are quasi-isomorphisms. It induces a morphism in the derived category. One says that $f:K^{.}\dto L^{.}$ is a pseudo-isomorphism when $f$ is a quasi-isomorphism. When $K^{.}$ and $L^{.}$ are filtered complexes it is understood that morphisms are filtered and that a quasi-isomorphism is a filtered quasi-isomorphism ($Gr_{F}^{.}(f_{i})$ is a quasi-isomorphism).

For further reference,  
we recall (\cite{Del} $1.4.6$) that if $(K^{.},d^{.})$ is a complex in an abelian category then: 
\begin{trivlist}
\item[-] The canonical filtration $\canfilt$ on $K$ is defined by: $\tau_{\leq p}K^{n}$ is equal to $K^{n}$ if $n<p$, to $Ker(d^{p})$ if $n=p$ and is the null object if $p<n$. \label{filtration tau}
\item[-] The trivial filtration   is defined by: $(\sigma_{.\geq p}K)^{n}$ is equal to the null object if $n<p$ and $K^{n}$ if $p\leq n$.
\end{trivlist}
 A quasi-isomorphism $f:K^{.}\to L^{.}$ defines a filtered quasi-isomorphism $f:(K^{.},\canfilt)\to (L^{.},\canfilt)$.


%
\begin{definition}\label{definition of MHS}Let $R$ be a ring which is also an $\RR-$algebra. Let $\tau$ be a (hereditary) torsion theory on $\Mod(R)$. 
\item[1)]A mixed Hodge structure $H=(H_{R},W,F)$ in $\Mod(R)_{/\tau}$
is given by

\item[i)] a left $R-$module $H_{R}$,
\item[ii)] a filtration $W$ on $H_{R}$ in $\Mod(R)_{/\tau}$,
\item[iii)] a filtration $F$ on $H_{R}\otimes_{\RR}\CC$ in $\Mod(R\otimes_{\RR}\CC)_{/\tau\otimes\CC}$ such that $W_{\CC},F,\bar F$(\ref{conjugate filtration}) are opposed (see \cite{Del} 1.2) in $\Mod(R\otimes_{\RR}\CC)_{/ \tau\otimes\CC}$.
\item[2)] A morphism of mixed Hodge stucture $f:H\to H'$ in $\Mod(R)_{/\tau}$  is a morphism \\$f\in Hom_{\Mod(R)_{/\tau}}(H_{R}, H'_{R})$ such that $f$ is compatible with $W$ and $ f_{\CC}$ is compatible with $F$. 
\end{definition}
 Theorem 1.2.10 of Deligne \cite{Del} implies:
\begin{theorem}\label{MHSabelian} The category of mixed Hodge structure in $\Mod(R)_{/\tau}$ is abelian. A morphism $f:H\to H'$ between MHS in $\Mod(R)_{/\tau}$ is strict for the filtrations.
\end{theorem}

In this article, we use only mixed Hodge complexes over $\RR-$algebras. This reflects the use of $l^{2}(G,\RR)$. 

\begin{definition}\label{Hodgecomplex}Let $R$ be a ring which is also an $\RR-$algebra. Let $\tau'$ and $\tau$ be torsion theories on $\Mod(R)$ such that $\tau'$ is smaller than $\tau$.
\item[] A $(\tau',\tau)-$Hodge complex  $(K_{R},(K_{R\otimes \CC},F))$ of weight $m$ is given by
\item[1)] A bounded below complex  of modules $K_{R}$ in $\Mod( R)_{/\tau'}$ 
\item[2)] A bounded below filtered complex of modules $(K_{R
\otimes\CC},F)$ in  $\Mod(R\otimes_{\RR}\CC)_{/\tau'\otimes\CC}$. 
\item[3)] A pseudo-morphism of bounded below complexes $\alpha:K_{R}
\dto K_{R\otimes_{\RR}\CC}$ 
(comparison morphism)  in $\Mod(R)_{/\tau'}$ such that $\alpha\otimes Id:K_
{R}\otimes\CC\dto K_{R\otimes \CC}$ is a pseudo-isomorphism. 

The isomorphism $H(K_{R})
\otimes\CC\tilde\to H(K_{R\otimes\CC})$ (in $\Mod(R\otimes\CC)_{/\tau'\otimes\CC'}$) defines a real structure on $H(K_{R\otimes
\CC})$. 
\item[] 
One requires that 
\item[1)] $d$ is strictly compatible with $F$ in $\Mod(R\otimes\CC)_{/\tau\otimes\CC}$.  
\item[2)]  $F$ and $\bar F$ (\ref{conjugate filtration}) are $m+k-$opposed on $H^k(K_{R\otimes\CC})\simeq H^
{k}(K_{R})\otimes\CC$  in $\Mod(R\otimes\CC)_{/\tau\otimes\CC}$.
\end{definition}

\begin{definition}(Following \cite{PetSte})
 \item[] Let $R$ be a ring which is also an $\RR-$algebra. Let $\tau',\tau$ be  torsion theories on $R$ such that $\tau'$ is smaller than $\tau$. 
\item[] A $(\tau',\tau)-$mixed Hodge complex $((K_{R},W), (K_{R\otimes \CC},W,F),\beta)$ consists in 
\item[1)] A bounded below filtered complex  $(K_{R},W)$ in $\Mod(R)_{/\tau'}$, 
\item[2)] A bounded below bi-filtered complex $(K_{R\otimes \CC},W,F))$ in $\Mod({R\otimes\CC})_{/\tau'\otimes\CC}$  and a pseudo-morphism $\beta:(K_{R},W)\dto (K_{R\otimes \CC},W)$ (comparison morphism) in the category of bounded below filtered complexes  in $\Mod(R)_{/\tau'}$ inducing a pseudo-isomorphism  $\beta\otimes Id_{\CC}:(K_{R}\otimes\CC,W)\dto (K_{R\otimes \CC},W)$.
\item[3)] One requires that for each $n$, $(Gr^{W}_{n}(K_{R}),(Gr^{W}_{n}(K_{R\otimes \CC},F))$ with pseudo-morphism
 $$Gr^{W}_{n}(\beta):Gr^{W}_{n}(K_{R})\dto Gr^{W}_{n}(K_{R\otimes\CC})$$ is a $(\tau',\tau)-$Hodge complex of weight $n$ .
 %
\end{definition}
A $(\tau',\tau)-$mixed Hodge complex will also be called a mixed Hodge complex mod$(\tau',\tau)$.  
\begin{example}
\item[1)]
 If $\tau'=(\{0\},\Mod(R))$ is the trivial torsion theory (no non zero torsion submodule), then a $(\tau',\tau)-$mixed Hodge complex will be referred as a $\tau-$mixed Hodge complex or a mixed Hodge complex modulo $\tau$. Therefore complexes and pseudo-morphisms are data in $\Mod(R)$ and properties of degenerescence are assumed in $\Mod(R)_{/\tau}$.
\item[2)] When $\tau'=\tau$, we will speak of mixed Hodge complex in $\Mod(R)_{/\tau}$. Hence each mixed Hodge complex modulo $\tau$ defines a mixed Hodge complex in $\Mod(R)_{/\tau}$.
\end{example}

The following theorem contains the  ''lemme des deux filtrations'' (\cite{Del} 1.3 and \cite{Deltrois} 7.2): 
\begin{theorem}\label{lemme des deux filtrations} Let $((K_{R},W), (K_{R\otimes \CC},W,F))$ be a mixed Hodge complex mod$(\tau',\, \tau)$. Then the recurrent filtration and the direct filtration induced by $F$ on $E_{r}^{p,q}(K_{R\otimes \CC},W)$ are equal in $\Mod(R\otimes\CC)_{/\tau\otimes\CC}$. The  sequence $0\to E_{r}(F^pK,W)\to E_{r}(K,W)\to E_{r}(K/F^pK)\to 0$ is exact for any $1\leq r\leq +\infty$ and for any $p$. The weight spectral sequence $(E_{r}(K_{R},W))_{r}$ degenerates at  $E_{2}$ in $\Mod(R)_{/\tau}$.  The Hodge spectral sequence $(E_{r}(K,F))_{r}$ degenerates at $E_{1}$ in $\Mod(R\otimes\CC)_{/\tau\otimes\CC}$.
\end{theorem}
\begin{proof}
 We have the following data within a torsion theory $\tau'$ smaller than $\tau$:
There exists a left fraction 
$$
\begin{diagram}
\node{} \node{(K'_{R\otimes\CC},W)}\node{}\\
\node{(K_{R},W)} \arrow{ne,t}{\beta_{1}}\node{} \node{(K_{R\otimes\CC},W,F)}\arrow{nw,tb}{\beta_{2}}{\tilde{qi}}
\end{diagram}
$$ with $\beta_{1}$ a morphism  in $\Mod(R)_{/\tau'}$ and $\beta_{2}$ a quasi-isomorphism in $\Mod(R\otimes_{\RR}\CC)_{/\tau'\otimes\CC}$ such that $\beta_{1}\otimes 1_{\CC}$ is a filtered quasi-isomorphism.
If, $r\geq 1$, the morphisms
\begin{eqnarray}
 E_{r}(\beta_{1}):E_{r}(K_{R},W)\to E_{r}(K'_{R\otimes\CC},W)\\
E_{r}(\beta_{1}\otimes 1_{\CC}):E_{r}(K_{R}\otimes \CC,W)\to E_{r}(K'_{R\otimes\CC},W)\\
E_{r}(\beta_{2}):E_{r}(K_{R\otimes\CC},W)\tilde\to E_{r}(K'_{R\otimes\CC},W) 
\end{eqnarray}
define morphisms of spectral sequences
\begin{eqnarray}
E_{r}(\beta)=  E_{r}(\beta_{2})^{-1} \circ E_{r}(\beta_{1})\\
E_{r}(\beta_{\CC})=  E_{r}(\beta_{2})^{-1} \circ E_{r}(\beta_{1}\otimes 1_{\CC})
\end{eqnarray}
Then $E_{r}(\beta_{1})\otimes1_{\CC}\simeq E_{r}(\beta_{1}\otimes 1_{\CC})$ is an isomorphism. Moreover the morphism of spectral sequences 
$(E_{r}(\beta))_{r}$ defines a real structure $\alpha_{r}$ on $E_{r}(K_{R\otimes\CC},W)$ such that  $d_{r}$ is real, for $$d_{r}\simeq (E_{r}(\beta)\otimes 1_{\CC}) (d_{r}\otimes 1_{\CC})\,.$$
Note that the real structure induced on $E_{r+1}\simeq H(E_{r})$ by $E_{r}$ is the same than its real structure. The differential
$d_{r}$ is compatible with the direct filtrations and their conjugates.

One reduces modulo $\tau$: 
One obtains a real structure modulo $\tau$ on each term of the spectral sequence, with $d_{r}$ real and compatible with the direct filtrations.
 
\item[1)] By hypothesis, $F_{d}=F_{d^*}=F_{rec}$ and their conjugates define a Hodge structure modulo $\tau$ of weight $-p+(p+q)=q$ on $E_{1}^{p,q}(W)=H^{p+q}(Gr_{-p}^{W})$ . But $d_{1}$ is real and compatible with $F$ so that it is a morphism of Hodge structures in $\Mod(R)_{/\tau}$. It is therefore strict for $F$  in $\Mod(R\otimes\CC)_{/\tau\otimes\CC}$. 
\item[2)]Proposition $(7.2.5)$ in \cite{Deltrois}~ implies $F_{d}=F_{d^*}=F_{rec}$ modulo $\tau\otimes\CC$ on $E_{2}(W)$. From \cite{Del} $(1.2.10)$, generalised in \ref{MHSabelian}, the category of mixed Hodge structures in $\Mod(R)_{/\tau}$ is abelian. Hence $(E_{2}^{p,q},\alpha_{2}^{p,q},F_{rec})$ is a  Hodge structure of weight $q$ in $\Mod(R)_{/\tau}$. The differential $d_{2}$ is a morphism of Hodge structures in $\Mod(R)_{/\tau}$ for it is real and compatible with $F_{rec}$ and its conjugate. This implies that $d_{2}$ is strict modulo $\tau\otimes\CC$.
But $d_{2}: E_{2}^{p,q}\to E_{2}^{p+2,q-1}$ is a morphism of Hodge structures of different weights. It must vanish.

An induction argument implies that $d_{r}=0$ if $r\geq2$.
\item[3)]  One concludes from section $(7.2)$ of \cite{Deltrois}: The following sequence is exact for any $1\leq r\leq +\infty$ and any $p\in\ZZ$: $$0\to E_{r}(F^pK,W)\to E_{r}(K,W)\to E_{r}(K/F^pK,W)\to 0\,,$$ and the spectral sequence $E(K,F)$ degenerates at $E_{1}$. 
\end{proof}
Note that the abelian category $\Mod(R)_{/\tau}$ admits inductive limits and that  the localisation functor $\Mod(R)\to \Mod(R)_{/\tau}$ commutes with inductive limits (\cite{Gab}~ prop. 9 p. 378). Hence, we consider the category $\M(R)_{/\tau}$ of sheaves with values in $\Mod(R)_{/\tau}$. 

\begin{definition}\label{definition of CHC}
 On a topological space $X$, a $(\tau',\tau)-$cohomological Hodge complex of weight $m$ $(\K_{R},(\K_{R\otimes \CC},F))$ consists in
\item[1)] A bounded 
below complex  $\K_{R}$ in $\M(R)_{/\tau'}$,
\item[2)] A bounded below filtered complex  $(\K_{R\otimes\CC},F)$ in $\M(R\otimes\CC)_{/\tau'\otimes\CC}$,
\item[3)] A pseudo-morphism of bounded below complexes $
\alpha:\K_{R}\dto \K_{R\otimes\CC}$ (first comparison morphism) in $\M(R)_{/\tau'} $
such that $\alpha\otimes Id:\K_{R}\otimes\CC\dto \K_{R\otimes \CC}$ is a pseudo-isomorphism in $\M(R\otimes\CC)_{/\tau'\otimes\CC}$ and
\item[] $R\Gamma(\K)$ is a $(\tau',\tau)-$Hodge complex of weight $m$. 
\end{definition}
\begin{definition}\label{definition of CHMC}
A $(\tau',\tau)-$cohomological mixed Hodge complex of sheaves \\$\K=((\K_{\RR},W);\,(\K_{\CC},F,W),\beta) $ is given by complexes of sheaves of $N(G)_{/\tau'}-$modules and pseudomorphism $\beta: (\K_{\RR},W)\dto (\K_{\CC},W)$ in that category, such that $\beta\otimes 1_{\CC}$ is a quasi-isomorphism and such that for all $m\in \ZZ$, $Gr_{W}^{m}\K$ is a $(\tau',\tau)-$cohomological  Hodge complex of sheaves. 
\end{definition}

\begin{definition} When $\tau'=(0,\Mod(R))$ is the trivial torsion theory, $(\K_{R},(\K_{R\otimes \CC},F))$ (resp. $((\K_{\RR},W);\,(\K_{\CC},F,W),\beta)$) is called a $\tau-$cohomological Hodge complex 
of sheaves of weight $m$ (resp. a $\tau-$cohomological mixed Hodge complex of sheaves).
\end{definition}
In this article, we will mostly deal with $\tau'=(0,\Mod(R))$ the trivial torsion theory. Hence we will use complexes of sheaves of $N(G)-$modules.

\section{The Hodge to De Rham spectral sequence.}

\subsection{The local Hodge to De Rham spectral sequence}
From now on, $p:\tilde X\to X$ will be a covering of a complex connected manifold $X$ by a complex manifold $\tilde X$, not necessarily connected. Let $G$ be the group of covering transformations.
\begin{definition}\label{$G-$cover}
When $G$ acts transitively, one says that $p:\tilde{X}\to X$ is a $G-$covering.
\end{definition}
\subsubsection{Sobolev spaces}\label{Sobolev spaces}
(see \cite{Shu} p. 511).
 Fix a hermitian (later K\"{a}hler) metric on $X$ and take its pullback on $\tilde X$. 
 
 Let $\Lambda^{k}_{\RR}:=\Lambda^{k} T^{*}_{\RR}(X)$, $\Lambda^{k}:=\Lambda^{k}_{\RR}\otimes\CC$ and $\Lambda^{p,q}:=\Lambda^{p} T^{*(1,0)}_{\CC}(X)\otimes\Lambda^q T^{*(0,1)}_{\CC}(X)$ be respectively the bundles of real $k$-forms, of complex $k-$forms, and of bi-degree $(p,q)-$forms on $X$. 
Let $\A^{k}_{\RR}$, $\A^{k}$ and $\A^{p,q}$ be the associated sheaves of differential forms. 
Let $(E,h_{E})$ be a hermitian complex vector bundle on $X$ (or a riemannian real vector bundle). Let $(\tilde E,h_{\tilde E})$ be its pullback by $p$.

\begin{definition}
We define the following classical spaces of sections of $\tilde E$, following \cite{Shu} \para. 3:
\item[1)] The Hilbert space of square integrable sections   $(L^{2}(\tilde X, \tilde E),\,||.||_{L^{2}(\tilde X,\tilde E)})$.
\item[2)] The Frechet space of smooth sections with compact support 
$\C^{\infty}_{c}(\tilde X,\tilde E)$.
 \item[3)] The space of distributional sections $\mathcal{D}'(\tilde X,\tilde E)$.
\item[4)] The Sobolev space $S^{j}(\tilde X,\tilde E)$,  $j\in \NN$, is 
the space of $u\in \mathcal{D}'(\tilde X,\tilde E)$ such that $\forall\, 0\leq i\leq j,\,\,\tilde  \nabla^{i}u\in   L^{2}(\tilde X,\tilde E)$ (with $\nabla$ a connection on $E$). Then $||u||^{2}_{S^{j}(\tilde X,\tilde E)}:=\sum_{0\leq i\leq j}||\tilde \nabla^{i}u||^{2}_{L^{2}(\tilde X,\tilde E)}$.
\item[5)] Denote also $S^{+\infty}(\tilde X,\tilde E)=\cap_{j\in \NN}S^{j}(\tilde X,\tilde E)$.
\end{definition}
 
\subsubsection{Local Sobolev spaces uniform with respect to $p$.}

 Let $U$ be an open subset in $X$. Let $j\in \NN\cup\{+\infty\}$.
Define the local Sobolev space uniform with respect to $\p S^{j}_{loc}(U,E)$ as the set $\{\alpha\in \mathcal{D}'(p^{-1}(U),\tilde E),\,\forall \theta\in\C^
{\infty}_{c}(U),\,  (\theta\circ p) \alpha\in S^j(\tilde X,\tilde E)\}$.

\subsubsection{Sheaf of uniform Sobolev spaces}

Let $j\in \NN\cup\{+\infty\}$, let $\p\cS^{j}\A^{p,q}(E)$  be the sheaf on $X$ associated to the presheaf $U\to \p S^{j}_{loc}(U,\Lambda^{p,q}\otimes \tilde E)$ and set $\p\cS^{j}\A^{n}(E)=\oplus_{p+q=n}\p\cS^{j}\A^{p,q}(E)$.

Let $D:\C^{\infty}(X,E)\to \C^{\infty}(X,E')$ be a differential operator on $X$ acting
on hermitian vector bundles. 
Let $L^{2}(p^{-1}(U),\tilde E)\cap \Dom(D)$ be the space of square summable sections $\alpha$ such that 
$D\alpha$ is square summable.
Let  $\p(\E \cap \Dom(D))$ be  the sheaf generated by the presheaf $U\to [ L^{2}(p^{-1}(U),\tilde E)\cap \Dom(D)]$. 

Note that the above sheaves do not depend on the smooth metrics on $X$ and $E$
and that these are sheaves of $\N(G)-$modules.

Let $\p\cS^{j}\A^{k}_{\RR}$ be the sheaf of $\N(G,\RR)-$modules on $X$ associated to the presheaf of real forms 
$U\to \p S^{j}_{loc}(U,\Lambda^{k}_{\RR})$.

\begin{lemma}\label{resolutions standards}
\item[1)] The complex $(\p\Omega^{.},d)$ is a $\N(G)-$resolution of $\p\CC$.
\item[2)]Let $k\in\NN\cup\{+\infty\}$. Then $ (\p\cS^{k-.}\A^{.}_{\RR},d)$ is a $\N(G,\RR)-$resolution of $\p\RR$ and $ (\p\cS^{k-.}\A^{.},d)$ is a $\N(G)-$resolution of $\p\CC$.
\item[3)]  The sheaf $\p\cS^{j}\A^{p,q}\simeq \p\cS^j\A^{0}\otimes_{\A}\A^{p,q}$ is a 
fine $\N(G)-$sheaf.
\end{lemma}
\begin{lemma}\label{holomorphic to de rham} Let $k\in\NN$, $k\geq n$. Let $F$ be the Hodge filtration.
\item[1)] The morphism $(\p\Omega^{.},d,F)\to (\p\cS^{k-.}\A^{.},d,F)$ is a filtered quasi-isomorphism.
 \item[2)] The following complex of $\N(G)-$sheaves is exact: $$0\to\p\Omega^{p}\overset{i}{\to}\p\cS^k\A^{p,0}\overset{\dbar}{\to}
\p\cS^{k-1}\A^{p,1}\overset{\dbar}{\to}\ldots \to\p\cS^{k-n}\A^{p,n}\to 0\,.$$
\end{lemma} 
\begin{proof} One notes that $(2)$ is equivalent to $(1)$ for $Gr_{F}^{p}(\p\cS^{k-.}\A^{.},d,F)$ is the Dolbeault complex $(\p\cS^{k-1}\A^{p,1},\dbar)$. One will proves the more general fact that $$(\p\Omega^{.},d,F)\to (\p\cS^{k-.}\A^{.},d,F)\to (\p\A^{.}\cap \Dom(d),F,d)$$  are filtered 
quasi-isomorphism. Indeed the $E_{1}-$term of the spectral sequence of the third complex is  $\frac{\p\A^{p,q}
\cap \Dom(d)\cap\ker (\dbar)}{\dbar(\p\A^{p,q-1}\cap \Dom( d))}\,.$ The assertions are local over the base manifold $X$. One may work over an open chart $U$ of  $X$ which is biholomorphic to some ball $B(0,2)$ in $\CC^n
$. Let $x$ be the center of the chart.
 Then the covering $p^{-1}(U)\to U$ is isomorphic to $U\times p^{-1}(x)\to U$. The standard estimates for the resolution of the $\dbar-$operator on a strictly pseudoconvex domains will implies the lemma. 
  Let $f$ be a  $(p,q)-$form, $q\geq 1$, on the strictly pseudoconvex domain $U
\subset\subset \CC^n$ which belongs to $\Dom( d)\cap \Ker (\dbar)\subset \Ker(\dbar)$. Then $f\in \Dom
(\dbar)$. Let $N: L^{2}(U,\Lambda^{p,q})\to     L^{2}(U,\Lambda^{p,q})$ be the (bounded) $\dbar-$Neuman operator (see \cite{Tak} p. 280-282) so that $Ran(N)\subset Dom(\dbar\dbar^*+ \dbar^*\dbar )$ and $f=\dbar\dbar^*Nf+\dbar^*\dbar Nf$. The last term is vanishing 
for $\dbar N=N\dbar$ on $\Dom(\dbar)$. Then $f=\dbar(\dbar^*Nf)$ and $f\in H^s_{loc}(U)$ implies $(\dbar^*Nf)\in L^{2}(U,\Lambda^{p,q})$, $Nf\in H^{s+2}_{loc}(U)$ (usual Sobolev space in euclidian space). Hence if $U_{1}\subset\subset U$, the map $H^{s}(U)\cap \Ker(\dbar)_{L^{2}(U,\Lambda^{p,q})}\ni f\mapsto\dbar^*Nf\in  H^{s+1}(U_{1})$ is continuous.

 Let $B(0,1)\simeq U_{1}\subset\subset U$. 
Let $\alpha\in (\p\A^{p,q}
\cap \Dom( d)\cap\ker (\dbar))(U)$, $q\geq1$, which belongs to $\p\cS^{s}\A^{p,q}_{2}(U)$ for some $s\geq 0$. Then   $\beta=(\dbar^*N\alpha_{| U_{1}
\times\{y\}})_{y\in p^{-1}(x)} \in [L^2(p^{-1}(U_{1})\cap \Dom(\dbar)]\cap \p\cS^{s+1}\A^{p,q-1}(U_{1})$ is such that $\dbar\beta=\alpha_{|U_{1}}$.
This proves that   $$0\to\p\Omega^{p}\overset{i}{\to}\p(\A^{p,0}\cap \Dom( d))\overset
{\dbar}{\to}\p(\A^{p,1}\cap \Dom( d))\overset{\dbar}{\to}\ldots\to \p(\A^{p,n}\cap \Dom( d))\to 0$$ 
and
  $$0\to\p\Omega^{p}\overset{i}{\to}\p\cS^k\A^{p,0}\overset{\dbar}{\to}\p\cS^{k-1}\A^{p,
1}\overset{\dbar}{\to}\ldots \to\p\cS^{k-n}\A^{p,n}\to 0$$ are exacts. Hence $(\p\cS^{k-.}\A^{.},d,F)\to (\p\A^{.}\cap \Dom( d),F,d)$ is a filtered quasi-isomorphism.
\end{proof}   
\subsection{The Global Hodge to De Rham spectral sequence.}

\subsubsection{Global Sobolev spaces}\label{identity of Sobolev spaces in bounded geometry}
We assume now that $X$ is a hermitian manifold of bounded geometry (\cite{ShubinNantes} Appendix 1). An example is a covering of a compact hermitian manifold.

Let $E$ be a uniformly bounded hermitian vector bundle on $X$ (loc. cit.), then the  smooth sections with compact support are dense in the Sobolev space $S^{m}(\tilde X,\tilde E)$ ($m\in \NN$).
Moreover any $\C^{\infty}-$bounded uniformly elliptic differential operator is essentially self-adjoint (\cite{ShubinNantes} prop. 4.1).
Hence   $A:=(1+\Delta)^{\frac{1}{2}}: L^{2}(\tilde X,\oplus_{p}\Lambda^{p})\to L^{2}(\tilde X,\oplus_{p}\Lambda^{p})$ (defined through the functionnal calculus) defines the isomorphisms $A^{m}:S^{k}(\tilde X,\oplus_{p}\Lambda^{p})\to S^{k-m}(\tilde X,\oplus_{p}\Lambda^{p})$ (see e.g. Roe \cite{Roe} th.5.5).
From  the (non unitary) isomorphisms
\begin{eqnarray*}
(S^{j}(\tilde X,\oplus_{p}\Lambda^{p}),\, ||.||_{S^{j}(\tilde X,\oplus_{p}\Lambda^{p})}) &\overset{\simeq}{\to} (D(A^j),\, ||A^{j}.||_{L^{2}(\tilde X,\oplus_{p}\Lambda^{p})})
\end{eqnarray*}

define a new Sobolev norm by $||\alpha||_{j}=||A^j\alpha||_{L^2(\tilde X,\oplus_{p}\Lambda^{p})}$.

If $j\geq 1$, let 
\begin{eqnarray*}
d: (S^{j}(\tilde X,\Lambda^{p}), ||.||_{j}) &\to & (S^{j-1}(\tilde X,\Lambda^{p+1}), ||.||_{j-1})\\
 \dbar:=\dbar_{q}: (S^{j}(\tilde X,\Lambda^{p,q}), ||.||_{j}) &\to &  (S^{j-1}(\tilde X,\Lambda^{p,q+1}), ||.||_{j-1})
\end{eqnarray*}
be the bounded operators induced by the differentials $d$ and $\dbar$.    Then\\ $A^k: (S^{j}(\tilde X,\Lambda^{.}),||.||_{j}) \to (S^{j-k}(\tilde X,\Lambda^{.}), ||.||_{j-k})$ is an 
isometric isomorphism and
operators $\partial$, $\dbar$ and $A$ commute if the metric is K\"{a}hler. 

\subsubsection{}\label{the global Hodge to DeRham spectral sequence}  

The Hodge to De Rham spectral sequence is the spectral sequence for the filtered complex $(S,F)=(S^
{k-.}(\tilde X,\Lambda^{.}),d,F)$. Then
 $F^{p} S^{k-r}(\tilde X,\Lambda^{r})=\oplus_
{p'+q=r,\, p'\geq p}\;S^{k-(p'+q)}(\tilde X,\,\Lambda^{p',q})$.
Hence, $(E^{p,q}_{0}(S,F),d_{0})\simeq_{N(G)} (S^
{k-p-q}(\tilde X, \Lambda^{p,q}),\dbar)$ and the Kodaira's decomposition (\cite{Shu} p. 499) reads $$(E_{1}^{p,q}(S,F),d_{1})\simeq_{N(G,\CC)} \frac{ \Ker (\dbar_
{q})}{\Im( \dbar_{q-1})}\simeq_{N(G,\CC)}\frac{\H^{p,q}_{\dbar}(\tilde X)\oplus \Adh{ \Im(\dbar_{q-1})} }{\Im (\dbar_
{q-1})}\simeq_{N(G,\CC)} \H^{p,q}_{\dbar (2)}(\tilde X)\oplus \frac{\Adh{ \Im(\dbar_{q-1})} }{\Im (\dbar_
{q-1})}\,.$$

 \subsection{The degenerescence of the Hodge to De Rham spectral sequence.}
Let $\H_{d(2)}^{.}(\tilde X)$ be the space of square integrable $\Delta_{d}-$harmonic forms and $ \H^{p,q}_{\dbar(2)}(\tilde X)$ be the  space of square integrable $\Delta_{\dbar}-$ harmonic $(p,q)-$forms. Let $(\H^{.}_{d(2)}(\tilde X),F)$ be the complex with trivial differential, and Hodge filtration.
Assume the metric is K\"{a}hler. Then for all $r\geq 0$,
\begin{eqnarray*}
 E_{0}^{p,q}(\H^{.}_{d(2)}(\tilde X),F)= & \H^{p,q}_{ \dbar (2)}(\tilde X) &=E_{r}^{p,q}(\H^{.}_{d(2)}(\tilde X),F)\,.
\end{eqnarray*}

\begin{lemma}\label{degenerescence lemma}
Assume the metric is K\"{a}hler and of bounded geometry. Let $ \H^{p,q}_{\dbar(2)}(\tilde X)$ be the space of square integrable harmonic $(p,q)-$forms. Fix an integer $k$ greater than $\dim_{\RR}X$.  Let $H^{p,q}_{\dbar (2)}(\tilde X):=H^{q}((S^{k-.}(\tilde X,\Lambda^{p,.}),\dbar))$. Let $\tau$ be a torsion theory such that $\H^{p,q}_{\dbar (2)}(\tilde X)\to H^{p,q}_{\dbar (2)}(\tilde X)$ is an isomorphism in $\Mod(N(G,\CC))_{/\tau}$.  Then:
\item[i)]
The spectral sequence for $(S,F)=(S^
{k-.}(\tilde X,\Lambda^{.}),d,F)$  degenerates at $E_{1}$ in $\Mod(N(G,\CC))_{/\tau}$. The differential $d$ is strictly compatible with $F$ in $\Mod(N(G,\CC))_{/\tau}$.
\item[ii)] $E^
{p,q}_{1}(S,F)\simeq E^{p,q}_{\infty}(S,F)\simeq \H^{p,q}_{\dbar (2)}(\tilde X)$ in $\Mod(N(G,\CC))_{/\tau}$.
\end{lemma}
\begin{proof} 
\item[i)] Note that  $i:(\H^{.}_{d(2)}(\tilde X),F)\to  (S^{k-.}(\tilde X,\Lambda^{.}),d,F)$ is a morphism of filtered $N(G)-
$modules. 
Let $\tau$ be a torsion theory on $N(G)$ such that 
 $E_{1}(i)$ is an 
isomorphism in $\Mod(N(G))_{/\tau}$. Then $E_{r}(i)$ is an isomorphism for any $r\geq 1$. But the spectral sequence of $(\H^{.}_{d(2)}(\tilde X),F)$ degenerates so that the spectral sequence for $(S^
{k-.}(\tilde X,\Lambda^{.}),d,F)$ degenerates at $E_{1}$  in $\Mod(N(G))_{/\tau}$.
\item[ii)] The assertion follows for $(E_{1}^{p,q}(S,F),d_{1}) \simeq_{N(G,\CC)} \H^{p,q}_{\dbar (2)}(\tilde X)\oplus \frac{\Adh{ \Im(\dbar_{q-1})} }{\Im (\dbar_
{q-1})}\,.$
 \end{proof}
 \begin{definition}\label{Dolbeault torsion theory} The torsion theory  generated by $\C=\{\Cok(\H^{p,q}_{\dbar (2)}\to H^{p,q}_{\dbar (2)}),\,\, p,q\geq 0\,\}$ is the smallest torsion theory on $\Mod(N(G))$ which satisfies the above lemma. Let $\tau_{\dbar}$ (or $\tau_{\dbar,\tilde X}$) be the torsion theory it defines on $\Mod(N(G))$. 
\end{definition}

\begin{remark}
 The use of the unitary isometry $A^m$ between various Sobolev spaces proves that  the torsion theory $\tau_{\dbar}$ does not depend on the order $k$ in the Sobolev complex $(S^
{k-.}(\tilde X,\Lambda^{.}),d,F)$. Moreover, following Bruning Lesch \cite{BruLes} Th. 2.12 (smoothing of cohomology), it can be shown that the torsion theory $\tau_{\dbar}$  is indeed an invariant of the elliptic complex $(\C^{\infty}_{0}(\tilde X,\oplus_{p}\Lambda^{p}), d)$ and the complete K\"{a}hler metric. 
\end{remark}

\subsection{The case of a compact K\"{a}hler manifold $X$}
\begin{lemma} Let $X$ be a compact complex hermitian manifold and let  $E$ be a hermitian vector bundle. Then $\Gamma(X,\p\cS^{j}\A^{p,q}(E))=S^{j}(\tilde X,\Lambda^{p,q}\otimes \tilde E)$.
\end{lemma}

\begin{definition}[\cite{Shu} prop 1.13, \cite{Luc} chap. 1]
A bounded complex $(L^.,d^.)$ of Hilbert $G-$modules is $G-$Fredholm if $\oplus_{i}d_{i}:\oplus_{i}L_{i}\to \oplus_{i}L_{i}$ is $G-$Fredholm (see \ref{fredholm}).
\end{definition}
In the following lemma, the manifold $\tilde X$ is not necessarily connected.
\begin{lemma}[\Ati, see also \cite{Shu}, \cite{Luc}] Let $\tilde X\to \tilde X/G=X$ be a $G-$covering of a compact complex manifold (see \ref{$G-$cover}).
Then the complexes of Hilbert modules 
\begin{eqnarray*}
( S^{k-.}
(\tilde X,\,\Lambda^{p,.}),\dbar) &\text{ and }&( S^{k-.}(\tilde X,\Lambda^{.}),d)
\end{eqnarray*}
are $G-$Fredholm. 
\end{lemma}
\begin{corollary}\label{example of torsion modules} 
With the same hypothesis,

\item[1)] The $N(G)-$module 
$ \Adh{ \Im (\dbar_{ q}) }  / \Im( \dbar_{ q}) $  is  a $\dimG-$torsion module. 
\item[2)] $\forall x\in \Adh{\Im(\dbar_{q})}$, there exists $ r\in N(G)$ such that $ker (r)=0$ and $rx\in \Im(\dbar_{q})$. Therefore $$\U(G)\otimes_{N(G)} \frac{\Adh{\Im(\dbar_{q})}}{\Im(\dbar_{q})}=0\,.$$
\end{corollary}

\begin{proof}  \item[1)] Lemma 2.12 of \cite{Shu} implies that $\Op=\dbar_{q}: H_{1}= \Adh{\Im(\dbar^{*}_{q
+1})}\to H_{2}=\Adh{\Im( \dbar_{q})}$ is $G-$Fredholm, hence  lemma 1.15 of \cite{Shu} implies 
that $\Im (\Op)$ is $G-$dense in its closure: $\forall\epsilon >0$, there exists $L_{\epsilon}\subset \Im (\Op)$, a closed $G-$invariant subspace, such that $\dimG \Adh{\Im( \Op)}^{H_{2}}\ominus L_{\epsilon}\leq \epsilon$. In this example, we may take $L_{\epsilon}:= \Im(\dbar\circ 1_{[\eta_{\epsilon},+\infty[}(\Delta_{\dbar}))$ (functional calculus) with $\eta_{\epsilon}>0$ small enough. 

 This is equivalent to 
$\dimG \Im (\Op)=\dimG \Adh{\Im (\Op)}$ hence $\dimG \frac{\Adh{\Im  
(\Op)}}{\Im (\Op)}=0$.
\item[2)] Apply the lemma \ref{T-lemma}.
\end{proof}

Analogue proof holds for an elliptic complex of vector bundles: 
\begin{theorem} 
 Let $\tilde X\to \tilde X/G=X$ be a $G-$cover of a compact complex manifold.
Let $(E^{.},d^{.})$ be an elliptic complex ($d^{.}$ is a differential operator of order one) between vector bundles on $X$. Let $(S^{j-.}(\tilde X,\tilde E^{.}),d^{.})$ be the associated Sobolev complex on $p: \tilde X\to X$.
\item[1)]  The complex $(S^{j-.}(\tilde X,\tilde E^.),d)$ is $G-$Fredholm.
\item[2)] The module $\frac{\Adh{\Im( d)}}{\Im (d)}$ is of $G-$dimension zero and $\U(G)\otimes_{N(G)} \frac{\Adh{\Im (d)}}{\Im (d)}=0$. 
\end{theorem}

This implies that if $d:L^2(\tilde X,E)\to L^2(\tilde X, F)$ acts as an unbounded elliptic operator, then $\U(G)\otimes\frac{\Adh{\Im (d)}}{\Im( d)}=0$: One uses conjugation by $(1+d^*d)^{-1}$ and $(1+dd^{*})^{-1}$. One may also considers currents in negative Sobolev scale (e.g. Dirac mesure on a point).

\begin{corollary}[A $\partial\dbar-$lemma]\label{A ddbar-lemma} Let $p:\tilde X\to (X,\omega)$ be a $G-$covering of a compact K\"{a}hler manifold. Let $\alpha$ be a $d-$closed square integrable $(p,q)-$form on $\tilde X$ which is orthogonal to the harmonic forms. Then there exists a weak isomorphism $r\in N(G)$, there exists a square integrable form $\gamma$ on $\tilde X$ such that $r\alpha=\partial\dbar \gamma$. 
\end{corollary}
\begin{proof} Above corollary implies that there exists a weak isomorphism $r\in N(G)$  and $\gamma_{1}\in\Adh{\Im\dbar}$, $\gamma_{2}\in \Adh{\Im\dbar^{*}}$ such that $r\alpha=\partial (\gamma_{1}+\gamma_{2})$. Then there exists a weak isomorphism $r'$ such that   $r'\gamma_{1}=\dbar\gamma_{3}$ and $r'\gamma_{2}=\dbar^{*}\gamma_{4}$. Hence $r'r\gamma-\partial\dbar\gamma_{3}=\partial\dbar^{*}\gamma_{4}$ is $\dbar-$closed. But the metric is K\"{a}hler, hence $\partial^{*}\dbar=\dbar\partial^{*}$. Therefore $\partial\dbar^{*}\gamma_{4}\in \Ker(\dbar)\cap\Ker(\dbar)^{\perp}$ is vanishing.
\end{proof}

The combination of sheaf theory and torsion theory enables to recover a result of Dodziuk \cite{Dod} in the case of Hermitian manifold. Standard sheaf theory proves  that unreduced De Rham $l^{2}-$cohomology groups and unreduced simplicial $l^{2}-$cohomology groups are isomorphic. The use of torsion theory is needed to provide an isomorphism between the reduced cohomology groups. 
\begin{corollary}\label{Dodziuk} Let $p:\tilde X\to \tilde X/G=X$ be a $G-$covering. Then the combinatorial reduced $l^{2}-$cohomology and the analytical reduced
 $l^{2}-$ cohomology are isomorphic in $N(G)_{/\tau_{dim}}$.
\end{corollary}
\begin{proof} Let $f:K\to X$ be a $C^{1}-$triangulation of $X$ (\cite{Whi}): $K$ is a rectilinear complex in some euclidian space and $f$ is a $C^{1}-$map which is a homeomorphism. In what follows, we identify $K$ and $M$.

Let $K'$ be the first barycentric subdivision of $K$ (\cite{Lef}, \cite{SeiThr}). Let $v\in K_{0}$ be a vertex in $K$. Let $F_{v}=St(v,K')$  be the closed star of $v$ in $K'$.
The set $F_{v_{0}}\cap\ldots\cap F_{v_{k}}$ is non empty if and only if $[v_{0},\ldots,v_{k}]$ is a simplexe of $K$. 
 Hence the simplicial complex defined by the nerve of the closed covering $\M=\{F_{v},\, v \text{ a vertex of }K \}$ is identified with $K_{.}$ the simplicial complex of $K$. 
Let $S=[v_{0},\ldots,v_{k}]\in K_{k}$ be a $k-$dimensional simplex of $K$. Define $F_{S}:= F_{v_{0}}\cap\ldots\cap F_{v_{k}}$. Then $F_{S}=S^{*}$  is equal to the dual cell of $[v_{0},\ldots,v_{k}]$ (which is homeomorphic to a $(n-k)-$closed ball).

 If $\F$ is a sheaf on $X$, let $(\C^{.}(\M,\F),\delta)$ be the differential sheaf $U\to \Pi_{S\in K_{.}}\F(F_{S}\cap U)$ (\cite{God} II.5.2). 
  Then  $(\C^{.}(\M,\p\CC),\delta)$ is a resolution of $\p\CC$, 
  $(\C^{.}(\M,\p\CC),\delta)\to (\C^{.}(\M,\p\cS^{k-.}\A^{.}),\delta+d)$ and $(\p\cS^{k-.}\A^{.},d)\to (\C^{.}(\M,\p\cS^{k-.}\A^{.}),\delta+d)$ are quasi-isomorphisms of $\N(G)-$sheaves.  Each term in theses complexes are $\Gamma(X,.)-$acyclic for $F_{S}$ is contractible and $\p\CC$ is locally constant; and the sheaves  $\p\cS^{k-.}\A^{.}$ are fine. Hence (\cite{God} II.5.2):
  $$H^{k}(X,\p\CC)\simeq H^{k}_{\delta}(\Gamma(X,\C^{.}(\M,\p\CC)))\simeq H^{k}_{d(2)}(\tilde X)\,.$$

   Let $(\tilde K_{.},\delta)$, $\tilde K'$ be the pullback simplicial structures, and define $\tilde F_{\tilde S}:=\tilde F_{\tilde v_{0}}\cap\ldots\cap \tilde F_{\tilde v_{k}}$. Then $ \tilde F_{\tilde S}=\tilde S^{*}$ is equal to the dual cell of $\tilde S$. Hence 
   $$ H^{k}_{\delta}(\Gamma(X,\C^{.}(\M,\p\CC)))\simeq_{N(G)} H^{k}(Hom_{\CC[G]}(C_{.}(\tilde K_{.}),l^{2}(G)))\,$$
(we  use the right $\CC[G]-$module structure in $\tilde K_{.}$ and the $N(G)-\CC[G]-$bi-module structure in $l^{2}(G)$). 
 The reduction with respect to the torsion dimension $\tau_{dim}$ gives the result, for the modules $\frac{\Adh{\Im( d)}}{\Im (d)}$ defined by the 
 combinatorial differential (see \ref{definition of dim tor}) or analytical differential (see above) have $N(G)-$dimension zero. 
\end{proof}

Note that Dodziuk proves that the isomorphism is given by integration of harmonic forms on the cells of a pull back of a triangulation. 

\subsection{Pure Hodge structures}
\begin{theorem}\label{structurepure}Let $X$ be a connected Hermitian manifold and $p:\tilde X\to X$ be a covering, not necessarily connected, with covering transformations group $G$. 
Let $k\geq n=\dim_{\CC}X$.
\item[1)\, i)] The morphism $((\p\Omega^{.},d),\, F)\to ((\p\cS^{k-.}\A^{.},d),\,F)$ is a filtered quasi-isomorphism of $\N(G)-$sheaves such that $Gr_{F}\p\cS^{k-.}\A^{.}$ is $\Gamma-$acyclic.
\item[ii)] The morphism $(\p\Omega^{.},d) \to (\p\cS^{k-.}\A^{.},d)$ is a quasi-isomorphism 
of $\N(G)-$sheaves. It defines a real structure on $\HH(X,(\p\Omega^{.},d))$ compatible with the real structure given by the Godement resolution and the
pseudo-isomorphism (\ref{pseudo-morphism}) $\C^{.}(\p\CC)\dto (\p\Omega^{.},d) $.
\item[2)] Assume the manifold $X$ is compact. Then the
 Fr\"{o}licher spectral sequence $$H^q(X,\p\Omega^p)\abut \HH^{p+q}(X,(\p\Omega^{.},d))\simeq_{N(G,\CC)} H^{p+q}(X,\p\CC)\,,$$ is isomorphic to the Hodge to De Rham spectral sequence $H^{p,q}_{\dbar(2)}(\tilde X)\abut H^{p+q}_{d(2)}(\tilde X)$ (see \ref{the global Hodge to DeRham spectral sequence}). 
 \item[3)] Assume the manifold $X$ is a compact K\"{a}hler manifold.
 Let $\tau$ be a torsion theory on $\Mod(N(G))$ greater than $\tau_{\dbar}$. Then the Fr\"{o}licher spectral sequence
  degenerates in $\Mod(N(G))_{/\tau}$. Hence 
 $d$ is strict for $F$ in $\Mod(N(G))_{/\tau}$: $$Gr_{F}^{p}H^{p+q}(X,\p\CC)\simeq_{\Mod(N(G))_{/\tau}}H^{q}(X,\p\Omega^{p})\,.$$
\item[4)] Assume moreover that $\tau$ is real. Then the Hodge filtration on the hypercohomology\\  $F^{.}\HH(X,\p\Omega^{.}):=\im(\HH(X,F^.\p\Omega^{.})\to \HH(X,(\p\Omega^{.},d)))$  and its complex conjugate $\bar F$ are
$k-$opposed on $\HH^{k}(X,\p\Omega^{.})$ in $\Mod(N(G))_{/\tau}$.
 It defines a pure Hodge structure of 
weight $k$ on  $H^{k}(X,\p\RR)$ in $\Mod(N(G))_{/\tau}$
(see definition \ref{definition of MHS}).
\end{theorem}
\begin{proof}\item[1)\, i)]  was proved in lemma \ref{holomorphic to de 
rham}. 
\item[ii)] $(\p\Omega^{.},d) \to (\p\cS^{k-.}\A^{.},d)$ is a quasi-isomorphism for $\p\CC\to (\p\Omega^{.},d)$ and $\p\CC\to (\p\cS^{k-.}\A^{.},d)$ are resolutions.
But $(\A^{.}_{\CC},d)=(\A^{.}_{\RR},d)\otimes_{\ZZ}\CC$ and the following diagram is commutative: $$
\begin{diagram}\node{\p\RR}\arrow{e,t}{}\arrow{s,l}{}\node{(\p\Omega^{.},d)}\arrow{s,l}{}\\
\node{(\p\cS^{k-.}\A^{.}_{\RR},d)}\arrow{e,t}{}\node{(\p\cS^{k-.}\A^{.}_{\CC},d)}
\end{diagram}
$$
 \item[(2)] Sheaves $\p\cS^{k-.}\A^{.}_{\RR}$ and $\p\cS^{k-.}\A^{.}_{\CC}$ are $\Gamma(X,.)-$acyclic and $\Gamma(X,\p\cS^j\A^{.})= S^{j}(\tilde X,\Lambda^{.})$ for $X$ is compact. Hence $\HH^{.}(X,\p\Omega^{.})\simeq_{N(G)} H^{.}(S^{k-.}(\tilde X,\Lambda^{.}),d)$. 

Note that $\p\cS^j\A^{p,q}\tilde{\to}Gr_{F}^p(\p\cS^j\A^{p
+q})$ is $\Gamma(X,.)-$acyclic.
Hence the spectral sequence for $(\HH^{.}(X,\p\Omega^{.}), F)$ is (isomorphic) given 
by the spectral sequence of $N(G)-$modules 
$$H^{p+q}(S^{k-.}(\tilde X,\Lambda^{.}),d)\Leftarrow E_{1}^
{p,q}=H^{q}(S^{k-.-p}(\tilde X,\,\Lambda^{(p,.)},\dbar))\,.$$ 

\item[3)] From lemma \ref{degenerescence lemma}, the above spectral sequence degenerates in $\Mod(N(G))_{/\tau}$ and $\H^{(i,n-i)}_{d(2)}(\tilde X)\to E_{1}^{(i,n-i)}\simeq  E_{\infty}^{(i,n-i)}\,$ is an isomorphism. Hence $\oplus_{i\geq p}\H^{(i,n-i)}_{d(2)}(\tilde X)\to F^{p}H^{n}(X,\p\CC)$ is an isomorphism.
\item[4)] If moreover $\tau$ is real then 
$\Adh{ \oplus_{i\geq q}\H^{(i,n-i)}_{d(2)}(\tilde X) } =\oplus_{i\geq q} \H^{(n-i,i)}_{d(2)}(\tilde X)\to \bar F^{q}H^{n}(X,\p\CC)$ (\ref{conjugate filtration}) is also an isomorphism. From \cite{Del} (1.2.4), we conclude that in $\Mod(N(G))_{/\tau}$, the filtrations $F$ and $\bar F$ on $H^{n}(X,\p\CC)$ are $n-$opposed.
\end{proof}
\begin{example}\item[1)] \label{structurepure2}
Let $\alpha:\p\RR\to (\p\Omega^{.},d^{.})$ be the natural map. We have seen that $(\p\RR,(\p\Omega^{.},F),\alpha)$ is a 
CHC of sheaves of weight $0$ (\ref{definition of CHC}) modulo $\tau_{dim}$ when $G$ is Galois, or $\gamma^{*}\tau_{\dbar}$ (\ref{ring extension}) in general. In the following diagram, the maps $R\Gamma(i)$ (Hodge to De Rham) and $m$ ($\Gamma-$acyclic sheaves) are 
filtered quasi-isomorphisms.
$$
\begin{diagram}
\node{R\Gamma(\p\RR)}\arrow{e,t}{R\Gamma(\alpha)}\node{R\Gamma(\p\Omega^
{.},F)}\arrow{e,t}{R\Gamma(i)}\node{R\Gamma(\p \cS^{k-.}\A^{.},d,F)}\\
\node{}\node{}\node{\Gamma(\p \cS^{k-.}\A^{.},d,F)}\arrow{n,l}{m}
\end{diagram}
$$
Therefore if $r\geq 1$, $$  E_{r}(m,F)^
{-1}\circ E_{r}(R\Gamma(i),F): E_{r}(R\Gamma(\p\Omega^{.},F))\to E_{r}(\Gamma(\cS^{k-.}\A^{.},d,F))$$ defines an 
isomorphism of spectral sequences. The second degenerates modulo $\tau_{dim}$ when $G$ is Galois ($E_{1}\simeq E_{\infty}$ in $\Mod(N(G))_{/\tau_{dim}}$). 
\item[1bis)] Let $\tau$ be a torsion theory such that $(\p\RR,(\p\Omega^{.},F),\alpha)$ is a CHC of sheaves of weight $0$ modulo $\tau$.
Let $k,m\in \ZZ$. One checks that $((2i\pi)^{k}\ZZ\otimes\p\RR[m],(\p\Omega^{.}[m],F[k+m]),\alpha[m]. (2i\pi)^{k} )$, where $F[k]^{p}:=F^{p+k}$ is the filtration shifted $k-$steps to the left, is a CHC of sheaves of weight $m-2k$ modulo $\tau$. Note the new conjugation induced by $\alpha (2i\pi)^{k}$ is $(-1)^{k}$ times the original one.
\item[2)] Assume that $\tilde X\to X$ is such that each connected component $
\tilde X_{i}$ of $\tilde X$ is compact and $G$ is transitive on fibers. Let $G_{i}$ be 
the stabiliser of $\tilde X_{i}$ and let $p_{i}$ be the restriction of $p$ to $X_{i}$.

 Then $d$ and  $\dbar$ have closed ranges. Indeed $$S^
{k}(\tilde X,\,\Lambda^{p})\ni\alpha\mapsto (g\mapsto g^{*}\alpha_{|\tilde X_{i}})\in l^{2}(G,  S^{k}(\tilde X_{i},\,\Lambda^{p}))^{G_{i}}$$ is a $G-$equivariant isometric isomorphism which commutes 
with $d$ and $\dbar$.
 But the functor of invariant with respect to $G_{i}$ is exact on $\QQ[G_{i}]-$modules, hence $H^{k}_{(2)}(\tilde X)\simeq l^{2}(G,H^{k}(\tilde X_{i}))^{G_{i}}$. The cohomology is reduced, and the Hodge structure is isomorphic to that of  $\tilde X_{i}$ twisted by that of $l^{2}(G,\CC)$.

If $\F$ is a sheaf for which $\p\F$ is defined then the $l^{2}-$cohomology is separated and 
$$H^{k}_{(2)}(X,\p\F) \simeq_{N(G)} l^{2}(G,H^{k}(X_{i},p_{i}^{*}\F))^{G_{i}}$$ and $$\dimG    l^{2}(G,H^{k}(X_{i},p_{i}^{*}\F))^{G_{i}} = \frac{\dim_{\CC}H^{k}(X_{i},p_{i}^{*}\F)}{|G_{i}|}\,.$$
\end{example}
\subsection{Addendum: Smoothing of cohomology} 
 It is well known that the cohomology of currents is isomorphic to the cohomology of smooth forms. Here we impose moreover that forms are square integrable in the fiber of $p$.

\begin{lemma}\label{smoothingofcohomology}
\item[1)] Let $U$ be some open set in $X$. Then $H^{.}(\Gamma(U,\p\cS^{\infty}\A^{.}),d)\to H^{.}(\Gamma(U,\p\cS^{k-.}\A^{.}),d)$ is an isomorphism. 
\item[2)] Assume  that $X$ is a K\"{a}hler manifold of bounded geometry. Then for any closed $l-$form $\alpha$ in $ \Dom[(1+\Delta)^{\frac{r}{2}}]$ on $\tilde X$, there exists $(\beta,\gamma)\in S^{r-1}(\tilde X,\Lambda^{l-1})\times S^{\infty}(\tilde X,\,\Lambda^{l})$ such that $\alpha=d\beta+\gamma$.
 \end{lemma}
 \begin{proof}\item[1)] Note that $(\p\cS^{\infty}\A^.,d)\to (\p\cS^{k-.}\A^.,d)$ is a quasi-isomorphism. The  result follows for theses sheaves are flabby.
 \item[2)] This follows from \cite{BruLes} th.2.12 and (\ref{identity of Sobolev spaces in bounded geometry}) above.
\end{proof}
We refer to \cite{BruLes} th.2.12 or th.3.5 for more general results on smoothing of cohomology in Hilbert complexes.
\subsection{An application: Mixed Hodge structure on the $l^{2}-$cohomology of a covering of a normal crossing divisor}
 Let $Y=Y_{1}\cup \ldots \cup Y_{p}$ be a compact normal crossing divisor in a smooth K\"{a}hler manifold.  Let $p:\tilde Y\to Y$ be a covering of $Y$. It induces covering maps $p:=p_{q}:\tilde Y_{q}\to Y_{q}$ with $Y_{q}=\sqcup_{i_{0}<\ldots<i_{q}}Y_{i_{0}}\cap\ldots \cap Y_{i_{q}}$. 

A classical generalisation of the Mayer-Vietoris argument relates the cohomology of $Y$ in  any coefficient systems to the cohomology of the subspace $Y_{q}$. This defines a mixed Hodge structure on $H^{.}(Y,\p\CC)$ in $\Mod(N(G))_{/\tau}$ for each $H^{.}(Y_{q},\p\CC)$ carries a Hodge structure in $\Mod(N(G))_{/\tau}$, if $\tau$ is big enough:

\subsubsection{} We follow the notations of \cite{Elz} section 3.5.
 Let $0\leq j\leq q$ and consider an increasing sequence of integers $1\leq i_{0}<\ldots<i_{q}\leq p$. The natural injections $\cap_{0\leq k\leq q}Y_{i_{k}}\to \cap_{0\leq k\leq q, \, k\not= q}Y_{i_{k}}$ define a closed immersion $\lambda_{j,q}: Y_{q}\to Y_{q-1}$. Let $\Pi:=\Pi_{q} :Y_{q}\to Y$ be the canonical projection.

The map $\lambda_{j,q}$ induces the morphism  
 $\lambda^{*(2)}_{j,q}:{\Pi}_{*}(p_{q-1})_{*(2)}\F_{Y_{q-1}}\to {\Pi}_{*}(p_{q})_{*(2)}\F_{Y_{q}}$ of sheaves over $Y$, with $\F_{Y_{q}}$ one of the sheaf $\CC_{Y_{q}}$, $\RR_{Y_{q}}$ or $\Omega^{.}_{Y_{q}}$.  One defines a boundary map $\delta^{*(2)}_{q-1}=\sum_{j\in [0,q]}(-1)^{j}\lambda^{*(2)}_{j,q}$.

Let $( \Pi_{*}\p\RR_{Y_{.}},\delta_{.}^{*(2)}, W_{.})$ be the complex of sheaves with increasing weight filtration \\ $W_{-q}( \Pi_{*}(p_{.})_{*(2)}\RR_{Y_{.}})=\sigma_{.\geq q}  \Pi_{*}(p_{.})_{*(2)}\RR_{Y_{.}}\,.$

Let $s(\p\Omega^{.}_{Y_{..}})$ be the single complex associated to the double complex of sheaves defined by 
$\p\Omega^{p}_{Y_{q}}=\Pi_{*}{p_{q}}_{*(2)}\Omega^{p}_{Y_{q}}$ and differential $d+\delta_{q}^{*(2)}$. 

The weight filtration on $s(\p\Omega^{.}_{Y_{..}})$ is the opposite of the filtration by the second index:  $W_{-q}s(\p\Omega^{.}_{Y_{..}})=s(\sigma_{..\geq q}\p\Omega^{.}_{Y_{..}})$.

The Hodge filtration corresponds to the filtration by the first index: \\ $F^{p}s(\p\Omega^{.}_{Y_{..}})=s(\sigma_{..\geq q}\p\Omega^{.}_{Y_{..}})$.
Let $i$ be the morphisms of sheaves induced by $i:\RR\to \O$.
\begin{theorem} Let $\tau$ be a real torsion theory on $\Mod(N(G))$ such that for any $q$, $\tau$ is greater than $\tau_{\dbar_{Y_{q}}}$.
\item[1)] Then $(R\Gamma(Y,\p\RR), (R\Gamma(Y, \Pi_{*}{p_{.}}_{*(2)}\RR_{Y_{.}}),W), (R\Gamma(Y, s(\p\Omega^{.}_{Y_{.}}),W,F),R\Gamma(i))$ is a  $\tau-$mixed \\Hodge complex.
\item[2)] The weight spectral sequence $H^{p}(Y_{q},{p_{q}}_{*(2)}\CC_{Y_{q}})\abut H^{p+q}(Y,\p\CC)$ with differential $d^{p,-q}_{1}=\delta_{q}^{*(2)}$ is a spectral sequence of Hodge structures in $\Mod(N(G))_{/\tau}$ and degenerates at $E_{2}$.
\item[] The Hodge spectral sequence $\HH^{q}(Y, (\Pi_{*}{p_{.}}_{*(2)}\Omega^{p}_{Y_{.}},\delta))\abut H^{p+q}(Y,\p\CC)$ degenerates at $E_{1}$. 
Hence $\HH^{n}(Y, F^{p}s(\p\Omega^{.}_{Y_{..}}))\to \HH^{n}(Y, s(\p\Omega^{.}_{Y_{..}}))$ is a monomorphism in $\Mod(N(G))_{/\tau}$.

\item[3)] The filtration $W[n]_{.}=W_{.-n}$ such that \\ $W[n]_{.} H^{n}(Y,\p\CC)=Im(\HH^{n}(Y,(\sigma_{..\geq n-.}\p\CC_{Y_{..}},\delta))\to H^{n}(Y,\p\CC))$ and \\ $F^{.}H^{n}(Y,\p\CC)\simeq \HH^{n}(Y,s(\sigma_{.\geq p}\p\Omega^{.}_{Y_{..}}))$ induces a mixed Hodge structure on $H^{n}(Y,\p\CC)$ in $\Mod(N(G))_{/\tau}$.
\end{theorem}
\begin{proof} This is a consequence of the theorem \ref{lemme des deux filtrations} and general theorems on cohomological mixed Hodge complexes (see \cite{Deltrois}, \cite{PetSte} Th.3.18 \rm I) and II)): 

Let $\K_{Y}:=((\p\RR,(\Pi_{*}{p_{.}}_{*(2)}\RR_{Y_{.}},W), (s(\p\Omega^{.}_{Y_{..}}),W,F), i)$.
The complex of sheaves\\ 
$(\Pi_{*}\p\CC_{Y_{.}},\delta_{.}^{*(2)})$
 is a resolution of $\p\CC$ for it is isomorphic to the usual Mayer-Vietoris resolution of $Y$ (\cite{Elz} 3.5.4) tensorised by the locally constant sheaf $\p\CC$.
Theorem \ref{structurepure} implies that  $(R\Gamma(Y, Gr^{W}_{-q}\K_{Y}),F)$ is a $\tau-$Hodge complex of weight $-q$ such that $H^{n}R\Gamma(Y, Gr^{W}_{-q}\K_{Y})\simeq H^{n-q}(Y_{q},\p\CC)$.   But $i\otimes 1_{\CC}$ is a quasi-isomorphism. Hence $\K_{Y}$ is a  $\tau-$cohomological Hodge complex of sheaves and $(R\Gamma(Y,\p\RR), (R\Gamma(Y, \Pi_{*}{p_{.}}_{*(2)}\RR_{Y_{.}}),W), (R\Gamma(Y, s(\p\Omega^{.}_{Y_{.}}),W,F),R\Gamma(i))$ is a $\tau-$mixed Hodge complex.
\end{proof}

\begin{example}
 Assume the irreducible components of $p^{-1}(Y)$ are compact. Then $\tau$ may be the trivial torsion theory.

\end{example}


\section{Mixed Hodge structures on the complement of a normal crossing divisor.}
We follow now the strategy of Deligne \cite{Del} to put a mixed Hodge structure on the $l^{2}-$cohomology groups associated to $p: \tilde X\setminus p^{-1}(D)\to X\setminus D$. 

In section \ref{Local setting}, we prove that the Leray spectral sequence which abuts to $H^{.}(X\setminus D,\p\CC)$ is isomorphic to the weight spectral sequence associated to the complex of square integrable forms with logarithmic singularity along the pullback of $D$.

In section \ref{global setting}, we use the existence of the Hodge structure mod$\tau$ on the $l^{2}-$cohomology of the compact manifolds $X$, $D_{i},\, D_{i}\cap D_{j},\ldots $ as it was developed in the previous section. We prove that the weight spectral sequence is a spectral sequence of Hodge structures mod$\tau$ which degenerates at $E_{2}$. 

In section \ref{interpretation}, we give a description of the first page in the weight spectral sequence in term of the homology associated to the $l^{2}-$Gysin morphisms of  the inclusions $\ldots\to D_{i}\cap D_{j}\to D_{i}\to X$. Section \ref{premiers exemples} will provide additional informations.

\subsection{Local setting}\label{Local setting}

We refer to \cite{Del} and \cite{PetSte} Chap. 4. for this section.
Let $X$ be a  complex manifold and $D\subset X$ be a normal crossing 
divisor. Let $j:U=X\setminus D\to X$ be the injection. Let $\dlog{.}{D}$ be the $\O_{X}-$subsheaf of $j_{*}\Omega^{.}_{U}$ of 
meromorphic forms with logarithmic poles on $D$: $\alpha\in \dlog{p}{D}$ if $\alpha
$ and $d\alpha:=j_{*}dj^{*}\alpha$ have pole of order at most one on $D$. 
Let $(z):V\to D(0,1)^n$ be a holomorphic chart centered at $x\in X$ such that   $D\cap V=\{z_{1}.\ldots . z_{k}=0\}$.  Then $\frac{dz_{1}}{z_{1}},\ldots,\, \frac{dz_{k}}{z_{k}},dz_{k+1},
\ldots,dz_{n}$ is a free basis of $\dlog{1}{D}_{x}$ and $\dlog{p}{D}=\wedge^{p}\dlog
{1}{D}$ is a locally free $\O_{X}-$sheaf.

Let $(\Lambda^{.}\{ \frac{dz_{1}}{z_{1}},\ldots, \frac{dz_{k}}{z_{k}}\})$ be the free antisymetric $\CC-$algebra built on $\frac{dz_{1}}{z_{1}},\ldots,\, \frac{dz_{k}}{z_{k}}$. The weight filtration $W$ is defined by 
$$W_{m}\dlog{p}{D}=\left\lbrace
\begin{array}{ll}
0 & \text{for } m<0\\
\Omega_{X}^{p-m}\wedge \dlog{m}{D} &  \text{for } 0\leq m\leq p\\
\dlog{p}{D} &  \text{for }  p\leq m\\
\end{array}\right. $$
 
 \subsubsection{Residues}
Let $D=\cup_{t\in T}D_{t}$ be the decomposition of $D$ into smooth irreducible 
components. If $I$ is a subset of $T$, let  $a_{I}:D_{I}=\cap_{t\in I}D_{t}\to X$ be the natural injection. Let $m\in \NN\setminus\{0\}$. Then set $D_{m}=\sqcup_{|I|=m}D_{I}$ and $a_{m}=\sqcup_{|I|=m} a_{I}$. Define $D_{\emptyset}=D_{0}=X$,  $a_{\emptyset}=a_{0}=Id$.
Let $ t\in T$. The residu map 
$$Res_{t}=Res_{D_{t}}:\dlog{.}{D}\to (a_{t})_{*}\Omega_{D_
{t}}^{.-1}(log (D_{t}\cap(\sum_{j\not= t}D_{j}))\, $$ is defined in a coordinate neighborhood $(U,(z))$:  $Res_{D_{t}}(\eta\wedge 
\frac{dz_{i}}{z_{i}}+\eta')=\eta_{|D_{t}}$ if $D_{t}\cap U=\{z_{i}=0\}$ and $\eta$ and $\eta'$ do not contain $\frac{dz_{i}}{z_{i}}$.
 Then  $Res_{D_{t}}$ is $\O_{X}-$linear and  commutes with $d$.

Let $I=(t_{1},\ldots,t_{k})$ be an ordered $k-$tuple.
Define $$Res_{I}=Res_{t_{1}}\circ \ldots \circ Res_{t_{k}}:\dlog{.}{D}\to (a_{I})_{*}
\Omega_{D_{I}}^{.-m}(log D_{I}\cap(\sum_{j\not\in I}D_{j}))\,$$ Then $Res_{I}$ is $\O_
{X}-$linear and commutes with $d$.
Let $\Lambda^{.} T$ be the free antisymetric $\CC-$algebra on $T$. We still denote by $\Lambda^{.}T$ the constant sheaf on $X$ it defines. Let $I\in T^k$, set $\Lambda^{I}t:=t_{1}\wedge\ldots\wedge t_{k}$. Let $I$ be a $k-$tuples of distincts elements and $J=\sigma(I)$ be a permutation of $I$.  Then $$\wedge^{I} t\otimes Res_{I}=\wedge^{J}t\otimes Res_{J}$$ 
for 
$Res_{\sigma(I)}=\epsilon(\sigma)Res_{I}$ with $\epsilon(\sigma)$ the signature of $\sigma$.
If  $I=(t_{1},\ldots,\, t_{l})\in T^{l}$, set $\{I\}=\{t_{1},\ldots,\, t_{l}\}\subset T$.  
Then if $\card\{I\}=l$, there is a well-defined map $$Res_{\{I\}}:=\wedge^{I}t\otimes Res_{I}: \dlog{.}{D}\to (a_{I})_{*} \Omega_{D_{I}}^{.-l}(Log D_{I}\cap(\sum_{j\not\in I}D_{j}))\otimes \Lambda^{l}T\,.$$ 
One sets $Res_{m}=\oplus_{|\{I\}|=m}Res_{\{I\}}$.

 Choice of an ordering on $T$   
   gives a trivialisation 
  $Res_{m}\simeq \oplus_{t_{1}<\ldots<t_{m}}Res_{(t_{1},\ldots,t_{m})}\,.$
From \cite{Del} 3.1.5.2, the map
\begin{eqnarray}\label{isomorphisme residu}
Res_{m}:(Gr^{W}_{m}\dlog{.}{D},d)\to (a_{m})_{*}(\Omega^{.-m}_{D_{m}},d)\otimes \Lambda^{m}T
\end{eqnarray}
 is an isomorphism of complexes.

Let $p: \tilde X\to X$ be a covering. If $f:Y\to X$ is a continuous map, let $f^{*}p:\tilde Y\to 
Y$ be the induced covering. 

\begin{lemma}\label{lemme d'isomorphisme residu}The residu morphism $Res_{m}$ induces an isomorphism
 $$\p Res_{m}:(Gr^{W}_{m}\p\dlog{.}{D},d)\to (a_{m})_{*}(\ap{m}\Omega^{.-
m}_{ D_{m}},d)\otimes \Lambda^{m}T\,.$$
\end{lemma}
\begin{proof} 
From \cite{CamDem}, the functor $\p$ is exact from the category of coherent sheaves of $
\O-$modules to the category of sheaves on $X$. Moreover the natural map $\p (a_{m})_{*}\Omega^{.- m}_{ D_{m}}\to (a_{m})_{*}\ap{m}\Omega^{.-m}_{ D_{m}}$ is an isomorphism (\cite{CamDem} prop. 2.9). Therefore (\ref{isomorphisme residu}) implies that $$0\to  (\p {W}_{m-1}\dlog{.}{D},d)\to  (\p{W}_{m}\dlog{.}{D},d) \to (a_{m})_{*}(\ap{m}\Omega^{.-m}_{ D_{m}},d)\otimes \Lambda^{m}T\to 0$$ is exact. A computation in  a local chart prove that the last maps define a map of differential complexes.
\end{proof}

From lemma \ref{squareholomorphicpoincare}, the following complexes are well 
defined (see \ref{filtration tau} for the $\canfilt-$filtration). 
\begin{proposition}\label{filteredquasiisomorphism}
\item[1)] The maps of filtered complexes 
\begin{eqnarray}(\p\dlog{.}{D},W,d)\overset{\alpha}{\leftarrow} (\p\dlog{.}{D},
\canfilt,d)\overset{\beta}{\to} (j_{*}(j^{*}p)_{*(2)}\Omega^{.}_{X\setminus D},\canfilt,d) 
\end{eqnarray}
 are filtered quasi-isomorphisms. 
 \item[2)] 
 This defines an isomorphism between the Leray spectral sequence for $j_{*}(j^{*}p)_{*(2)}\CC$ and the spectral sequence for the hypercohomology of the filtered complex $(\p\dlog{.}{D},W,d)$.
 \item[3)]One deduces the $N(G)-$isomorphisms:  
 \begin{eqnarray*}
 \lefteqn{ \HH^{.}(X,(\p\dlog{.}{D},d)) \simeq \HH^{.}(X, j_{*}(\jp\Omega_{X\setminus D}^{.},d) )  } \\
&& \simeq \HH^{.}(X\setminus D,(\jp\Omega_{X\setminus D}^{.},d))\simeq H^{.}(X\setminus D,\p\CC)\, .
 \end{eqnarray*}
\end{proposition}
\begin{proof}
\item[ ] The corresponding statement for a trivial covering map is proposition 3.1.8 of \cite{Del}. Let $(z):V\to D(0,1)^{n}$ be a chart in $X$ such that $D\cap V=\{z_{1}=\ldots =z_{k}=0\}$. Let $x$ be the center of the chart.

Define a residue map
$R_{m}:\Gamma(V\setminus D,\Omega^m)\cap \Ker (d)\to    \CC^{c(m,k)}$ through integration on  $m-$cycles $\{|z_{i_{1}}|=\epsilon_{1},\ldots, |z_{i_{m}}|=\epsilon_{m}\}$, $1\leq i_{1}<\ldots <i_{m}\leq k$. 

It is known (see Griffiths-Harris \cite{GriHar}) that  $d \Gamma(V\setminus D,\Omega^{m-1})=\Ker (R_{m})$. An explicit (continuous) left inverse $S_{m}$ is constructed in \cite{GriHar} p 451-452. The construction given there works for any closed holomorphic form, and is not restricted to polar singularities. It maps logarithmic forms to logarithmic forms. 

Fix a trivialisation of $G-$coverings $p^{-1}(V)\simeq  V\times p^{-1}(x)$. 
This defines\\ $R_{m(2)}: \Gamma(p^{-1}(V\setminus D),\p\Omega^m)\cap \Ker (d)\to l^{2}(p^{-1}(x))^{c(m,k)}$ and\\ $S_{m(2)}:\Ker (R_{m(2)})\to  \Gamma(p^{-1}(V\setminus D),\p\Omega^{m-1})$ a left inverse of $d$ which maps logarithmic forms to logarithmic forms.
Then $$l^{2}( p^{-1}(x),\CC)\otimes (\Lambda^m\{ \frac{dz_{1}}{z_{1}},\ldots \frac{dz_{k}}{z_{k}}\})\to H^{m}(\Gamma(V,\p\dlog{.}{D}),d)\to H^{m}(\Gamma(V\setminus D,\p\Omega^{.}),d))$$ are isomorphisms.

\item[1)]
Maps $\alpha$ and $\beta$ are filtered morphisms.
We proved that $\beta: (p_{2*}\dlog{.}{D},d)\to (j_{*}(j^{*}p)_{*(2)}\Omega^{.}_{U},d)$ is a quasi-isomorphism hence a filtered quasi-isomorphism for the canonical filtration.

Now $\tilde D_{m}\to D_{m}$ is a covering of a 
manifold. Hence, using  (\ref{lemme d'isomorphisme residu}), the sheaf $$\H^{i}(Gr^{W}_{m}(\p\dlog{.}
{D},d))\simeq \H^{i}(\ap{m}\Omega^{.-m}_{ D_{m}},d)$$ is vanishing  if $i\not=m$  and is 
isomorphic (as $\N(G)-$sheaf) to $\ap{m}\CC_{ D_{m}}$ if $i=m$ (Lemma \ref
{squareholomorphicpoincare}). 
This implies that  $\H^{m}((\p\dlog{.}
{D},d))\simeq a_{*}\ap{m}\CC_{ D_{m}}$ as $\N(G)-$sheaves. Hence $\alpha$ is a filtered quasi-isomorphism.  
\item[2)]

Note that $(\jp\Omega^{.}_{U},d)$ is a $\N(G)-$resolution of $\jp\CC$ 
 by $j_{*}-$acyclic sheaves: This follows from \cite{CamDem} th. 3.6 
  and $j$ is a Stein morphism. Then $R^{m}j_{*}[\jp\CC]$ is isomorphic to $\H^m(j_{*}(\jp\Omega^{.}_{U},\partial))$ as $\N(G)-$sheaves. This defines a pseudo-isomorphism \\$(Rj_{*}\jp\CC,\canfilt)=(j_{*}\C^{.}(\jp\CC,\canfilt))\dto (j_{*}\jp\Omega^{.}_{X\setminus D},\partial)$
 which composed with \\$E_{r}(\alpha)\circ E_{r}(\beta)^{-1}$ induces an isomorphism of spectral sequences.
 \item[3)] The quasi-isomorphism $(\p\dlog{.}{D},d)\to (j_{*}\jp\Omega_{X\setminus D}^{.},d)$ implies 
 \begin{eqnarray*}
 \HH^{.}(X,\p\dlog{.}{D})\simeq \HH^{.}(X, j_{*}\jp\Omega_{X\setminus D}^{.})\simeq \HH^{.}(X\setminus D,\jp\Omega_{X\setminus D}^{.})\\ \simeq H^{.}(X\setminus D,\jp\CC) \text{   as $N(G)-$modules. }
 \end{eqnarray*}

\end{proof}
\begin{definition}\label{normal crossing Dolbeault torsion theory}Let $\tau_{\dbar,\tilde D_{.}}$ be the Serre category generated by $\cup_{I\subset T}\tau_{\dbar,\tilde D_{I}}$ (see \ref{Dolbeault torsion theory}) 
\end{definition}
 
\begin{lemma}\label{definition de tilde beta}
\item[1)] 
 In the following diagram, the maps $f_{1}$ and $f_{2}$ are filered quasi-isomorphisms:
$$
\begin{diagram}
\node{(Rj_{*}\jp\RR,\canfilt)} \arrow{e,t}{i}\node{(Rj_{*}\jp\CC,\canfilt)} \arrow{e,t}{f_{1}}
\node{(Rj_{*}\jp\Omega^{.}_{X\setminus D},\canfilt,d)} \\
\node{(\p\dlog{.}{D},W,d)}
 \node{(\p\dlog{.}{D},\canfilt,d)}\arrow{w,t}{\alpha}\arrow{e,t}{\beta}\node{(j_{*}\jp\Omega^
{.}_{X\setminus D},\canfilt,d)}\arrow{n,t}{f_{2} }
\end{diagram}
$$ 
\item[2)]This defines the second comparison morphism $\tilde\beta :(Rj_{*}\jp\RR,\canfilt)\dto (\p\dlog
{.}{D},W,d)$  such that $\tilde\beta\otimes1_{\CC}:(Rj_{*}\jp\CC,\canfilt)\dto (\p\dlog{.}
{D},W,d)$ is a pseudo-isomorphism.
\item[3)] Let $\tau$ be a real torsion theory greater than $\gamma^{*}\tau_{\dbar,\tilde D_{.}}$ (\ref{ring extension}). Assume that $X$ is a compact K\"{a}hler manifold.

Then $\K:=[(Rj_{*}\jp\RR,\canfilt=W);(\p\dlog{.}{D},W,F),\beta]$ is a $\tau-$mixed Hodge 
complex of $\N(G,\RR)-$sheaves. Hence $R\Gamma(Gr^{W}_{m}\K)$ is a Hodge complex of weight $m$ in $\Mod(N(G,\CC))_{/\tau}$ (see \ref{Hodgecomplex}).
\end{lemma}
\begin{proof}
\item[1)]  According to \ref{Godement-resolution}, 
   $$Rj_{*}(\jp\RR)\otimes\CC\overset{i\otimes1_{\CC}}{\to} Rj_{*}(\jp\CC)$$ is a $
\N(G)-$isomorphism. This defines a filtered isomorphism for the canonical 
filtration.
From Prop. \ref{filteredquasiisomorphism} and exactness of the Godement resolution, $f_{1}$ is a quasi-isomorphism.
Also $f_{2}$ is quasi-isomorphism. 
Hence $f_{1},\,f_{2}$ are  filtered quasi-isomorphisms for the canonical filtration. 
\item[2)] Then $\tilde\beta$ is defined through $$E_{1}(\tilde\beta)=E_{1}(\alpha)E_{1}(f_{2}
\circ \beta)^{-1}E_{1}(f_{1})E_{1}(i) $$ so that $E_{1}^{-m,l}(\tilde\beta,{\canfilt}_{-})$   maps $\H^{m}(j_{*}
\jp\RR)$ to $ \H^{-m+l}((Gr^{W}_{m}\dlog{.}{D},d))$ if $2m=l$, all terms vanish otherwise. 
 Then $\tilde\beta\otimes 1_{\CC}$ is a pseudo-isomorphism ($E_{1}(\tilde\beta\otimes 1_{\CC})$ is an isomorphism).
 
 \item[3,a)] The following lemma is proved as in \cite{Del} prop. 3.1.9.
 \begin{lemma}The residue morphism $$\p Res_{m}: Gr^{W}_{m}\p\dlog{.}{D}\to a_
{m*}\ap{m}\Omega_{ D_{m}}^{.-m}\otimes\Lambda^{m}T$$ maps $R^m{j_{*}\p\RR}[-m]\overset{qi}{\to}Gr^{\canfilt}_
{m}Rj_{*}\p\RR$ to 
$ (2i\pi)^{-m} a_{m*}\ap{m}\RR \otimes\Lambda^{m}T[-m]$.
\end{lemma}

\item[3,b)] According to example \ref{structurepure2}, $\K_{D_{m}}$ defined by  
\begin{eqnarray}\label{definition du complexe residuel}
&& ((2i\pi)^{-m} \ZZ\otimes\ap{m}\RR_{ D_{m}}[-m], (\ap{m}\Omega_{ D_{m} }[-m],  F_{D_{m}}[-m],d),(2i\pi)^{-m}.\alpha_{ D_{m}})
 \end{eqnarray} 
 is a $\N(G,\RR)-$Hodge complex of sheaves of weight $m$ modulo $\tau$.  From lemma \ref{lemme d'isomorphisme residu}, we deduce that   $(R^m{j_{*}\jp\RR}[-m],  Gr^{W}_{m}\p\dlog{.}{D}, Gr^{\canfilt}_{m}(\tilde\beta))$ is a cohomological Hodge complex of sheaves of weight $m$ modulo $\tau$: 
 $R\Gamma(Gr^{W}_{m}\K)\simeq R\Gamma(\K_{ D_{m}}) $ is a Hodge complex of weight $m$. 
\end{proof}
\subsection{Global setting}\label{global setting}
One gets our first theorem on mixed Hodge structure on $l^{2}-$cohomology (compare with \cite{Deltrois} \para 8 and \cite{PetSte} th. 3.18):

\begin{theorem}\label{theoreme de structure mixte}
Let $p:\tilde X\to X$ be a covering of a compact K\"{a}hler manifold with covering transformations group $G$.
Let $\Gamma$ be  the global section functor over $X$, from the category of $\N(G,\RR)-$sheaves of modules (resp. $\N(G)-$sheaves of modules) to the 
category of  $N(G,\RR)-$modules (resp.  $N(G)-$ modules).
Let $R\Gamma$ be its derived functor realized through the Godement resolution. 

\item[] Assume that a torsion theory $\tau$ on $N(G,\RR)$ is chosen so that for each $p\in\ZZ$, $$R
\Gamma(Gr^{W}_{p}\K):=[(R\Gamma( R^pj_{*}\p\RR)),\, (R\Gamma (Gr^{W}_{p}\p\dlog{.}
{D},F)),\, R\Gamma(Gr^{W}_{p}\tilde\beta)]$$ is a Hodge complex in $N(G,\RR)_{/\tau}$ , then
          $$R\Gamma(\K):=[ (\Gamma Rj_{*}\jp\RR,\canfilt), (R\Gamma(
\p\dlog{.}{D},W,F), R\Gamma(\tilde\beta)]$$ is a mixed Hodge complex in $N(G,\RR)_{/\tau}$. Therefore:
\item[i)] The spectral sequence for  $(R\Gamma(\p\dlog{.}{D}),W,d)$, which $E_{1}^{p,q}-$term is \\$\HH^{q-p}(Gr^{W}_{p}\p\dlog{.}{D},d)\simeq \HH^{-2p+q}(D_{p},(\ap{p}\Omega^{.}_{D_{p}},d))$, 
degenerates at $E_{2}$ in $N(G,\CC)_{/\tau\otimes\CC}$.
\item[ii)] The differential $d_{1}$ on $E_{1}^{-p,q}(R\Gamma(\p\dlog{.}{D}),W)\simeq H^
{-2p+q}(D_{p},\ap{p}\CC)$ is real, and is a morphism of Hodge structures in $\Mod(N(G))_{/\tau}$.
\item[iii)]Through the isomorphism $$Gr^{\canfilt}_{p}H^{-p+q}(X\setminus D, \p\RR)\otimes\CC\simeq Gr^{W}
_{p}\HH^{-p+q}(X,(\p\dlog{.}{D},d))\,,$$ the Hodge filtration induces a Hodge structure (modulo
$\tau$). It is the same than the Hodge structure induced by the isomorphism $E_{2}(R
\Gamma(\p\dlog{.}{D}),W)\simeq E_{\infty}(R\Gamma(\p\dlog{.}{D}),W)$.
\item[iv)] The spectral sequence in $N(G,\CC)_{/\tau\otimes\CC}$ 
$$E_{1}^{p,q}(R\Gamma(\p\dlog{.}{D}),F)\simeq H^{q}(X,\p
\dlog{p}{D})\abut \HH^{p+q}(X,\p\dlog{.}{D})$$ degenerates at $E_{1}$. 
\item[v)] Define the weight filtration $W_{.}$ on $H^{k}(X\setminus D,\p\CC)$, to be the shifted filtration $k$ step to the left of the filtration induced by the Leray spectral sequence: $$W_{.}H^{k}(X\setminus D,\p\RR)=\im(\HH^{k}(X, {\canfilt}_{.-k}Rj_{*}\p\RR)\to H^{k}(X\setminus D,\p\RR))\,.$$ It is equal to  $\im( \HH^{k}(X, (W_{.-k}\p\dlog{..}{D},d))\to \HH^{k}(X\setminus D,\p\CC))$.

Let $F$ be the Hodge filtration on $H^{k}(X\setminus D,\p\CC)$, then $(H^{k}(X\setminus D,\p\RR), W, F)$ is a mixed Hodge structure in $\Mod(N(G))_{/\tau}$.

\end{theorem}
\begin{proof} Recall that $\C^{.}(.)$ is an exact functor to the category of flabby sheaves and that $R\Gamma^{.}(.):=\Gamma(X,\C^{.}(.))$ is an exact functor from the category of complex of $\N(G)-$sheaves on $X$ to the category of $N(G)-$complexes.
Consider the following diagram:
$$
\begin{diagram}
\node{(R\Gamma Rj_{*}\jp\RR,\canfilt)} \arrow{e,t}{R\Gamma (i)} \arrow{s,l,..}{R\Gamma
(\tilde\beta)}\node{(R\Gamma Rj_{*}\jp\CC,\canfilt)} \arrow{e,tb}{R\Gamma(f_{1})}{g_{1,
\CC}} \node{(R\Gamma Rj_{*}\jp\Omega^{.}_{U},\canfilt,d)} \\
\node{(R\Gamma \p\dlog{.}{D},W,d)}\arrow{e,tb}{g_{3}}{R\Gamma(\beta)}\node{(R
\Gamma(\p\dlog{.}{D}),\canfilt,d)}\arrow{e,t}{R\Gamma(\alpha)}\arrow{ne,t}{g_{2}}  \node{(R\Gamma (j_{*}\jp\Omega^{.},\canfilt,d)}\arrow{n,l}{R\Gamma(f_{2})} 
\end{diagram}
$$

Set $g_{1}=R\Gamma(f_{1})\circ R\Gamma(i)$. Then  $g_{1,\CC}=g_{1}\otimes 1_{\CC}$.
Refering to the proof of theorem \ref{lemme des deux filtrations}, defines $\beta_{2}=g_{2}\circ g_{3}$ and $\beta_{1}=g_{1}$.
This defines  pseudo-morphisms 
\begin{eqnarray}
\beta=R\Gamma(\tilde\beta):(R\Gamma Rj_{*}\jp\RR,\canfilt)\dto (R\Gamma \p\dlog{.}{D},W,d)\\
\beta_{\CC}: (R\Gamma Rj_{*}\jp\CC,\canfilt)\tilde\dto (R\Gamma \p\dlog{.}{D},W,d)\,.
\end{eqnarray}
Using the exactness of $R\Gamma$ and the previous lemma, one sees that $\beta_{\CC}$ is a pseudo-isomorphism. 

Moreover $$E_{1}^{p,q}(R\Gamma\K) \simeq H^{p+q}(R
\Gamma(Gr^{W}_{p}\K))\simeq H^{p+q}(R
\Gamma(\K_{D_{p}})) $$
 has a Hodge structure of weight $q$ modulo $\tau$ (see (\ref{definition du complexe residuel})). One concludes from  theorem \ref{lemme des deux filtrations} and section $7.2$ of \cite{Deltrois}. 

 %
\end{proof} 

\begin{example}\item[1)]\label{torsion mixte}Let $\tau_{\dbar,\tilde D_{.}}$ be the Serre category generated by $$\{\frac{\Adh{\Im  \dbar}}{\Im  \dbar},\, \dbar:\oplus_{p,q,I}S^{k-q,(p,q)}(p^{-1}( D_{I}))\to \oplus_{p,q,I}S^{k-q-1,(p,q+1)}(p^{-1}( D_{I}))\,\}$$ where $k$ is an integer big enough and $D_{\emptyset}=\tilde X$. Then $\gamma^{*}\tau_{\dbar,\tilde D_{.}}$ (\ref{ring extension}) is the smallest torsion theory which fulfills the above assumptions. 
\item[2)] Assume that $p:\tilde X\to X$ is a Galois covering, then the torsion theory associated to the dimension function or the $\U(G)-$torsion modules fulfills the above assumptions.
\end{example}

\subsection{Interpretation in equivariant cohomology}

We give three isomorphisms for the cohomology groups of the local system $\p\CC$ over $X\setminus D$ (see \ref{Link to singular cohomology}). For simplicity, we assume that $p:\tilde X\to X$ is a Galois covering.
\subsubsection{} 
One knows that $H^{k}(X\setminus D,\p\CC)$ is isomorphic to 
 $H^{k}(Hom_{\CC[G]}(C_{.}(p^{-1}( X\setminus D)),\, l^{2}(G)))$, the singular equivariant cohomology of $p^{-1}( X\setminus D) $ with coefficient in $l^{2}(G)$ (\cite{Eil}).

\subsubsection{}
$D$ admits a basis of neighborhood $V$ such that $V$ retracts onto $D$ and $X\setminus D$ retracts onto $X\setminus  V$. 
One may assume that $\bar V$ and $X\setminus V$ are manifolds with boundary.
Using the invariance of cohomology of a locally constant sheaf through homotopy (\cite{MebNar} cor.I.3.5), we get that $H^{k}(X\setminus D,\p\CC)\overset{i^{*}}{\to} H^{k}(X\setminus V,\p\CC)$ is an isomorphism.
Using a triangulation $T$ 
of the manifold with boundary $X\setminus V$, we get an isomorphism $H^{k}(X\setminus V,\p\CC)\simeq H^{k}(Hom_{\CC[G]}(\tilde T,l^{2}(G)))$ as in cor. \ref{Dodziuk}.
\subsubsection{} To obtain a direct combinatorial description of the group $H^{k}(X\setminus D,\p\CC)$, we use the Leray cover defined by the open stars of a triangulation  of $X$:  Let $f:K\to X$ be a $C^{1}-$triangulation of $X$ (\cite{Whi}) such that $D$ is a subcomplex. Let $v\in K_{0}$ be a vertex in $K$. Let $U_{v}=\Int{St(v,K)}$  be the open star of $v$ in $K$ (\Mun \S.2).
The open covering $\M=\{U_{v},\, v\in K_{0}\setminus D\}$ 
of $X\setminus D$ is such that any finite intersection of its elements is acyclic for any locally constant sheaf.
Let $K(- D)$ be the subcomplex of simplexes $\sigma$ of $K$ such that $\sigma$ does not intersect $D$.
Then the nerve of $\M$ is isomorphic as a simplicial set to $K(- D)$.  Let $\widetilde{ K(-D)}$ and $\tilde K$ be the pullback simplicial complexes. Using Leray's theorem (\cite{Bre} III.4.13 or \cite{God} II.5.2.4), we get as in the proof of cor. \ref{Dodziuk},  
$$H^{k}(X\setminus D,\p\CC)\simeq_{N(G)} \check H^{k}(\M,\p\CC)\simeq_{N(G)} H^{k}(\Hom_{\CC[G]}(\widetilde{ K(-D)},l^{2}(G)))\,.$$
This description as the homology of a complex of  $G-$Hilbert modules allows one to define a reduced cohomology which is isomorphic to the corresponding harmonic space (\cite{Luc} 1.1.4, \cite{Shu}).
The two last complexes have finite Von Neumann dimensions. Hence their reduced cohomology are isomorphic to their cohomology in $\Mod(N(G))_{/\tau_{\U(G)}}$ or $\Mod(N(G))_{/\tau_{dim}}$. 
\begin{corollary}\label{MHS on simplicial} There exists a mixed Hodge structure on $H^{k}(Hom_{\CC[G]}(C_{.}(p^{-1}(X\setminus D)),\, l^{2}(G)))$ \\
in $\Mod(N(G))_{/\gamma^{*}\tau_{\dbar,\,\tilde D_{.}}}$. 
In particular in $\Mod(N(G))_{/\tau_{dim}}$, there exists a mixed Hodge structure on 
 $\bar H^{k}_{(2)}(\widetilde {K(-D)}) $.
\end{corollary}

\subsection{Interpretation of $(E_{1},d_{1})$}\label{interpretation}
One aim of this section is formula (\ref{formulepourd1}) below which interprets $d_{1}$ through the Gysin morphisms. For this, we use a smooth logarithmic complex as in Griffiths-Schmid \cite{GriSch} p. 73.

\begin{definition} Let $\p\Slog{.}{D}$ be the subsheaf of $j_{*}[\jp\cS^{\infty}\A^{.}]$ such that a germ\\$\alpha$ belongs to $[\p\Slog{.}{D}]_{x}$ if
\begin{eqnarray}
h\alpha\in [\p\cS^{\infty}\A^{.}]_{x} &\text{ and } & h d\alpha\in [\p\cS^{\infty}\A^{.}]_{x}
\end{eqnarray}
where $h$ a defining equation for $D$ at $x$.
\end{definition}
From its definition, $(\p\Slog{.}{D},d)$ is a complex. Moreover 
\begin{lemma} $\p\Slog{.}{D}=\p\cS^{\infty}\A^{.}\otimes_{\O} \dlog{.}{D}$.
\end{lemma}
\begin{proof} One first proves the corresponding result in a coordinate chart $(V,(z))$ for a trivial one-sheeted covering. This will give estimates that carries over to the local situation of a trivialised covering $p^{-1}(V)\simeq V\times p^{-1}(x)$.
\item[1)]
Let us first recall a proof in a coordinate chart  $(V,(z_{1},\ldots, z_{n}))$ : The hypothesis implies that $h\alpha$ and $(h\alpha)\wedge \frac{dh}{h}$extends as smooth forms. Assume that $h=z_{1}\ldots z_{k}$ in $(V,(z))$ and that $h\alpha=\sum_{I}\beta_{I}dz^{I}$ where $I$ is a subset of $K=\{1,\ldots,k\}$ and the forms $\beta_{I}$ are smooth on $V$ which belongs to the exterior algebra generated by $dz_{k+1},\ldots, dz_{n}$. Then  $\gamma=(h\alpha)\wedge \frac{dh}{h}=\sum_{i,I}\frac{\beta_{I}}{z_{i}}dz^{I}\wedge dz^{i}$
 is a smooth form on $V\setminus \{z_{1}\ldots z_{k}=0\}$ such that any of its partial derivative $\frac{\partial^{|A|+|B|}}{\partial z^{A}\partial \bar z^{B}}\gamma$
admits a limit on $V\cap D$. Let $p\in (z_{1}=0)\setminus (z_{2}.\ldots .z_{k}=0)$. Then  $$ 0=\lim_{V\setminus D\ni p'\to p} z_{1} \frac{\partial^{|B|}}{ \partial\bar z^{B}}\gamma=
\lim_{V\setminus D\ni p'\to p} \sum_{i,I}\frac{ z_{1}}{z_{i}} \frac{\partial^{|B|}}{\partial \bar z^{B}}\beta_{I}dz^{I}\wedge dz^{i}
= \sum_{I\not\ni 1}  \frac{\partial^{|B|}}{ \partial\bar z^{B}}\beta_{I}(p)dz^{I}\wedge dz^{1}\,.$$
 Hence if $1\not\in I$ then $\frac{\partial^{|B|}}{\partial \bar z^{B}}\beta_{I}(p)=0$.
 By continuity, this holds for any $p\in \{z_{1}=0\}$. This implies that $\beta_{I}\in z_{1}\A^{.}(V)$ (\Mal, Schwartz \cite{Sch} th.2).
In the same way we see that if $i\not\in I$ then $\beta_{I}\in z_{i}\A^{.}(V)$. Hence $\beta_{I}\in z^{K-I}\A^{.}(V)$ by \cite{Sch} th.3 and 
  $\alpha=\sum_{I}\frac{\beta_{I}}{z^{K-I}}\frac{dz^{I}}{z^{I}}:=\sum_{I}\alpha_{I}\frac{dz^{I}}{z^{I}}$ belongs to $\A^{.}\otimes_{\O} \dlog{.}{D}(V)$.
  
 \item[2)]
But $z^{K-I}\A^{.}(V)$ is a closed ideal, therefore the surjective maps
$ \alpha_{I}\to z^{K-I}\alpha_{I}$ are open :
for all compact subset $K_{1}\subset V$, all $m_{1}\in \NN$, all $I\subset \{1,\ldots, k\}$, there exists compact subsets $ K_{2}\subset K_{3}\subset V$, integers $m_{2}\leq m_{3}$ and positive numbers $C_{1},\,C_{2}$ such that 
$||\alpha_{I}||_{\C^{m_{1}}(K_{1})}\leq C_{1} ||z^{K-I}\alpha_{I}||_{\C^{m_{2}}(K_{2})}\leq C_{2}||h\alpha||_{H^{m_{3}}(K_{3})}$. The last norm is the Sobolev norm of order $m_{3}$.

This implies that $\p\cS^{\infty}\dlog{.}{D}=\p\cS^{\infty}\A^{.}\otimes_{\O} \dlog{.}{D}$.
\end{proof}  

\begin{lemma}The sheaf $\p\cS^{\infty}\A^{0}$ is a flat $\O_{X}-$module.
\end{lemma}
\begin{proof}The statement for a trivial one sheeted covering is the flatness theorem of Malgrange \cite{Mal}:
Let $\I_{x}\subset \O_{x}$ be a finitely generated ideal. Let $V$ be a neighborhood of $x$ such that $\I(V)$ admits a finite presentation $$\O^{k}(V)\overset{r}{\to}\O^{p}(V)\overset{g}{\to} \I(V)\to 0$$  (the map $r$ gives the module of relations and $g:h\to\sum_{i=1}^{p}h_{i}g_{i}$ is defined using generators of $\I(V)$).

The flatness of $\C^{\infty}(V)$ on $\O(V)$ implies that $(\C^{\infty}(V))^k\overset{r}{\to} (\C^{\infty}(V))^p\to \I(V).\C^{\infty}(V)$ is exact. Hence $\Im (r)$ is closed and the above exact sequence splits. If $V$ is small enough, the pullback of this splitting provides a splitting of $\p\cS^{\infty}\A^{0}(V)^k\to \p\cS^{\infty}\A^{0}(V)^p\to \I(V).\p\cS^{\infty}\A^{0}(V)$. By Tougeron \cite{Tou} I.4, we conclude that $Tor_{1}^{\O_{x}}(\O_{x}/I_{x},\p\cS^{\infty}\A^{0}_{x})=0$.    
\end{proof}
\begin{corollary}\label{smoothing of the logarithmic complex} 
\item[1)]The morphism $(\p\dlog{.}{D},d)\overset{ }{\to} (\p\Slog{.}{D},d)$ is a quasi-isomorphism.  Moreover the morphism $$(\p\dlog{.}{D},d, W, F)\to (\p\Slog{.}{D},d, W, F)$$ defines a bi-filtered $\Gamma-$acyclic resolution.
\item[2)] Let $I$ be a $p-$tuple of distinct elements in $T$. Then there is a  map $$Res_{I} : \p\Slog{.}{D}\to \p\Slog{.-p}{(D_{I}\cap\sum_{j\not\in I}D_{j})}\,.$$
\item[3)]
The maps in the following diagram are filtered quasi-isomorphisms:

$$
\begin{diagram}
\node{(Gr^{W}_{l}\p\dlog{.}{D},d, F)} \arrow{e,t}{Res_{l}}\arrow{s,l}{}\node{ a_{l*}( \ap{l}\Omega^{.}_{ D_{l}},d,F)[-l]} \arrow{s,l}{}\\
\node{(Gr^{W}_{l}\p\Slog{.}{D}, d, F)} \arrow{e,t}{Res_{l}} \node{a_{l*}( \ap{l}\cS^{\infty}\A^{.}_{D_{l}},d, F)[-l] }
\end{diagram}
$$
\end{corollary}
\subsubsection{Gysin morphisms.}\label{Gysin morphisms}
 Let $I\subset T$ be such that $\sharp I=p$.

 Fix a (hermitian) metric on $D_{I}$ and choose an orthogonal splitting $T(X)_{|D_{I}}\simeq T_{D_{I}}\oplus N_{D_{I}/X}$.  For any $\epsilon$ small enough, the exponential map induces a diffeomorphism
$N_{D_{I}/X}(\epsilon)\overset{e_{I}}{\to} U_{I}(\epsilon)$ from a tubular neighborhood of $D_{I}$ in $N_{D_{I}/X}$ into a neighborhood of $D_{I}$ in $X$. 
 Define a retraction $r_{I}: U_{I}(\epsilon)\to D_{I}$ by compositing $e^{-1}_{I}$ with the bundle projection $N_{D_{I}/X}\to D_{I}$. Let $n=\dim X$. One choose $\epsilon_{1},\ldots,\epsilon_{n}$ such that the domains of the maps $r_{I}:U_{I}(\epsilon_{p})\to D_{I}$, $\sharp I=p$, satisfy the conditions $\cap_{ J\subset I, \,\sharp J=q}U_{J}(\epsilon_{q})\subset U_{I}(\epsilon_{p})$. 
Hence the intersection of the tubular neighborhoods of $D_{1}$ and $D_{2}$ is contained in the tubular neighborhood of $D_{1}\cap D_{2}$ (for simplicity assume that $X$ is of bounded geometry).
Then $r_{I}$ is a smooth submersion and $r_{I}^{*}$ maps local Sobolev space of order $s$ to local Sobolev space of order $s$ (for tensor product of separated variables commutes with the Fourier transform). 

Let $s_{t}$ be a  section of $[D_{t}]$ vanishing on $D_{t}$, $t\in T$. Let $h_{t}$ be a hermitian metric on $[D_{t}]$, such that $|s_{t}(x)|=1$ if $x\not \in U_{t}(\epsilon_{1}')$ with $\epsilon_{1}'<\epsilon_{1}$. Let  $\eta_{t}=\frac{1}{2i\pi}\partial log|s_{t}|^{2}$. The Poincar\'{e}-Lelong formula reads $d[\eta_{t}]=[D_{t}]-\omega_{t}$ where $\omega_{t}$ is the first Chern form of $([D_{t}],\, h_{t})$. Assume that $X$ is compact. Then these currents are supported in a compact neighborhood of $D_{t}$. For $I\in T^p$, define $\eta_{I}=\eta_{t_{1}}\wedge\ldots\wedge \eta_{t_{p}}$. If $t\not \in \{I\}$, we have
\begin{eqnarray}\label{support} supp( \omega_{t}\wedge \eta_{I})\subset\subset U_{\{t\}\cup \{I\}}(\epsilon_{|I|+1})\,.
\end{eqnarray}
We lift this construction by the covering map. Note that  a connected component of $p^{-1}(U_{I}(\epsilon_{p}))$ contains only one connected component of $p^{-1}(D_{I})$.

Let us introduce the following notations: for $I\in T^p$ such that $\Card\{I\}=p$, denotes  
$r_{I}, \eta_{I}, \ldots$ the lifts $p^{*}(r_{\{I\}}), p^{*}\eta_{I},\ldots$ to $p^{-1}(U_{\{I\}})$.

Let $[\alpha]\in H^{q-2p}( D_{I}, \ap{I}\cS^{k-.}\A^{.})$. From \ref{smoothingofcohomology}, we may assume that $\alpha$ belongs to \\$\Gamma(p^{-1} (D_{I}),\A^{q-2p})\cap_{n\in\NN}\Dom[(1+\Delta_{p^{-1}( D_{I})})^n]$ so that  $r_{I}^{*}(\alpha)\eta_{I}\in \Gamma(\cap_{t\in I} U_{t}(\epsilon'_{1}),\p\Slog{q-p}{D})$. Denote by the same symbol the extension of $r_{I}^{*}(\alpha)\eta_{I}$ by zero to $\tilde X$. Then $r_{I}^{*}(\alpha)\eta_{I}$ is an element of $\Gamma(\tilde X, \Slog{q-p}{\tilde D})$ such that $Res_{I}(r_{I}^{*}(\alpha)\eta_{I})=\alpha$. From corollary \ref{smoothing of the logarithmic complex},  one represents a class in $\HH^{-p+q}(X,\p (Gr^{W}_{p}\dlog{.}{ D},d))$ by a form
$\frac{1}{p!}\sum_{I}r_{I}^{*}(\alpha_{I})\eta_{I}\in \Gamma(\tilde X,\Slog{q}{D})$,
with the convention that $I\to \alpha_{I}$ is antisymmetrical in $I$. Hence $\alpha_{I}=\epsilon(t)\alpha_{(t,\Rm{t}{I})}$ where $\epsilon(t)$ is the signature of the permutation $I\to (t,\Rm{t}{I})$.
Then, 
\begin{eqnarray*}
 d(r_{I}^{*}(\alpha_{I})\eta_{I})=0+(-1)^{q-2p}r_{I}^*(\alpha_{I})\sum_{t\in I}\epsilon(t)\omega_{t} \eta_{\Rm{t}{I}} & =  & (-1)^{q-2p}\sum_{t\in I}r_{I}^*(\alpha_{(t,\Rm{t}{I})})\omega_{t}\eta_{\Rm{t}{I}}\\
d \sum_{I}r_{I}^{*}(\alpha_{I})\eta_{I} &= &(-1)^{q-2p}\sum_{J}\sum_{t}r_{t,J}^*(\alpha_{(t,J)})\omega_{t}\eta_{J}\\
Res_{\{J\}}d\sum_{I}r_{I}^{*}(\alpha_{I})\eta_{I}& = &(-1)^{q-2p}a_{J*}a_{J}^{*}(\sum_{t}\omega_{t}r_{t,J}^{*}(\alpha_{t,J}))\otimes \Lambda^{J}t
\end{eqnarray*}
Therefore
\begin{eqnarray} \label{formulepourd1} d_{1}[\sum_{I\in T^p}\alpha_{I}\otimes\Lambda^{I}t]&=&[\sum_{J\in T^{p-1}}(-1)^{q-2p}a_{J}^{*}(\sum_{t\in T}\omega_{t}r_{t,J}^{*}(\alpha_{t,J}))\otimes \Lambda^{J}t]\,.
\end{eqnarray}
\subsection{$l^{2}-$characteristics}
In this section, we state Atiyah's equalities between $l^{2}-$Euler characteristics and Euler characteristics of $X\setminus D$  under the hypothesis that the covering $p:\tilde X\to X$ is Galois.
Let $\tau_{dim}$ be the torsion theory such that $\T_{dim}=\{M\in\Mod(N(G))\,:\dim_{N(G)}M=0\}$ (see \ref{definition of dim tor}). 
Theorem \ref{structurepure} shows that $\tau_{dim}$ satisfies hypothesis of theorem \ref{theoreme de structure mixte}. Hence the weight spectral sequence degenerates at $E_{2}$ and the Hodge spectral sequence degenerates at $E_{1}$ in  $\Mod(N(G))_{/\tau_{\dim}}$. 
\begin{proposition} Let $\tilde X\to X$ be a $G-$cover of a compact K\"{a}hler manifold of dimension $n$ and let $D\to X$ be a normal crossing divisor. Then $ H^{n}(X\setminus D,\jp\CC)$ has finite Von Neumann dimension. Moreover if $F$ is the Hodge filtration given by the theorem  \ref{theoreme de structure mixte} then 

\begin{eqnarray}
Gr_{F}^{p} H^{i}(X\setminus D,\jp\CC) =  \bar H^{i-p}_{\dbar(2)}(\tilde X,\dlog{p}{\tilde D}) \text{ in } \Mod(N(G))_{/\tau_{\dim}} \\
\sum_{i=0}^{n}(-1)^{i}dim_{N(G)}Gr_{F^p}H^{i}(X\setminus D,\jp\CC)  = \sum_{i=0}^{n}(-1)^{i}dim_{\CC} Gr_{F^p}H^{i}(X\setminus D,\CC)\\
\chi_{(2)}(p^{-1}(X\setminus D))=\sum_{i}(-1)^{i}\chi_{(2)}(p^{-1}(D_{i}))= \sum_{i}(-1)^{i}\chi(p^{-1}(D_{i}))=\chi(X\setminus D)\,.
\end{eqnarray}
\end{proposition}
\begin{proof} Let $ \bar H^{k}_{\dbar (2)}(\tilde X,\, p^{*}(\dlog{p}{D}))$ be the Reduced Dolbeault cohomology group. Then inn $\Mod(N(G))_{/\tau_{dim}}$,  we have $H^{k}(X,\p\dlog{p}{D})\simeq \bar H^{k}_{\dbar (2)}(\tilde X,\, p^{*}(\dlog{p}{D}))$ (see \cite{CamDem} 3.4 and 3.14) and 
$$Gr_{F}^{p} H^{i}(X\setminus D,\jp\CC)\simeq  H^{i-p}(X,\p \dlog{p}{ D})\simeq  H^{i-p}_{\dbar(2)}(\tilde X,\, p^{*}\dlog{p}{D})$$   is finite $N(G)-$ dimensional. Using Atiyah's theorem (see \cite{Ati}, \cite{CamDem}), we get $\chi_{(2)}(\tilde X,p^{*}\dlog{p}{D})=\chi(X,\dlog{p}{D})$.
But Deligne's theorem (\cite{Del}) gives $Gr_{F}^{p} H^{i}(X\setminus D,\CC)\simeq  H^{i-p}( X,\dlog{p}{D})$. Hence  
\begin{eqnarray*}
\sum_{i=0}^{n}(-1)^{i}dim_{N(G)}Gr_{F}^{p}H^{i}(X\setminus D,\jp\CC)&=&\sum_{i=0}^{n}(-1)^{i}dim_{\CC} Gr_{F}^{p}H^{i}(X\setminus D,\CC)\\
\sum_{i=0}^{n}(-1)^{i}dim_{N(G)} H^{i}(X\setminus D,\jp\CC)&=&\sum_{i=0}^{n}(-1)^{i}dim_{\CC} H^{i}(X\setminus D,\CC)
\end{eqnarray*}
Using the weight spectral sequence, we get $\sum_{i,j}(-1)^{i+j}E_{1}^{i,j}=\sum_{i,j}(-1)^{i+j}E_{2}^{i+j}$ in the Grothendieck group of the category  $\Mod(N(G))$. From lemma \ref{lemme d'isomorphisme residu} and theorem \ref{structurepure}, one gets  $$E_{1}^{i,j}\simeq \HH^{j-2i}(D_{i},(\ap{i}\Omega^{.}_{D_{i}},d))\simeq H^{j-2i}(D_{i},\ap{i}\CC)\,.$$
This gives the equalities in the last line using additivity of dimension functions. 
\end{proof} 
Note in particular that $F^{n}H^{k}(X\setminus D,\,\jp\CC)=Gr_{F}^{n}H^{k}\simeq H^{k-n}(X,\p(K\otimes [D]))= \bar H^{k-n}_{\dbar(2)}(\tilde X, K_{\tilde X}\otimes [p^{-1}( D)])$ in $\Mod(N(G))_{/\tau_{dim}}$.

%
%
%
\subsection{Examples of $(\tau',\tau)-$cohomological mixed Hodge complex (CMHC) }
\subsubsection{} We have seen that $\K:=[(Rj_{*}\jp\RR, (Rj_{*}\jp\RR,\tau), (\p\dlog{.}{D},W,F), \tilde\beta]$ is a \\ $(0,\gamma^{*}\tau_{\dbar,\tilde D_{.}})-$CMHC (see lemma \ref{definition de tilde beta} for a definition of $\tilde\beta$). 
\subsubsection{} The following complex is functorial under morphism and its behavior under change of group may be studied through homological algebra. 

Recall that if $p:\tilde X\to X$ is a $G-$covering map  between locally compact spaces, then the functor $p_{!}p^*$ from the category of sheaves on $X$ to the category of $\ZZ[G]-$sheaves on $X$ is exact (Iversen \cite{Ive} p. 99 and p. 315) . Moreover, we have an isomorphism $ p_{!}p^*(F)\simeq (p_{!}\ZZ)\otimes F$ as $\ZZ[G]-$sheaves.

Then tensoring over $\CC[G]$ by $N(G)$ gives a functor from the category of $\RR-$mixed Hodge complexes of sheaves on $X$ to the category of $N(G)-$mixed Hodge complexes of sheaves modulo the dimension function: 
the functor $N:[\text{Sheaves over X}]\to[\N(G)-\text{Sheaves over X}]$ given by $\F\to \N(G)\otimes_{\CC[G]}p_{!}p^*\F$ is exact for a local model is an induced trivial module tensorised with $\CC[G]$.
 There are  natural maps
 $ \N(G)\otimes_{\CC[G]}p_{!}p^*\F\to l^{2}(G)\otimes_{\CC[G]}p_{!}p^*\F$
 defined by $  n\otimes f\mapsto n(\delta_{e})\otimes f$.
 In the case where $\F$ is holomorphic coherent or locally constant sheaf, this map takes values into $\p\F$. But is far from being surjective.   
However, if we reduce to the category of sheaves in $\Mod(N(G))_{/\tau_{\dim}}$, then we get from the usual CMHC $\K$ on $X$ a CMHC $mod(\tau_{\dim},\tau_{\dim})$ $\N(G)\otimes p_{!}p^{*}\K$.  
\section{Examples.}
\subsection{}\label{premiers exemples}
We give below the first terms of the unreduced $(E_{1},d_{1})$ complexes of the weight spectral sequence  (after a choice of ordering of the divisors). 

The opposite of the weight index is the column index. We recall that the shift in the weight filtration on the $l^{2}-$cohomology groups of $p^{-1}(X\setminus D)\to X\setminus D$ is given by
$Gr^{W}_{j}H^{i}(X\setminus D,\p\CC)\simeq Gr^{W}_{j-i}\HH^{i}(X, (\p\dlog{}{D},d))$.
For example, in suitable quotient category, homology at column $-i$ of the top line will compute $Gr^{W}_{6}H^{i}(X\setminus D,\p\CC)$.
 $$
\xymatrix{
    H^{0}_{(2)}(p^{-1} (D_{3}))  \ar@{.}[dr]  \ar[r]^{d_{1}} &   H^{2}_{(2)}(p^{-1} (D_{2}))   \ar@{.}[dr]  \ar[r]^{d_{1}} &  H^{4}_{(2)}(p^{-1}( D_{1}) ) \ar@{.}[dr]  \ar[r]^{d_{1}} &
 H^{6}_{(2)}(\tilde X)    & 0    \\
   0 &      H^{1}_{(2)}(p^{-1} (D_{2}))   \ar@{.}[dr]  \ar[r]^{d_{1}} & H^{3}_{(2)}(p^{-1}( D_{1}) )   \ar@{.}[dr]  \ar[r]^{d_{1}} &   H^{5}_{(2)}(\tilde X)  & 0 \\
  0&    H^{0}_{(2)}(p^{-1}( D_{2})) \ar@{.}[dr]  \ar[r]^{d_{1}} &   H^{2}_{(2)}(p^{-1} (D_{1}) ) \ar@{.}[dr]  \ar[r]^{d_{1}} &      H^{4}_{(2)}(\tilde X)  &  0 \\
  & 0 &   H^{1}_{(2)}(p^{-1} (D_{1}) ) \ar@{.}[dr]  \ar[r]^{d_{1}} &   H^{3}_{(2)}(\tilde X)    &   0 \\
   & 0&     H^{0}_{(2)}(p^{-1} (D_{1}) ) \ar@{.}[dr]  \ar[r]^{d_{1}} &   H^{2}_{(2)}(\tilde X) & 0 \\
   &  &    0    &  H^{1}_{(2)}(\tilde X) & 0  \\
     & & 0 &  H^{0}_{(2)}(\tilde X) &  0  \\
    Gr_{W_{-}}^{-3} &  Gr_{W_{-}}^{-2}& Gr_{W_{-}}^{-1} &  Gr_{W_{-}}^{0} &  
}
$$
We refer to the formula (\ref{formulepourd1}) for an interpretation of $d_{1}$ in term of the Gysin morphisms.
\subsubsection{}
The homology of  the $i-$th line at the column $Gr^{0}_{W_{-}}$ is $E_{2}^{0,i}(W_{-})$. But (\ref{filteredquasiisomorphism}) implies $$Gr_{i}^{W}H^{i}(X\setminus D,\p\CC)=\Im (H^{i}(X,\p\CC)\to H^{i}(X\setminus D,\p\CC))\,.$$ 
Hence under the hypothesis of theorem \ref{theoreme de structure mixte}, 
 the kernel of the restriction mapping is isomorphic in $\Mod(N(G))_{/\tau}$ to the image of the Gysin homomorphism.
\subsubsection{}
The $E_{1}-$page of the Hodge filtration has differential $d_{1}$ induced by the operator $\partial$. We recall that the $l^{2}-$cohomology groups $H^{.}(X,\p\E)$ of a holomorphic vector bundle are isomorphic to the unreduced $l^{2}-$Dolbeault cohomology groups $H^{.}_{\dbar(2)}(\tilde X,\, p^{*}(E))$  (see \cite{CamDem}).
In the case of a surface this gives:
$$
\xymatrix{
  { H^{2}(X,\p\O) }  \ar[r]^-{d_{1}} & {H^2(X,\p\dlog{1}{D})}  \ar[r]^-{d_{1}} &     { H^{2}(X,\p K\otimes [D]) } \\
    {H^{1}(X,\p\O)}  \ar[r]^-{d_{1}} & {H^1(X,\p\dlog{1}{D})}  \ar[r]^-{d_{1}} &   { H^{1}(X,\p K\otimes[D]) }  \\
    {H^{0}(X,\p\O) }  \ar[r]^-{d_{1}} & {H^0(X,\p\dlog{1}{D})}  \ar[r]^-{d_{1}}  &{ H^{0}(X,\p K\otimes [ D] ) }\\ 
     {Gr_{F}^{0}}   & {Gr_{F}^{1}} & {Gr_{F}^{2}}      
}$$

\subsubsection{Reduction with respect to the dimension torsion theory}
Let $p:\tilde X\to \tilde X/G=X$ be a $G-$covering of the complex compact K\"{a}hler manifold. Let us choose a K\"{a}hler metric on each divisor $D_{I}$ of $X$, and a hermitian metric on the vector bundle $\dlog{.}{D}\to X$. And consider the pullback structures on $\tilde X$ and $p^{-1}(D_{I})$.

 Let $q$ be one of the 
 projections from the $l^{2}-$cohomology group onto the respective harmonic spaces (which are induced by the orthogonal projections from the square integrable cocycle to the harmonic spaces).
Then the reduction of $q$ in $\Mod(N(G))_{/\tau_{dim}}$ or $ \Mod(N(G))_{/\tau_{\U(G)}}$ is isomorphic to the identity (see \ref{example of torsion modules}) and the torsion theory $\tau_{dim}$ satisfies hypothesis of theorem \ref{theoreme de structure mixte}. Therefore:
\begin{corollary}\label{Reduction with respect to the dimension torsion theory} In  $\Mod(N(G))_{/\tau_{dim}}$ or $ \Mod(N(G))_{/\tau_{\U(G)}}$, we have the following isomorphisms.
\item[i)] For the weight spectral sequence:
\begin{gather*}
 E^{-l,k+l}_{1}(W_{-})=\HH^{k}(X, Gr^{W}_{l}\p\dlog{.}{D})\simeq \H^{k-l}_{d(2)}(p^{-1}(D_{l}))   \\
[d_{1}]=[q\circ d_{1}\circ q]   \\
E^{-l,k+l}_{2}(W_{-})\simeq \text{Homology of }\left(  \H^{k-l-2}_{d(2)}(p^{-1}(D_{l+1}))\overset{q\circ d_{1}}{\to}\H^{k-l}_{d(2)}(p^{-1}(D_{l}))\overset{q\circ d_{1}}{\to}\H^{k+l+2}_{d(2)}(p^{-1}(D_{l-1}))  \right)
\end{gather*}
\item[ii)] For the Hodge spectral sequence: 
\begin{eqnarray*}
 E^{k,l}_{1}(F)\simeq \H^{l}_{\dbar(2)}(\tilde X,p^{*}\dlog{k}{D}) &\text{ and }& [d_{1}]=0\, .
 \end{eqnarray*}
\end{corollary}

\subsection{The dual $CW-$complex associated to $p:(\tilde X,p^{-1}(D))\to (X,D)$ and a combinatorial interpretation of the highest weight cohomology}\label{Simplicial structure}
%
The top line of $(E_{1}(W),d_{1})$ is interpreted through the dual complex associated to the divisors. A $k-$cell of this complex is a connected component of $p^{-1}(D_{k+1})$.
The $l^{2}-$homology of this dual complex is interesting: if the action of the infinite group $G$ is cocompact, the reduction in $\tau_{dim}$ or $\tau_{\U(G)}$ only keeps track of the complex associated to the compact connected components, for the other components have infinite isotropy subgoups, see lemma \ref{lemme un} .

On the other hand, 
the $l^{2}-$harmonic forms of maximal degree of a connected covering manifold is vanishing when the covering is infinite. This means that, when one neglects torsion, the only non trivial top dimensional $l^{2}-$cohomology spaces which enters in the top line of $(E_{1}(W),d_{1})$ are associated to compact connected componant. In this case they are multiple of the dual of the fundamental class.

This correspondance between cells with finite isotropy of the dual complex and dual of the fundamental class of compact connected componants is formalised in lemma \ref{lemme deux}.
In the case of a Galois cover, one then describes, modulo $\tau_{\U(G)}$ or $\tau_{dim}$, the groups $Gr^{W}_{2n}H^{.}(X\setminus D)$ as the relative $l^{2}-$homology of the dual complex relative to the complex of cells with infinite isotropy (theorem \ref{relation avec le complexe dual}).

Consideration of non Galois coverings leads us to distinguish between non compact connected componant and componant with infinite isotropy.
When working in general non Galois coverings, we may use the edge homomorphisms (\ref{edge homomorphisms}) in unreduced cohomology in order to derive informations from the simplicial $l^{2}-$homology of the dual simplicial complex.

\subsubsection{The $\Delta-$complex $\tilde K$.}\label{the delta complex}
  Let $T_{k}\ni \beta\to D_{\beta}$ be a parametrisation of the set of connected components $D_{\beta}$ of $\sqcup_{|I|=k}D_{I}$. We assume that $T_{1}=T$. Let $\P(T)$ be the set of subset of $T$. Let $n: T_{k}\to \P(T)$ be the map such that $D_{\beta}$ is a connected component of $\cap_{t\in n(\beta)}D_{t}$ and $\card n(\beta)=k$.
Let $\tilde T_{k}\ni \beta\to \tilde D_{\beta}$ be a parametrisation of the set of connected components of $p^{-1}(\sqcup_{|I|=k}D_{I})$. 
Then $G$ acts on $\tilde T_{k}$ and $p$ induces a projection  $p: \tilde T_{k} \to T_{k}$. 
We recall that $D_{\emptyset}=D_{0}=X$  so that $T_{0}$ is a set with a single element, and $\tilde T_{0}\ni \gamma\to \tilde X_{\gamma}$ parametrizes the set of connected components of $\tilde X$. In order to have uniform notations, the component $\tilde X_{\gamma}$ will also be denoted by $\tilde D_{\gamma}$.

We now define the $\Delta-$set $\tilde K$ (\RouSan):
Let $(T, \leq)$ be an order on $T$. Let $\F_{\inc}([k],\,T)$ be the set of injective increasing maps from $([k],\leq )
=(\{0,\ldots,\,k\},\leq)$ to $(T,\leq)$.  A  $k-$cell is an element $(\beta,o)\in  \tilde T_{k+1}\times \F_{\inc}([k],\,T)$ such that $\tilde D_{\beta}$ is a connected component of  $p^{-1}(D_{o(0)}\cap \ldots\cap D_{o(k)})$. 
In the following the map $o\in \F_{\inc}([k],\,T)$ will be identified with the $(k+1)-$uple $(o(0),\ldots,\,o(k))\in T^{k+1}$.

Let $f:[s]\to [k]$ be an injective increasing map. It induces the $f-$th face map $\tilde K(f):\tilde K_{k}\to \tilde K_{s}$:   
Let $(\beta,o)\in \tilde K_{k}$. Then  $\tilde K(f)(\beta,o)=(\gamma,o\circ f)$ where $\gamma\in \tilde T_{s}$ is the connected component  of  $p^{-1}(\cap_{t\in f\circ o([s])}D_{t})$ which contains $\tilde D_{\beta}$. 

Let $\Delta_{k}$ be the standard simplex in $\RR^{k}$. Then $f: [s]\to [k]$ induces the usual linear injection $f: \Delta_{s}\to \Delta_{k}$.
A geometric realization $|\tilde K|$  of $\tilde K$ is formed from  the disjoint union $\sqcup_{l\in \NN}\tilde K_{l}\times \Delta_{l} $ by identifying pairs $ ((\beta,o),f(x))$ and $(\tilde K(f)(\beta,o),x)$ (\cite{RouSan} \para.2).

Let $C_{k}(\tilde K)$ be the free $\ZZ-$module with basis the $k-$simplices $(\beta,o)$ in $\tilde K_{k}$.  Let $f_{m}:[k-1]\to [k]$ be the usual $m-$th face map and let $\partial_{m}(\beta,o):=\tilde K(f_{m})(\beta,o)$. The boundary map $\partial: C_{k}(\tilde K)\to C_{k-1}(\tilde K)$ is defined by $\partial (\beta,o)=\sum_{m=0}^{k}(-1)^{m}\partial_{m}(\beta,o)=\sum_{i\in o}\epsilon(i)\Rm{i}{(\beta,o)}$ where if $i=o(m)$, we set $\Rm{i}{(\beta,o)}=(\Rm{i}{\beta},\Rm{i}{o}):=\partial_{m}(\beta,o)$ and $\epsilon(i)=(-1)^{m}$ is the signature of the permutation $o\to (i,\Rm{i}{o})$.

 $G$ acts naturally on $\tilde K$, $|\tilde K|$ and $|\tilde K|$ is a $G-$CW complex (see \cite{Luc} 1.2.1).
Moreover any other choice of ordering $(T,\leq')$ on $T$ is given by an increasing permutation $\sigma: (T,\leq)\to (T,\leq')$ which defines $G-$isomorphic $\Delta-$complexes,  $G-$homeomorphic geometric realizations and $G-$homotopic chain complexes.

 Let $\epsilon:\tilde K_{0}\to \tilde K_{-1}=\tilde T_{0}$ be the augmentation map which assigns to $(\beta,t)$ the element $\gamma\in \tilde T_{0}$ such that $\tilde D_{\beta}\subset \tilde X_{\gamma}$. Let $(\tilde K_{\epsilon},\partial)$ be the augmented complex with respect to $\epsilon$.

  The group $G$ acts naturally on $(\tilde K_{\epsilon},\partial)$.

One notes that $\epsilon$ sends a non compact connected divisor (resp.  with infinite isotropy) to a non compact connected componant of $p^{-1}(X)$ (resp. with infinite isotropy).

\begin{definition} Recall that $G$ acts on $\tilde K$.
\item[1)] Let $\tilde K(\infty)$ be the $\Delta-$complex consisting of simplexes $(\beta,o)$ of $\tilde K$ whose isotropy subgroup $G_{(\beta,o)}$ under the action of $G$ is  infinite.
\item[2)] Let $\tilde K(\nc)$ be the $\Delta-$complex consisting of simplexes $(\beta,o)$ such that $\tilde D_{\beta}$ is non compact.
\item[3)] Let $\tilde K_{\epsilon}(\infty)$ and $\tilde K_{\epsilon}(\nc)$ be the correponding augmented complexes .
\end{definition}

Notations: Let $A$ be an abelian group, and $K'$ be a $\Delta-$subcomplex of a $\Delta-$complex $K$ with augmentation $\epsilon:K_{0}\to K_{-1}$. Then the induced augmented complex $K'_{\epsilon}$ is defined by $\epsilon:K'_{0}\to \epsilon(K'_{0})=K'_{-1}$. Then $(C_{.}(K_{\epsilon}, K_{\epsilon}', A),\partial)=([C_{.}(K_{\epsilon})/C_{.}(K'_{\epsilon})]\otimes A,\partial\otimes 1_{A})$ is the relative augmented chain complex with coefficients in $A$.

\subsubsection{}\label{modified boundary} One can compute the relative homology groups $H_{k}(K_{\epsilon},K'_{\epsilon})$ using a modified boundary map which anihilates boundaries liying in $K'_{\epsilon}$ (See \cite{FerPic} remark II.4.8). Such boundary maps will appear in computation of the relative homology of square integrable chains:

Let $C_{k}(K'_{\epsilon})$, respectively $C_{k}(K_{\epsilon}\setminus K'_{\epsilon})$, be the free abelian groups generated by the simplexes contained in $K'_{\epsilon}$, respectively not contained in $K'_{\epsilon}$. Let $a'_{.}: C_{.}(K'_{\epsilon})\to C_{.}(K_{\epsilon})$, $a_{.}: C_{.}(K_{\epsilon}\setminus K'_{\epsilon})\to C_{.}(K_{\epsilon})$, $b'_{.}: C_{.}(K_{\epsilon})\to C_{.}(K'_{\epsilon})$, $b_{.}:C_{.}(K_{\epsilon})\to C_{.}(K_{\epsilon}\setminus K'_{\epsilon})$ be the natural maps defining the splitting  $C_{.}(K_{\epsilon})= C_{k}(K'_{\epsilon})\oplus C_{k}(K_{\epsilon}\setminus K'_{\epsilon})$. The maps $\partial=b_{k-1}\partial a_{k}: C_{k}(K_{\epsilon}\setminus K'_{\epsilon})\to C_{k-1}(K_{\epsilon}\setminus K'_{\epsilon})$ define a complex $(C_{.}(K_{\epsilon}\setminus K'_{\epsilon}),\partial)$.
  The maps $i_{k}:C_{k}(K_{\epsilon}\setminus K'_{\epsilon})\to C_{k}(K_{\epsilon})\to C_{k}(K_{\epsilon})/C_{k}(K'_{\epsilon})$  define an isomorphism of complexes $i: (C_{.}(K_{\epsilon}\setminus K'_{\epsilon}),\partial)\to (C_{.}(K_{\epsilon},K'_{\epsilon}),\partial)$.

Hence the relative homology groups $H_{k}(K,K')$ are isomorphic to the homology of the complex $(C_{.}(K_{\epsilon}\setminus K'_{\epsilon}),\partial)$.
If $K'$ is a $G-\Delta-$subcomplex then $i$ is an isomorphism of $\ZZ[G]-$complexes.

\begin{example}
\label{one connected non compact}
\item[1)]Assume $\tilde X$ is connected non compact and $G$ is infinite then $\tilde K_{-1}=\tilde K_{-1}(\infty)=\tilde K_{-1}(\nc)=\{\tilde X\}$. Hence the relative chain complexes $(C_{.}(\tilde K_{\epsilon},\tilde K_{\epsilon}(\infty)),\partial)$ and $(C_{.}(\tilde K,\tilde K(\infty)),\partial)$ are equal. 
\item[2)] If $G$ is finite then $\tilde K_{\epsilon} (\infty)=\emptyset$.
\item[3)] If the action of $G$ on $p:\tilde X\to X$ is proper then $\tilde K(\infty)=\tilde K(\nc)$. Moreover, $\tilde K_{\epsilon} \setminus\tilde K_{\epsilon}(\nc)$ is the set of (closed) simplexes which are faces of finitely many simplexes. 
\end{example}

Let $(C_{.}^{sing}(|\tilde K|),\partial)$ be the complex of singular chains.

\begin{lemma}\label{lemme un}
\item[1)]
 Let $\tilde K(\infty)$ be the $\Delta-$complex consisting of simplexes $(\beta,o)$ whose isotropy subgroup $G_{(\beta,o)}$ under the action of $G$ is  infinite. Then the map of relative chain complexes $(C_{.}(\tilde K,\tilde K(\infty),\QQ),\partial)\to (C_{.}^{sing}(|\tilde K|,|\tilde K(\infty)|,\QQ),\partial)$ is a $\QQ[G]-$homotopy equivalence (see \cite{Luc} p.264).
\item[2)] Assume that $G$ is infinite. Then  the morphism
\begin{eqnarray*}
(l^{2}(G)\otimes_{\CC[G]} C_{.}(\tilde K_{\epsilon},\CC),\partial) &\to & (l^{2}(G) \otimes_{\CC[G]} C_{.}(\tilde K_{\epsilon},\tilde K_{\epsilon}(\infty),\CC),\partial) 
\end{eqnarray*}
is an isomorphism in 
   $ \Mod(N(G))_{/\tau_{\U(G)}}$ (see \cite{Luc} th.6.54).

\item[3)] Define a complex of square integrable chains by
 $$C^{(2)}_{.}(\tilde K_{\epsilon}\setminus\tilde K_{\epsilon}(\infty))=\{ \sum a_{(\beta,o)}(\beta,o) ,\, \text{ with } (\beta,o)\in \tilde K_{\epsilon, .}\setminus \tilde K_{\epsilon}(\infty) \,{/} \,\sum |a_{(\beta,o)}|^{2}|G_{\beta}|^{-1}<+\infty\}$$ with boundary 
  \begin{eqnarray}\label{equation of the modified boundary}
 \partial (\beta,o)=\begin{cases}
 \sum_{\,\partial_{m}(\beta,o)\not\in \tilde K(\infty) }(-1)^{m}\partial_{m}(\beta,o) & \text{ if }(\beta,o)\not\in \tilde K_{0}, \\
 \epsilon(\beta,o)=\gamma & \text{ if }  (\beta,o)\in \tilde K_{0}\setminus \tilde K_{0}(\infty), \,\gamma\not\in\tilde K_{-1}(\infty) ,\\
 0 & otherwise.
\end{cases}
\end{eqnarray}
Assume the number of $G-$orbits of $\tilde K_{\epsilon}\setminus\tilde K_{\epsilon}(\infty)$  is finite then 

 $(l^{2}(G) \otimes_{\CC[G]} C_{.}(\tilde K_{\epsilon},\tilde K_{\epsilon}(\infty),\CC),\partial)$ is $\CC[G]-$isomorphic to $( C^{(2)}_{.}(\tilde K_{\epsilon}\setminus\tilde K_{\epsilon}(\infty)),\partial)$.
\end{lemma}
\begin{proof} Indeed   $C_{k}(\tilde K)\simeq \oplus_{\sigma\in \Sigma_{k}}\ZZ[G/G_{\sigma}]$ where $\Sigma_{k}$ is a set of  representatives for the $G-$orbits of  $k-$simplexes (compare with Brown \cite{Bro} p. 68 example 5.5b). Let $\tilde K'$ be a sub-$(G,\Delta)-$complex of $\tilde K$.  Let  $C_{.}(\tilde K\setminus \tilde K')\simeq \oplus_{\sigma\in \Sigma_{.}\setminus \tilde K'}\ZZ[G/G_{\sigma}]$ be the free group on cells in $\tilde K\setminus \tilde K'$. Then $ C_{.}(\tilde K\setminus \tilde K')\to C_{.}(\tilde K,\tilde K')$ is a $\ZZ[G]-$isomorphism whose inverse  defines a $\ZZ[G]-$splitting of $0\to C_{.}(\tilde K')\to  C_{.}(\tilde K)\to C_{.}(\tilde K,\tilde K')\to 0$.  
Assume moreover that $\tilde K(\infty)\subset \tilde K'$.  Then $C_{.}(\tilde K,\tilde K')\otimes\QQ$ and $C_{.}^{sing}(|\tilde K|,|\tilde K'|)\otimes\QQ$ are projective $\QQ[G]-$modules (\cite{Bro} ex.4 p. 30). 
 If $\tilde K'$ is a sub-$\Delta-$complex then the map of relative chain complexes $(C_{.}(\tilde K,\tilde K'),\partial)\to (C_{.}^{sing}(|\tilde K|,|\tilde K'|),\partial)$ is a quasi-isomorphism (Hatcher \cite{Hat} th.2.27 p. 137). 
We conclude that the $\QQ[G]-$quasi-isomorphism of projective complexes $(C_{.}(\tilde K,\tilde K(\infty),\QQ),\partial)\to (C_{.}^{sing}(|\tilde K|,|\tilde K(\infty)|,\QQ),\partial)$ is a homotopy equivalence (\cite{Bro} p. 29 th.8.4). This prove $1)$.
%

 Let us prove $2)$. The proof of $1)$ shows that
 $0\to C_{.}(\tilde K_{\epsilon}(\infty))\to  C_{.}(\tilde K_{\epsilon})\to C_{.}(\tilde K_{\epsilon},\tilde K_{\epsilon}(\infty))\to 0$ is a split exact $\ZZ[G]-$sequence.
Hence the following sequence is exact: $$0\to l^{2}(G)\otimes_{\CC[G]}C_{.}(\tilde K_{\epsilon}(\infty),\CC)\to l^{2}(G)\otimes_{\CC[G]}  C_{.}(\tilde K_{\epsilon},\CC)\to l^{2}(G)\otimes_{\CC[G]} C_{.}(\tilde K_{\epsilon},\tilde K_{\epsilon}(\infty),\CC)\to 0\,.$$ From lemma \ref{infinite isotropy}, the first term is isomorphic to zero in $\Mod(N(G))_{/\tau_{\U(G)}}$. 

In  order to prove $3)$, we define a map from $l^{2}(G)\otimes_{\CC[G]} C_{.}(\tilde K_{\epsilon}\setminus \tilde K_{\epsilon}(\infty),\CC)$ to $C^{(2)}_{.}(\tilde K_{\epsilon}\setminus\tilde K_{\epsilon}(\infty))$ by $f\otimes (\beta,o)\to \sum_{g\in G}f(g^{-1})(g\beta,o)=\sum_{g\in A } [\sum_{h\in G_{\beta}}f (hg^{-1})](g\beta,o)$ with $A$ a set of representative for $G_{\beta}/G$. This map is injective for $f\otimes (\beta,o)=f.\rho(h)\otimes (\beta,o)$ if $h\in G_{\beta}$.
 The boundary map of $(C_{.}(\tilde K_{\epsilon}\setminus \tilde K_{\epsilon}(\infty),\CC),\partial)$ (\ref{modified boundary}) induces a boundary map $\partial$ on  $C^{(2)}_{.}(\tilde K_{\epsilon}\setminus\tilde K_{\epsilon}(\infty))$ given by $(\ref{equation of the modified boundary})$. 
It is bounded if $\sup_{\gamma\in \tilde K_{\epsilon}\setminus \tilde K_{\epsilon}(\infty)}\sum_{\tilde D_{\beta}\subset \tilde D_{\gamma}}\frac{|G_{\beta}|}{ |G_{\gamma}|}<+\infty$. 
This follows if the orbit type of $\tilde K_{\epsilon}\setminus \tilde K_{\epsilon}(\infty)$ is finite. In this case, the above map is moreover surjective.
\end{proof}

\begin{lemma} \label{lemme deux} Let $\beta\in \tilde T_{k}$ and let $n_{\beta}$ be the degree of the covering map $\tilde D_{\beta}\to D_{p(\beta)}=p(\tilde D_{\beta})$ (see \ref{the delta complex}).

Let $(C^{(2)}_{.}(\tilde K_{\epsilon}\setminus\tilde K_{\epsilon}(\nc)),\partial)$ be the complex of square integrable chains $$\{ \sum a_{(\beta,o)}(\beta,o) ,\, \text{ with } (\beta,o)\in \tilde K_{\epsilon, .}\setminus \tilde K_{\epsilon}(\nc) \text{ s.t. }\sum |a_{(\beta,o)}|^{2}(n_{\beta})^{-1}<+\infty\}$$
with bounded boundary map $\partial$ given in $(\ref{equation of the modified boundary})$ with $\tilde K_{\epsilon}(\infty)$ replaced by $\tilde K_{\epsilon}(\nc)$.
%
There exists an isomorphism of complex 
\begin{eqnarray*}
(C_{.}^{(2)}(\tilde K_{\epsilon}\setminus \tilde K_{\epsilon}(\nc)),\partial) &\to & (\H_{d(2)}^{2(n-(.+1))}(p^{-1}(D_{.+1})),q\circ d_{1}).
\end{eqnarray*}
which is compatible with the action of $G$.
\end{lemma}
\begin{proof} If $\beta\in \tilde T_{k}$, then $\H_{d(2)}^{2(n-k)}(\tilde D_{\beta})$ is vanishing if $\tilde D_{\beta}$ is non compact. If $\tilde D_{\beta}$ is compact, we denote by $C_{\beta}$ the unique harmonic form on $\tilde D_{\beta}$ such that $\int_{\tilde D_{\beta}}C_{\beta}=1$. Then $||C_{\beta}||^{2}=V_{\beta}^{-1}=(n_{\beta}V_{p(\beta)})^{-1}$ with $V_{\beta}$ (resp. $V_{p(\beta)}$) is the volume of $\tilde D_{\beta}$  with respect to a pullback metric  on $D_{p(\beta)}$ (resp. the volume of $D_{p(\beta)}$). 

 One extends $C_{\beta}$ by zero to $p^{-1}(D_{k})$. This extension, still denoted by $C_{\beta}$, satisfies the relations $C_{\beta}\rho(g):=g^{*}C_{\beta}=C_{g^{-1}\beta}$. 
Let $\Sigma_{k}$ be a set of representatives for the $G-$orbits of $k-$simplices. 
There is a hilbertian orthogonal decomposition $\H_{d(2)}^{2(n-k)}(p^{-1}(D_{k}))=\oplus_{\beta\in\tilde T_{k}\setminus \tilde T(\nc)}\H_{d(2)}^{2(n-k)}(\tilde D_{\beta})$ and an isometric isomorphism 
$$\H_{d(2)}^{2(n-k)}(p^{-1}(D_{k}))\simeq \oplus_{\beta\in\Sigma_{k}\setminus \tilde T(\nc)}l^{2}(G\diagup G_{\beta})  
\text{  given by  }  f\to \oplus_{\beta\in \Sigma_{k}}(gG_{k}\mapsto <f,C_{g\beta}||C_{g\beta}||^{-1}>).$$
Then, we define 
\begin{eqnarray*}
\chainr:  C^{(2)}_{.}(\tilde K_{\epsilon}\setminus \tilde K_{\epsilon}(\nc)) & \to &	  \H_{d(2)}^{2(n-(.+1))}(p^{-1}(D_{.+1}))\otimes\Lambda^{.}T\\
\sum_{(\beta,o)\in \tilde K_{\epsilon, .}\setminus \tilde K_{\epsilon}(\nc)} a_{(\beta,o)} (\beta,o) &\mapsto & \sum_{(\beta,o)}a_{(\beta,o)}C_{\beta}\otimes\Lambda^{o}t
\end{eqnarray*}
It converts the left action on chains to the right action on cohomology classes.

We prove now that it is a chain map. The use of exterior differential forms in computation of $d_{1}$ lead us to use antisymetric chains:
if $\sigma$ is a permutation of $[k]$ then $(\beta,o\circ\sigma)$ is by definition equal to $\epsilon(\sigma)(\beta,o)$ where $\epsilon(\sigma)$ is the signature of $\sigma$. 
Therefore a chain in  $l^{2}(G)\otimes_{\CC[G]} C_{.}(\tilde K_{\epsilon},\tilde K_{\epsilon}(\infty))$ is defined by $\sum_{(\beta,\{o\}) }a_{(\beta,o)}(\beta,o)$ where $\{o\}=\{o(0),\ldots,o(k)\}$ is identified with the orbit class under the permutation group of the injective map $o:[k]\to T^{k+1}$  and $o\to a_{(\beta,o)}$ is antisymetric in $o$. 
If $\beta\in \tilde T^k$ or $T_{k}$, let $ r_{\beta},  \ldots$ denotes the restrictions of  $p^{*}(r_{n p(\beta)}), \ldots$ to $U_{\beta}$, that connected component of $p^{-1}(U_{n p(\beta)}(\epsilon_{k}))$ which contains $\tilde D_{\beta}$ (notations of \ref{Gysin morphisms}). 
With these conventions, we associate to the square integrable singular chain $c=\sum_{(\beta,\{o\})}a_{(\beta,o)} (\beta,o)\in C^{(2)}_{k}(\tilde K_{\epsilon}\setminus \tilde K_{\epsilon}(\nc))$ the logarithmic form (see \ref{Gysin morphisms})
$$L(c)=\sum_{(\beta,\{o\})}a_{(\beta,o)} r^{*}_{\beta}(C_{\beta})\eta_{(\beta,o)}=\sum_{\{o\}}r^{*}_{o}(\sum_{\beta}a_{(\beta,o)} C_{\beta})\eta_{o}$$ 
whose residu is $$\chainr(c)=\sum_{\{o\}}(\sum_{(\beta,o)}a_{(\beta,o)}C_{\beta})\otimes \Lambda^{o}t \in \oplus_{\{o\}}\H^{2(n-(k+1))}_{d(2)}(p^{-1}(D_{o}))\otimes \Lambda^{k+1}T\,.$$
Here $\eta_{(\beta,o)}$ is the extension by $0$ to $\tilde X$ of $p^{*}(\eta_{(o(0),\ldots,\, o(k))})_{|U_{\beta}}$. Let $(\gamma,o')\in \tilde K_{k-1}$. Then, as in (\ref{Gysin morphisms}), 
\begin{eqnarray*}
Res_{\gamma}dL(c)& = &a_{\gamma}^{*}\left( \sum_{\tilde D_{\beta}\subset \tilde D_{\gamma}}(-1)^{2(n-k-1)}\sum_{i\in T\setminus o'([k-1])}a_{(\beta,(i,o'))} \omega_{i}r_{\beta}^{*}(C_{\beta}) \right) \otimes \Lambda^{o'}t  \,.\\
\end{eqnarray*}

Let $q_{\gamma}$ be the orthogonal projection $L^{2}(\tilde D_{\gamma},\Lambda^{2(n-k)})\to \H^{2(n-k)}_{d(2)}(\tilde D_{\gamma})$. It is non vanishing only if $\tilde D_{\gamma}$ is compact. In this case,  
$$\int_{\tilde D_{\gamma}}\omega_{i}r^{*}_{\beta}C_{\beta}=\int_{p^{-1}( D_{i})\cap \tilde D_{\gamma}}r^{*}_{\beta}C_{\beta}=\int_{\tilde D_{\beta}}r^{*}_{\beta}C_{\beta}=1$$ 
for $\supp (r^{*}_{\beta}C_{\beta})\cap p^{-1}( D_{i})\cap \tilde D_{\gamma}=\tilde D_{\gamma}$  (see \ref{Gysin morphisms}). Hence $q_{\gamma}(\omega_{i}r^{*}(C_{\beta}))=C_{\gamma}$ and 
\begin{eqnarray*}
q_{\gamma}\circ Res_{\gamma}d L(c)& = &(-1)^{2(n-k-1)}\sum_{\tilde D_{\beta}\subset \tilde D_{\gamma}}(\sum_{i\in T\setminus o'([k-1])}a_{(\beta,(i,o'))})C_{\gamma} \otimes \Lambda^{o'}t\\
&=&\left(\sum_{(\beta,\{o\})}a_{(\beta,o)}\epsilon((\beta,o):(\gamma,o'))\right)C_{\gamma} \otimes \Lambda^{o'}t
\end{eqnarray*}
where $\epsilon((\beta,o):(\gamma,o'))$ is vanishing if $\tilde D_{\beta}\not\subset\tilde D_{\gamma}$ and is equal to the signature of the permutation mapping $o$ to $(i,o')$ if  $\tilde D_{\beta}\subset \tilde D_{\gamma}$ and $o([k])\setminus o'([k-1])=\{i\}$.
 
Let $q=\oplus_{\gamma}q_{\gamma}$ be the orthogonal projection $L^{2}(p^{-1} (D_{k}),\Lambda^{2(n-k)})\to \H^{2(n-k)}_{d(2)}(p^{-1} (D_{k}))$. It will be identified with the induced projection $H^{2(n-k)}_{d(2)}(p^{-1}(D_{k}))\to \H^{2(n-k)}_{d(2)}(p^{-1} (D_{k}))$.
Then 
\begin{eqnarray*}
q\circ d_{1}(\chainr(c)) =  \sum_{\gamma,\{o'\}} q_{\gamma}\circ Res_{\gamma}d L(c)
 & = & \sum_{(\beta,\{o\})}\sum_{(\gamma,\{o'\})}a_{(\beta,o)}\epsilon((\beta,o):(\gamma,o'))C_{\gamma} \otimes \Lambda^{o'}t\\
 \sum_{(\beta,\{o\})}\sum_{i\in \{o\}}a_{(\beta,o)}\epsilon((\beta,o):\Rm{i}{(\beta,o)})C_{\Rm{i}{\beta}} \otimes \Lambda^{\Rm{i}{o}}t &=  &\chainr(\partial c) \,.
\end{eqnarray*}

One notes that the boundary maps are bounded for 
\begin{eqnarray*}
\sum_{\gamma\in \tilde T_{k}}(\sum_{\tilde D_{\beta}\subset \tilde D_{\gamma}}|a_{(\beta,o)}|)^{2}\frac{1}{n_{\gamma}}& \leq & \sum_{\gamma}(\sum_{\tilde D_{\beta}\subset \tilde D_{\gamma}}\frac{|a_{(\beta,o)}|^{2}}{n_{\beta}})(\sum_{\tilde D_{\beta}\subset \tilde D_{\gamma}}\frac{n_{\beta}}{n_{\gamma}})\\
&\leq & (k+1) \left(\sum_{\beta}|a_{(\beta,o)}|^{2}\frac{1}{n_{\beta}} \right) \max_{\gamma\in T_{k}}\card\{\beta\in T_{k+1}: D_{\beta}\subset D_{\gamma})\,.
\end{eqnarray*}
\end{proof}

The following theorem holds for any real torsion theory greater or equal to $\tau_{\U(G)}$ (e.g. $\tau_{dim}$). 
%
%
%
%
\begin{theorem} \label{relation avec le complexe dual} Let $\tilde X\to X$ be a $G-$covering of a compact $n-$dimensional K\"{a}hler manifold, let $D$ be a normal crossing divisor in $X$ and let $\tilde K$ be the associated dual complex.
Let $\tilde K_{\epsilon}(\infty)$  be the augmented subcomplex of cells of infinite isotropy subgroups. Theses cells are in one-to-one correspondance with the set of the non compact connected components of $\bigsqcup_{I\in \P(T)}p^{-1}(D_{I})$. 
 There exists  isomorphisms in $\Mod (N(G))$
\begin{eqnarray*}
   \chainr: (l^{2}(G)\otimes_{\CC[G]} C_{.}(\tilde K_{\epsilon},\tilde K_{\epsilon}(\infty)),\partial)  &\to & (\H_{d(2)}^{2(n-(.+1))}(p^{-1}(D_{.+1})),q\circ d_{1}) 
\end{eqnarray*}
which induce the isomorphisms in $\Mod(N(G))_{/\tau_{\U(G)}}$
\begin{eqnarray*}H_{k,(2)}(\tilde K_{\epsilon},\tilde K_{\epsilon}(\infty)) & \to & Gr_{2n}^{W}H^{2n-(k+1)}(X\setminus D,\p\CC)\hspace{2cm} \text{for  } k\geq -1.
\end{eqnarray*}

 Hence $Gr_{2n}^{W}H^{2n-k}(X\setminus D,\p\CC)$ is non vanishing in $\Mod(N(G))_{/\tau_{\U(G)}}$ implies that there exists compact connected components in $p^{-1}(D_{k})$ $(k\geq 0)$. 
 \end{theorem}
\begin{proof}
\item[] The hypothesis implies that $\tilde K_{\epsilon}(\infty)=\tilde K_{\epsilon}(n.c.)$, $|G_{\beta}|=n_{\beta}$ and the set of orbits of $\tilde K_{\epsilon}\setminus \tilde K_{\epsilon}(\infty)$ is finite. Then  $(l^{2}(G)\otimes_{\CC[G]} C_{.}(\tilde K_{\epsilon},\tilde K_{\epsilon}(\infty)),\partial)$ is identified with $(C_{.}^{(2)}(\tilde K_{\epsilon}\setminus\tilde K(n.c.)),\partial)$ and we conclude from lemma \ref{lemme un} and \ref{lemme deux}.

\item[]
Corollary \ref{Reduction with respect to the dimension torsion theory} implies that $H_{k-1,(2)}(\tilde K_{\epsilon},\tilde K_{\epsilon}(\infty))  \to  Gr_{2n}^{W}H^{2n-k}(X\setminus D,\p\CC)$ is an isomorphism in $\Mod(N(G))_{/\tau_{\U(G)}}$ if $G$ is infinite.
\end{proof}

\begin{remarks}
\item[1)] Let  $\dim_{\CC}\tilde D_{\beta}=n-k$ and let $G_{\beta}$ be the stabiliser of the irreducible component $\tilde D_{\beta}$. Then the operator$\Im (\dbar):L^{2}(\tilde D_{\beta},\Lambda^{2n-2k-1})\to L^{2}(\tilde D_{\beta},\Lambda^{2n-2k})$ has close range iff $G_{\beta}$ is non-amenable.
(see \Broo, \SalWoe). 
\item[2)] In $Mod(N(G))_{/\tau_{\U(G)}}$, we have $ H_{k,(2)}(\tilde K,\tilde K(\infty))\simeq \H_{k,(2)}(\tilde K,\tilde K(\infty))$ where $\H_{k,(2)}(\tilde K,\tilde K(\infty))$ is the (simplicial) harmonic space associated to the finitely generated hilbertian complex\\ $(l^{2}(G)\otimes_{\CC[G]} C_{.}(\tilde K,\tilde K(\infty)),\partial)$.
\end{remarks}

\subsection{Normal crossing divisor such that $X\setminus D$ is Stein.}
 
 Assume that $X\setminus D$ is Stein. Then the sheaves $\jp\Omega^{.}_{X\setminus D}$ are $\Gamma(X\setminus D,.)-$acyclic (\cite{CamDem}). From proposition \ref{filteredquasiisomorphism}, we deduce:
 \begin{lemma} Let $p:\tilde X\to X$ be a covering of a complex manifold with transformation group $G$. Assume that $X\setminus D$ is Stein. Then
 $H^{.}(X\setminus D,\p\CC)$ is $N(G)-$isomorphic to $H^{.}(\Gamma(X\setminus D,\jp\Omega^{.}_{(2)}),d)\,.$
 In particular, $k>\dim X$ implies $H^{k}(X\setminus D,\p\CC)=0\,.$
 \end{lemma}
 This last statement may be proved by other topological methods.
 
 One deduces that the dual complex associated to $p^{-1}(D)$ is acyclic for the reduced $l^{2}-$cohomology up to degree $\dim_{\CC} X-1$:
 \begin{proposition} Let $p:\tilde X\to X$ be a $G-$covering of a compact K\"{a}hler manifold. Let $D$ be a normal crossing divisor in $X$. Assume that $X\setminus D$ is Stein. Then the  relative homology group $H_{i(2)}(\tilde K\setminus \tilde K(\infty))$ of the dual complex  is vanishing in $\Mod(N(G))_{/\tau_{\U(G)}}$ when $i+1<n=\dim_{\CC}X$. 
\end{proposition}
\begin{proof} This is a consequence of theorem \ref{relation avec le complexe dual} and the above lemma.
\end{proof}

\begin{example} Assume $X$ is projective and $\dim_{\CC}X=n\geq 2$. Let $D$ be a divisor in $X$ and let $D_{v}$ be a generic hyperplane section such that $D\cup D_{v}$ is a normal crossing divisor. We compare the dual $CW-$complex $\tilde K$ associated to $p^{-1}(D)$ with $\tilde K'$ the one associated to $p^{-1}(D\cup D_{v})$. We recover directly the vanishing result of $H_{i}(l^{2}(G)\otimes_{\CC[G]}C_{.}(\tilde K',\tilde K'(\infty),\CC))$ for $i\leq n-2$.

Roughly speaking, $\tilde K'$ is obtained from the formal cone $w\star \tilde K$ by removing $n-$cells and replacing a formal $(n-1)-$dimensional cell $w\star (\beta,o)$, with $\tilde D_{\beta}$ a curve, by the collection of $(n-1)-$cells which parametrizes points in $\tilde D_{\beta}\cap p^{-1}(D_{v})$ and attached along there common boundary $w\star\partial(\beta,o) \cup (\beta,o)$: Let $\tilde T'_{k}\ni \gamma\to \tilde D_{\gamma}$ be a set which parametrizes the connected components of codimension $k$  of $p^{-1}(\sqcup_{\card I=k,I\subset T\cup\{v\}}D_{I})$. Let us assume that $T_{k}\subset T'_{k}$.
One sets $\tilde D_{w}=p^{-1}(D_{v})$. We choose an ordering on $T'=T\cup\{v\}$ such that $v$ is the biggest element of $T'$. Let $(\gamma,o')\in \tilde K'_{k-1}$ be such that $\tilde D_{\gamma}\subset p^{-1}(D_{v})$ and let $k\geq 2$. There exists a unique $(\beta, o)$ in $\tilde K_{k-2}$ such that $\tilde D_{\gamma}\subset \tilde D_{\beta}\cap p^{-1}(D_{v})$. Hence $(o'(0),\ldots,o'(k-1))=(o(0),\ldots,o(k-2),w)$ and 
 $\partial_{k-1}(\gamma, o')=(\beta,o)$.
 The Lefschetz's theorem implies that $\tilde D_{\gamma}=\tilde D_{\beta}\cap p^{-1}(D_{v})$ if $\dim_{\CC}\tilde D_{\beta}\geq 2$. When $\tilde D_{\beta}$ is a curve,  $p^{-1}(D_{v})\cap\tilde D_{\beta}$ is a discrete  set of cardinality $n_{\beta}\card( p(\tilde D_{\beta})\cap D_{v} )$ with $n_{\beta}$ the degree of the covering $\tilde D_{\beta}\to p(\tilde D_{\beta})$.
Let $n- 1\geq k\geq 2$, we can define a map $\star w: \tilde K_{k-2}\to \tilde K'_{k-1}$ by $\star w((\beta,o)):=(\gamma, o')$. Then $\tilde K'_{k-1}=\star w(\tilde K_{k-2})\cup \tilde K_{k-1}$. 
Let $(\beta,o)\in \tilde K_{n-2}$. Choose arbitrarly one connected component $\tilde D_{\gamma}$ in $\tilde D_{\beta}\cap p^{-1}(D_{v})$.  We define $\star w:\tilde K_{n-2}\to \tilde K'_{n-1}$ by $\star w((\beta,o))=(\gamma,(o,w))$. 

Therefore $\tilde K'_{\leq n-2}$ is embedded in the subcomplex $\star w(\tilde K_{\leq n-2})\cup \tilde K_{\leq n-1}\cup\{w\}$ of $\tilde K'$. This subcomplex is isomorphic to a cone with vertex $w$ over $\tilde K_{\leq n-2}$ (see \Mun p. 43). This implies that  $H_{i}(l^{2}(G)\otimes_{\CC[G]}\tilde K')=0$ when $i\leq n-2$. Hence $H_{i}(l^{2}(G)\otimes_{\CC[G]}C_{.}(\tilde K',\tilde K'(\infty),\CC))$ is vanishing for $i\leq n-2$.

\end{example}
 \subsection{Edge Homomorphisms}\label{edge homomorphisms} 
One states here the Edge homomorphisms of the non reduced weight spectral sequences:
 From $E_{1}^{l,k}(W_{-})\simeq H^{k+2l}_{(2)}(p^{-1}(D_{-l}))$, we get that $E_{r}^{l,k}(W_{-})$, $(r\geq 1)$, is non vanishing implies  $-n\leq l\leq 0$ and $-2l\leq k\leq 2n$. 
\begin{lemma}
\item[1)]There exists an epimorphism $E_{2}^{0,k}(W_{-})\to \Im(H^{k}(\tilde X,\p\CC)\to H^{k}(X\setminus D,\p\CC))$.
\item[2)]There exists a  monomorphism $Gr^{W}_{2n}H^{2n-l}(X\setminus D,\p\CC)\to E_{2}^{-l,2n}(W_{-})$.
\item[3)]
The following sequence is exact:
\begin{eqnarray*}
H^{2n-2}(X\setminus D,\p\CC) \to E_{2}^{-2,2n}(W_{-})\overset{d_{2}}{\to} E_{2}^{0,2n-1}(W_{-})\to\\ H^{2n-1}(X\setminus D,\p\CC) \to E_{2}^{-1,2n}(W_{-}) \to 0
\end{eqnarray*}

\end{lemma}
\begin{proof}
We refer to (\cite{God} II.4.5 p. 81) for these well known assertions. 
\end{proof}

\subsection{K\"{a}hler Hyperbolic manifolds} Let $p:\tilde X\to (X,\omega)$ be a $G-$cover. 
  Following \GroKahHyp, the form $\tilde\omega=p^{*}\omega$ is said to be $d(bounded)$ if $\tilde\omega=d \eta$ with $\eta$ a $1-$form which is bounded with respect to $\tilde\omega$. 
 We recall the following theorems on K\"{a}hler hyperbolic manifolds (\cite{GroKahHyp} 1.4.A and 2.5): 
  \begin{theorem}[Gromov]  \item[i)] Let $(\tilde X,\tilde \omega)$ be a complete K\"{a}hler manifold of real dimension $n=2m$ and assume that $\tilde\omega=d\eta$ where $\eta$ is a bounded $1-$form on $\tilde X$. Then there exists a strictly positive constant $\lambda_{0}\geq const_{n} ||\eta||^{-1}_{L\infty}$ such that every $l^{2}-$form $\psi$ of degree $p\not=m$ satisfies the inequality $$<\psi,\Delta \psi>\geq \lambda_{0}^{2}<\psi,\psi>\,.$$
  Furthemore, the above inequality is satisfied by the $l^{2}-$forms of degree $m$ which are orthogonal to the harmonic $m-$forms.
\item[ii)]  Let $(\tilde X,\tilde \omega)\to (X,\omega)$ be the universal cover of a K\"{a}hler manifold. Assume that $\tilde\omega$ is $d(bounded)$. Then the space $\H^{p,q}_{\dbar(2)}(\tilde X)$ of harmonic $l^{2}-$forms on $X$ of bi-degree $(p,q)$ is non vanishing if $p+q=m$.
\end{theorem}

A spectral gap for $\Delta$ acting on forms on any degree implies that $d$, $\delta$, $\dbar$ and $\dbar^{*}$ have closed ranges for the metric is  complete K\"{a}hler and $\Delta=\Delta_{\dbar}$.

Moreover, for any complex submanifold $i:Y\to X$, the pullback induced metric $\widetilde{ i^{*}\omega}$ of the pullback $G-$cover $i^{*}p:(\tilde Y,\widetilde{ i^{*}\omega})\to (Y,i^{*}\omega)$ is also $d(bounded)$. Hence the property of being $d(bounded)$, and then of spectral gap in any degree, in the $G-$cover is hereditary.
 
 From the Lefschetz theorem \cite{BarHulPet}, if $X$ is a projective manifold and $Y$ is a generic intersection of at most $n-2$ hyperplane sections then $\Pi_{1}(Y)\simeq \Pi_{1}(X)$. Hence the above theorem implies that $\H^{p,q}_{\dbar(2)}(\tilde Y)\not=0$ if $p+q=\dim Y$.

 \begin{theorem}
 Let $p:(\tilde X,\tilde\omega)\to (X,\omega)$ be a $G-$cover such that $\tilde\omega$ is $d(bounded)$.  Let $D$ be a normal crossing divisor in $X$.  Then any torsion theory $\tau$ fulfills asumption of theorem \ref{theoreme de structure mixte}. 
 Moreover the weight spectral sequence degenerates at $E_{1}$:
 $$Gr_{k+l}^{W}H^{k}(X\setminus D,\p\CC)\simeq \begin{cases} H^{k-l}_{d(2)}(p^{-1}( D_{l})) &\text{ if }k=n \\
 0 &\text {if } k\not= n
\end{cases}
  $$
  \end{theorem}
Therefore we may choose the trivial torsion theory $\tau=(\{0\},\Mod(N(G))$ (no non zero torsion modules) in the statement of theorem \ref{theoreme de structure mixte}. 
\begin{proof} Recall that $D_{I}:=\cap_{t\in I}D_{t}\to X$.
 Indeed the $l^{2}-$cohomology groups $H^{p,q}_{\dbar(2)}(p^{-1}( D_{I}))$  are reduced for $\dbar_{\tilde D_{I}}$ has a closed range. Hence $\tau_{\dbar, p^{-1}(D_{I})}=(0,\Mod(N(G)))$ (\ref{Dolbeault torsion theory}) is the trivial torsion theory: the $l^{2}-$Hodge to De Rham spectral sequence of $p^{-1}( D_{I})$ degenerates.  Therefore  $\tau_{\dbar,p^{-1}( D_{.})}$ (\ref{normal crossing Dolbeault torsion theory}) is also trivial. According to Lemma \ref{definition de tilde beta} $(3)$, this implies that for each $l\in\ZZ$, $R
\Gamma(Gr^{W}_{l}\K):=[(\Gamma R^l j_{*}\p\RR),\, (R\Gamma Gr^{W}_{l}\p\dlog{.}
{D},F),\, R\Gamma(Gr^{W}_{l}\tilde\beta)]$ is a Hodge complex in $N(G,\RR)$ as required in theorem  \ref{theoreme de structure mixte}.
 
In the $E_{1}-$page of the weight spectral sequence, the only non vanishing terms are $H^{n-l}_{d(2)}(p^{-1}( D_{l}) )$ which are isomorphic to the corresponding harmonic spaces. Therefore the weight spectral sequence degenerates at $E_{1}$, $Gr^{W}_{n+l}H^{n}(X\setminus D,\p\CC)\simeq H^{n-l}_{d(2)}(p^{-1}( D_{l}))$ and $H^{k}(U,\p\CC)$ is vanishing if $k\not=n$.
\end{proof}
 
\nocite{AncGav}
\nocite{Nic}

\end{document}